\documentclass[a4paper,reqno,oneside]{amsart}
\usepackage{tikz}
\usepackage{tikz-cd}
\tikzcdset{scale cd/.style={every label/.append style={scale=#1},cells={nodes={scale=#1}}}}
\usetikzlibrary{arrows,shapes,chains,matrix,positioning,scopes,cd,patterns,decorations.pathreplacing}

\usepackage{amssymb, amstext, amsmath, amscd, amsthm, amsfonts, latexsym, mathrsfs, mathtools, lscape, comment, mathrsfs, empheq}
\usepackage{xcolor}
\usepackage{bm, bbm}
\usepackage{adjustbox}
\usepackage[all]{xy}
\usepackage[includehead, includefoot, top=20.5truemm, bottom=20.5truemm, left=25truemm, right=25truemm]{geometry}
\definecolor{darkblue}{rgb}{0.0, 0.0, 0.55}
\definecolor{darkmagenta}{rgb}{0.55, 0.0, 0.55}
\definecolor{darkorange}{rgb}{1.0, 0.55, 0.0}
\usepackage[colorlinks=true]{hyperref}
\hypersetup{
	linkcolor  = darkblue,
	citecolor  = darkmagenta,
	urlcolor   = darkorange,
	colorlinks = true,
}

\theoremstyle{plain} 
\newtheorem{thm}{Theorem}[section]
\newtheorem{cor}[thm]{Corollary}
\newtheorem*{thm*}{Theorem}
\newtheorem*{cor*}{Corollary}
\newtheorem{lem}[thm]{Lemma}
\newtheorem{prop}[thm]{Proposition}
\newtheorem{dfn-prop}[thm]{Definition-Proposition}

\theoremstyle{definition}
\newtheorem{dfn}[thm]{Definition}
\newtheorem{ex}[thm]{Example}

\numberwithin{equation}{section}

\renewcommand{\AA}{\mathcal{A}}
\newcommand{\BB}{\mathcal{B}}
\newcommand{\DD}{\mathcal{D}}
\newcommand{\NN}{\mathbb{N}}
\newcommand{\ZZ}{\mathbb{Z}}
\newcommand{\QQQ}{\mathbb{Q}}
\newcommand{\RR}{\mathbb{R}}
\newcommand{\CC}{\mathcal{C}}

\newcommand{\PP}{\mathcal{P}}
\newcommand{\QQ}{\mathcal{Q}}
\renewcommand{\SS}{\mathcal{S}}
\newcommand{\TT}{\mathcal{T}}
\newcommand{\XX}{\mathcal{X}}
\newcommand{\YY}{\mathcal{Y}}

\newcommand{\WW}{\mathcal{W}}
\newcommand{\MM}{\mathcal{M}}
\newcommand{\II}{\mathbb{I}}
\newcommand{\JJ}{\mathbb{J}}
\def\fm{{\mathfrak m}}
\def\fp{{\mathfrak p}}
\def\fs{{\mathfrak s}}
\def\ft{{\mathfrak t}}
\newcommand{\m}{\mathfrak{m}}
\newcommand{\rM}{\mathrm{M}}
\newcommand{\rL}{\mathrm{L}}

\def\Aut{\operatorname{Aut}}

\def\depth{\operatorname{depth}}

\def\dim{\operatorname{dim}}

\def\e{\operatorname{e}}
\def\End{\operatorname{End}}
\def\Ext{\operatorname{Ext}}

\def\gldim{\operatorname{gldim}}

\def\H{\operatorname{H}}

\def\Hom{\operatorname{Hom}}
\def\Hom{\operatorname{Hom}}
\def\HOM{\operatorname{HOM}}

\def\Ind{\operatorname{Ind}}
\def\injdim{\operatorname{injdim}}

\def\max{\operatorname{max}}
\def\min{\operatorname{min}}

\def\op{\operatorname{o}}
\def\Projs{\operatorname{Proj}}

\def\rank{\operatorname{rank}}
\def\Spec{\operatorname{Spec}}
\def\sup{\operatorname{sup}}

\def\Tor{\operatorname{Tor}}
\def\Tr{\operatorname{Tr}}

\def\lcd{\operatorname{lcd}}
\def\lcm{\operatorname{lcm}}

\def\infd{\operatorname{inf}}
\def\soc{\operatorname{soc}}
\def\rad{\operatorname{rad}}

\def\add{\operatorname{\mathsf{add}}}
\def\mod{\operatorname{\mathsf{mod}}}
\def\Mod{\operatorname{\mathsf{Mod}}}
\def\proj{\operatorname{\mathsf{proj}}}
\def\Proj{\operatorname{\mathsf{Proj}}}
\def\lin{\operatorname{\mathsf{lin}}}

\def\CM{\operatorname{\mathsf{CM}}}

\def\qgr{\operatorname{\mathsf{qgr}}}

\def\coh{\operatorname{\mathsf{coh}}}
\def\uCM{\operatorname{\underline{\mathsf{CM}}}}
\def\umod{\operatorname{\underline{\mathsf{mod}}}}
\def\top{\operatorname{top}}
\def\D{\mathsf{D}}
\def\Db{\mathsf{D}^\mathrm{b}}
\def\K{\mathsf{K}}
\def\Kb{\mathsf{K}^\mathrm{ b}}
\def\sg{\operatorname{sg}}
\def\thick{\operatorname{\mathsf{thick}}}
\def\per{\operatorname{\mathsf{per}}}
\def\Sim{\operatorname{\mathsf{sim}}}
\def\Sq{\operatorname{\mathrm{Sq}}}

\DeclareMathOperator{\RHom}{\mathrm{R}\!\Hom}
\DeclareMathOperator{\RHOM}{\mathrm{R}\!\HOM}
\DeclareMathOperator{\Lotimes}{\stackrel{L}{\otimes}}
\DeclareMathOperator{\RGamma}{\mathrm{R}\Gamma}

\begin{document}

\title[CM representations of AS-Gorenstein algebras of dimension one]
{Cohen-Macaulay representations of\\Artin-Schelter Gorenstein algebras of dimension one}

\author[O. Iyama]{Osamu Iyama}
\author[Y. Kimura]{Yuta Kimura}
\author[K. Ueyama]{Kenta Ueyama}
\address{Osamu Iyama : Graduate School of Mathematical Sciences, The University of Tokyo, 3-8-1 Komaba Meguro-ku Tokyo 153-8914, Japan}
\email{iyama@ms.u-tokyo.ac.jp}
\address{Yuta Kimura : Department of Mechanical Systems Engineering, Faculty of Engineering, Hiroshima Institute of Technology, 2-1-1 Miyake, Saeki-ku Hiroshima 731-5143, Japan, and Osaka Central Advanced Mathematical Institute, Osaka Metropolitan University, 3-3-138 Sugimoto, Sumiyoshi-ku Osaka 558-8585, Japan}
\email{y.kimura.4r@cc.it-hiroshima.ac.jp}
\address{Kenta Ueyama : Department of Mathematics, Faculty of Science, Shinshu University, 3-1-1 Asahi, Matsumoto, Nagano 390-8621, Japan}
\email{ueyama@shinshu-u.ac.jp}

\thanks{The first author was supported by JSPS Grant-in-Aid for Scientific Research (B) JP22H01113 and (C) JP18K03209.
The second author was supported by Grant-in-Aid for JSPS Fellows JP22J01155.
The third author was supported by JSPS Grant-in-Aid for Scientific Research (C) JP22K03222.}



\begin{abstract}
Tilting theory is one of the central tools in modern representation theory, in particular in the study of Cohen-Macaulay representations.
We study Cohen-Macaulay representations of $\mathbb N$-graded Artin-Schelter Gorenstein algebras $A$ of dimension one, without imposing the connectedness condition $A_0=k$. 
This framework covers a broad class of noncommutative Gorenstein rings, including classical Gorenstein orders that are $\mathbb N$-graded.
We prove that the stable category $\operatorname{\underline{\mathsf{CM}}}_0^{\mathbb Z}A$ admits a silting object if and only if $A_0$ has finite global dimension. In this case, we give a silting object in $\operatorname{\underline{\mathsf{CM}}}_0^{\mathbb Z}A$ explicitly.
Moreover, without loss of generality, we assume that $A$ is ring-indecomposable.
Then we prove that $\operatorname{\underline{\mathsf{CM}}}_0^{\mathbb Z}A$ admits a tilting object if and only if either $A$ is Artin-Schelter regular or the average Gorenstein parameter $p^A_\mathrm{av} \in {\mathbb Q}$ of $A$ is non-positive. These results are far-reaching generalizations of the results of Buchweitz, Iyama, and Yamaura.
We give two different proofs of the second result; one is based on Orlov-type semiorthogonal decompositions, and the other is based on a more direct calculation.
We apply our results to a Gorenstein tiled order $A$ to prove that $\operatorname{\underline{\mathsf{CM}}}^{\mathbb Z}A$ is equivalent to the derived category of the incidence algebra of an (explicitly constructed) poset. 

We also apply our results and Koszul duality to study smooth
noncommutative projective quadric hypersurfaces
$\operatorname{\mathsf{qgr}} B$ of arbitrary dimension. We prove that
the derived category
$\mathsf{D}^\mathrm{b}(\operatorname{\mathsf{qgr}} B)$ admits an
(explicitly constructed) tilting object. Through Orlov's
semiorthogonal decomposition, our tilting object has the tilting
object of $\operatorname{\underline{\mathsf{CM}}}^{\mathbb Z}B$ due to
Smith and Van den Bergh as a direct summand.
\end{abstract}

\maketitle
\setcounter{tocdepth}{1}
\tableofcontents

\section{Introduction}

The study of maximal Cohen-Macaulay modules is one of the most active subjects in representation theory and commutative algebra \cite{Au2,CR,LW,Rogg,Si,Yo}, and has increasing importance in algebraic geometry and physics.
When the ring $A$ is Gorenstein, the category
\[\CM A=\{M\in\mod A\mid\Ext^i_A(M,A)=0\ \text{ for all}\ i\geq 1\}\]
of \emph{maximal Cohen-Macaulay} (\emph{CM} for short) $A$-modules forms a Frobenius category, which enhances the singularity category $\D_{\sg}(A)$ \cite{Bu,Or}, that is, the stable category $\uCM A$ is triangle equivalent to $\D_{\sg}(A)$.
When $A$ is a hypersurface, it is also triangle equivalent to the stable category of matrix factorizations \cite{Ei}.

Tilting theory controls triangle equivalences between derived categories of rings, and plays a significant role in various areas of mathematics (see e.g.\ \cite{AHK}).
Tilting theory also gives a powerful tool to study Cohen-Macaulay representations over (commutative and noncommutative) Gorenstein rings. A central notion in tilting theory is a \emph{tilting object} (respectively, \emph{silting object}), which is an object $T$ in a triangulated category $\mathscr T$ with suspension functor $[1]$ satisfying the following conditions.
\begin{itemize}
\item $\Hom_{\mathscr T}(T,T[i])=0$ for any $i\in\ZZ\setminus\{0\}$ (respectively, $i\in\ZZ_{\ge1}$).
\item The minimal thick subcategory of $\mathscr T$ containing $T$ is $\mathscr T$.
\end{itemize}
The class of silting objects complements that of tilting objects from a point of view of mutation. 
Recent advances in representation theory show that the existence of silting objects is a fundamental property of a triangulated category. Indeed, it induces both a t-structure and a co-t-structure
\cite{KY,BY}, and provides a framework in which mutation theory can be used to construct and study families of silting objects \cite{AI,AdIR}.

There exist a number of Gorenstein rings $A$ graded by abelian groups $G$ whose stable categories of $G$-graded CM $A$-modules admit tilting objects. In this case there exists a triangle equivalence
\begin{equation}\label{simple dynkin}
\uCM^{G}A\simeq\per\Lambda,
\end{equation}
where $\per\Lambda$ is the thick subcategory of the derived category $\Db(\mod\Lambda)$ generated by $\Lambda$ (see e.g.\ \cite{AIR,BIY,DL,FU,Ge,HI,HIMO,HU,IO,IT,JKS,KST1,KST2,Kim,KMY,KLM,LP,LZ,MU1,MY,SV,Ued,Ya} and a survey article \cite{Iy}).

\medskip
The aim of this paper is to study Cohen-Macaulay representations of a large class of noncommutative $\NN$-graded Gorenstein rings, called \emph{Artin-Schelter Gorenstein algebras} $A$ (Definition \ref{define AS-Gorenstein}) over a field $k$.
They are a Gorenstein analog of \emph{Artin-Schelter regular algebras}, which are main objects in noncommutative algebraic geometry (see e.g.\ \cite{AS,ATV1, ATV2, AZ, JZ, LPWZ, MM, Mo2, MU2, RR, VdB2, YZ} and survey articles \cite{Ro1, Ro2}).

In this paper, we allow Artin-Schelter Gorenstein algebras with arbitrary $A_0$, thereby dropping the commonly imposed connectedness assumption $A_0 = k$. This generality is not merely technical: once the connectedness restriction is removed, Artin-Schelter regular (respectively, Gorenstein) algebras can be naturally viewed as graded as well as twisted versions of Calabi-Yau (respectively, singular Calabi-Yau) algebras \cite{Gi,Ke,RR}. Moreover, this broader framework brings into play a rich and geometrically as well as representation-theoretically meaningful class of examples, namely Gorenstein orders (see Proposition \ref{lem-Gorder-ASG}) \cite{Au2,CR,IR,IW,Rogg}. The following table shows a hierarchy of commutative and noncommutative Gorenstein/regular rings we discussed. 
\[
\adjustbox{center}{\small
\scalebox{0.95}[1]{
$\begin{xy}
(0,55)*+[F:<5pt>]{\begin{array}{c}\mbox{Iwanaga-Gorenstein}\\ \mbox{rings}\end{array}}="0",
(18,39)*+[F:<5pt>]{\begin{array}{c}\mbox{Artin-Schelter}\\ \mbox{Gorenstein algebras}\end{array}}="1",
(88,39)*+[F:<5pt>]{\begin{array}{c}\mbox{Gorenstein}\\ \mbox{orders}\end{array}}="11",
(36,23)*+[F:<5pt>]{\begin{array}{c}\mbox{singular}\\ \mbox{Calabi-Yau algebras}\end{array}}="2",
(106,23)*+[F:<5pt>]{\begin{array}{c}\mbox{symmetric}\\ \mbox{orders}\end{array}}="12",
(124,7)*+[F:<5pt>]{\begin{array}{c}\mbox{Gorenstein}\\ \mbox{rings}\end{array}}="22",
(0,16)*+[F:<5pt>]{\begin{array}{c}\mbox{Artin-Schelter}\\ \mbox{regular algebras}\end{array}}="3",
(70,16)*+[F:<5pt>]{\begin{array}{c}\mbox{regular}\\ \mbox{orders}\end{array}}="13",
(18,0)*+[F:<5pt>]{\begin{array}{c}\mbox{Calabi-Yau}\\ \mbox{algebras}\end{array}}="4",
(88,0)*+[F:<5pt>]{\begin{array}{c}\mbox{symmetric}\\ \mbox{regular orders}\end{array}}="14",
(106,-16)*+[F:<5pt>]{\begin{array}{c}\mbox{regular}\\ \mbox{rings}\end{array}}="24",
\ar"1";"0",
\ar"4";"2",
\ar@{.>}"4";"3",
\ar@{.>}"2";"1",
\ar"3";"1",
\ar"14";"12",
\ar"14";"13",
\ar"12";"11",
\ar"13";"11",
\ar"24";"22",
\ar@{<.}"1";"11",
\ar@{<--}"2";"12",
\ar@{<.}"3";"13",
\ar@{<--}"4";"14",
\ar@{<-}"12";"22",
\ar@{<-}"14";"24",
\end{xy}
$}}
\]
Here, solid arrows indicate inclusions in general, dotted arrows indicate inclusions when $A$ is a graded $k$-algebra, and dashed arrows indicate inclusions when the base ring of $A$ is a finitely generated Gorenstein $k$-algebra.

It is well-known that the homological behavior of an $\NN$-graded commutative Gorenstein ring is strongly influenced by the sign of a numerical invariant called the \emph{Gorenstein parameter} (a.k.a.\ the \emph{$a$-invariant} up to sign). We introduce a noncommutative and non-connected version of Gorenstein parameters; for a basic $\NN$-graded Artin-Schelter Gorenstein algebra $A$ of dimension $d$, let $1=\sum_{i\in \II_A}e_i$ be a complete set of primitive orthogonal idempotents of $A$, and let $S_i$ ($i\in\II_A$) be the corresponding simple right $A$-module concentrated in degree $0$. Then the \emph{Gorenstein parameter} of $A$ is a tuple $p_A=(p_i)_{i\in\II_A}$ defined by the property
\[\Ext^d_A(S_i,A)\simeq D(S_{\nu i})(p_i),\]
where $\nu:\II_A\rightarrow \II_A; i \mapsto \nu i$ is a bijection called the \emph{Nakayama permutation}, and $D=\Hom_k(-,k)$ (see Section \ref{section: AS Gorenstein}). It plays an essential role in our results of this paper as we shall see below.

As already said, for an $\NN$-graded Artin-Schelter Gorenstein algebra $A$ of dimension $d$, we study the category of \emph{$\ZZ$-graded CM $A$-modules}
\[\CM^{\ZZ}A=\{M\in\mod^{\ZZ}A\mid \Ext^i_A(M,A)=0\ \text{for all}\ i>0\}.\]
As in the ungraded case, the category $\CM^{\ZZ}A$ forms a Frobenius category and enhances the graded singularity category $\D_{\sg}^{\ZZ}(A)$.

\medskip

In the case $d=0$, $\NN$-graded Artin-Schelter Gorenstein algebras of
dimension $0$ are precisely $\NN$-graded finite dimensional
selfinjective algebras. In this case, it is known that the stable
category $\uCM^{\ZZ}A$ always admits a tilting object if $\gldim A_0$
is finite (see \cite{Ya}).

In this paper, we concentrate on the next fundamental case $d=1$. We
consider the Serre quotient category
\[\qgr A:=\mod^{\ZZ}A/\mod^{\ZZ}_0A\]
which is traditionally called the \emph{noncommutative projective scheme} \cite{AZ}, and define the \emph{graded total quotient ring} $Q$ of $A$ (Definition \ref{define Q}) as the graded endomorphism algebra of $A$ in $\qgr A$. As in classical Auslander-Reiten theory for orders, the full subcategory
\begin{align*}
\CM_0^{\ZZ}A=\{M\in\CM^{\ZZ}A\mid M\otimes_A Q\ \text{is a graded projective $Q$-module}\}
\end{align*}
behaves much nicer than $\CM^{\ZZ}A$. In fact, it enjoys Auslander-Reiten-Serre duality (Theorem \ref{thm.Serref}), and hence it has almost split sequences.

Now we state the main result of this paper. We assume the following condition.
\begin{itemize}
\item[(A1)] $A$ is a ring-indecomposable basic $\NN$-graded Artin-Schelter Gorenstein algebra of dimension $1$.
\end{itemize}
We denote by $(p_i)_{i\in\II_A}$ the Gorenstein parameter of $A$, by
\[p^A_\mathrm{av}:=(\#\II_A)^{-1}\sum_{i \in \II_A}p_i\in\QQQ\]
their average,
and by $\nu$ the Nakayama permutation of $A$.
Let $Q$ be the graded total quotient ring of $A$. Then there exists a positive integer $q$ such that $\proj^{\ZZ}Q=\add\bigoplus_{i=1}^qQ(i)$ (see Theorem \ref{prop.Q4}(1)).

\begin{thm}[{Proposition \ref{non-existence silting},\ Theorem \ref{a and tilting}},\ Proposition \ref{information on V}, Corollary \ref{a and tilting 2}]\label{intro:a and tilting}
Under the assumption \textup{(A1)}, the following assertions hold true.
	\begin{enumerate}
\item $\uCM_0^{\ZZ}A$ adimits a silting object if and only if $\gldim A_0$ is finite.
\end{enumerate}
In the rest, assume that $\gldim A_0$ is finite.
\begin{enumerate}
\setcounter{enumi}{1}
		\item $\uCM_0^{\ZZ}A$ always admits a silting object
		\begin{equation}\label{define V intro}
			V:=\bigoplus_{s\in\II_A}\bigoplus_{i=1}^{-p_s+q}e_{\nu s}A(i)_{\ge 0}.
		\end{equation}
\item $\uCM_0^{\ZZ}A$ admits a tilting object if and only if either $p^A_\mathrm{av}\leq0$ or $A$ is Artin-Schelter regular.
\end{enumerate}
In the rest, assume that $p_i\le0$ holds for each $i\in\II_A$.
\begin{enumerate}
\setcounter{enumi}{3}
\item The object $V$ in \eqref{define V intro} is a tilting object in $\uCM_0^{\ZZ}A$. Thus for $\Gamma:=\underline{\End}^{\ZZ}_A(V)$, we have a triangle equivalence
\[\uCM_0^{\ZZ}A\simeq\per\Gamma.\]
\item $\Gamma$ is an Iwanaga-Gorenstein algebra. Moreover, there is an explicit description of $\Gamma$; see Proposition \ref{information on V}\textup{(3)(4)}.
\item If the quiver of $A_0$ is acyclic, then there exists an ordering in the isomorphism classes of indecomposable direct summands of $V$, which forms a full strong exceptional collection in $\uCM_0^{\ZZ}A$.
\end{enumerate}
\end{thm}

As an application of Theorem \ref{intro:a and tilting}, we immediately obtain the following result.

\begin{cor}[Corollary \ref{grothendieck}]
Assume that \textup{(A1)}, $\gldim A_0<\infty$, and $p^A_\mathrm{av}\leq0$ hold. Then the Grothendieck group $K_0(\uCM_0^{\ZZ}A)$ of $\uCM_0^{\ZZ}A$ is a free abelian group of
\[\rank K_0(\uCM_0^{\ZZ}A)=-\sum_{s\in\II_A}p_s+\#\Ind(\proj^\ZZ Q).\]
\end{cor}

We give two different proofs of Theorem \ref{intro:a and tilting}(4). The first one given in Section \ref{section: tilting theory} is based on a calculation of syzygies and Auslander-Reiten-Serre duality.
The second one given in Section \ref{section: tilting theory 2} is based on Orlov-type semiorthogonal decompositions of the categories $\D_{\sg}^{\ZZ}(A)$ and $\Db(\qgr A)$ prepared in Section \ref{section: orlov}.

In both proofs, the category $\qgr A$ and the thick subcategory $\per(\qgr A)$ of $\Db(\qgr A)$ generated by $\proj^{\ZZ}A$ play an important role. One of the main properties is the following.

\begin{thm}[Theorem \ref{prop.Q4}]\label{intro:prop.Q4}
Under the assumption \textup{(A1)}, $\qgr A$ has a progenerator
\[P:=\bigoplus_{i=1}^{q}A(i).\]
Therefore $\per(\qgr A)$ has a tilting object $P$ and we have a triangle equivalence $\per(\qgr A)\simeq\per\Lambda$ for $\Lambda:=\End_{\qgr A}(P)$.
\end{thm}

It is also important in our proofs to observe the change of Gorenstein parameters under graded Morita equivalence (see Definition-Proposition \ref{graded Morita equivalent}).
In particular, the average Gorenstein parameter is invariant under graded Morita equivalence (see Proposition \ref{prop-eq-a-inv}). We also prove the following key result by using a purely combinatorial argument in Appendix \ref{appendix}.

\begin{thm}[Theorem \ref{thm-a-inv}]
For a ring-indecomposable basic $\NN$-graded Artin-Schelter Gorenstein algebra $A$, there is a ring-indecomposable basic $\NN$-graded Artin-Schelter Gorenstein algebra $B$ satisfying the following conditions.
\begin{enumerate}
\item $B$ is graded Morita equivalent to $A$.
\item $|p_i^{B}-p^{B}_\mathrm{av}|<1$ holds for each $i\in\II_B$.
\end{enumerate}
In particular, if $p^A_\mathrm{av}\leq 0$ holds, then $p_i^{B}\leq 0$ for each $i\in\II_B$.
\end{thm}

Now we present some examples of our Theorem \ref{intro:a and tilting}.
Considering the case where our Artin-Schelter Gorenstein algebra $A$ is commutative, we immediately recover the following main result of Buchweitz-Iyama-Yamaura's paper \cite{BIY}.

\begin{ex}{\cite{BIY}}\label{BIY main}
Let $R$ be an $\NN$-graded commutative Gorenstein ring of Krull dimension $1$ with Gorenstein parameter $p$ such that $R_0$ is a field. Let $Q$ be the graded total quotient ring of $R$, and $q\ge1$ an integer such that $Q(q)\simeq Q$ in $\mod^{\ZZ}Q$. Then the following holds true.
\begin{itemize}
\item[(1)] $\uCM_0^{\ZZ}R$ has a silting object $V:=\bigoplus_{i=1}^{-p+q}R(i)_{\ge0}$.
\item[(2)] $\uCM_0^{\ZZ}R$ has a tilting object if and only if either $p\leq 0$ or $R$ is regular. In this case, $V$ above is a tilting object.
\item[(3)] $\per(\qgr R)$ has a tilting object $P:=\bigoplus_{i=1}^{q}R(i)$.
Therefore we have a triangle equivalence $\per(\qgr R)\simeq\per\Lambda$ for $\Lambda:=\End_{\qgr R}(P)$.
\end{itemize}
\end{ex}

Next we apply our Theorem \ref{intro:a and tilting} to important classes of noncommutative  Artin-Schelter Gorenstein algebras. The first one is Gorenstein tiled orders \cite{Si,ZK,KKMPZ}.

\begin{thm}[Theorem \ref{thm-GTO-Inc}]
Let $A$ be a basic $\NN$-graded Gorenstein tiled order such that $p_i\leq 0$ for any $i\in\mathbb{I}_A$, and let $(\mathbb{V}_A, \le)$ be the poset introduced in \eqref{define V_A}.
Then the following statements hold.
\begin{enumerate}
\item $V:=\bigoplus_{i\in\II_A}\bigoplus_{j=1}^{1-p_i}e_iA(j)_{\ge0}$ is a tilting object in $\uCM^{\ZZ}A$.
\item The number of non-isomorphic indecomposable direct summands of $V$ in $\uCM^{\ZZ}A$ is $1-\sum_{i\in\II_A}p_i$.
In particular, the Grothendieck group $K_0(\uCM^{\ZZ}A)$ is a free abelian group of rank $1-\sum_{i\in\II_A}p_i$. 
\item $\End_{\uCM^{\ZZ}A}(V)$ is Morita equivalent to an incidence algebra $k\mathbb{V}_A^{\op}$. In particular, the global dimension of $\End_{\uCM^{\ZZ}A}(V)$ is finite, and we have a triangle equivalence
\[\uCM^{\ZZ}A\simeq \Db(\mod k\mathbb{V}_A^{\op}).\]
\end{enumerate}
\end{thm}

The second one is the Koszul duals 
of \emph{noncommutative quadric hypersurfaces} $B$ (see Definition \ref{dfn.qh}).
By combining Theorem \ref{intro:a and tilting} and Koszul duality, we give the following existence theorem of a tilting object for the derived category $\Db(\qgr B)$.

\begin{thm}[Proposition \ref{prop.KdGor}, Theorem \ref{thm.nqh}]\label{thm.nqh in intro}
Let $B$ be a noncommutative quadric hypersurface of dimension $d-1$ with $d\geq 2$, and let $A$ be the opposite ring of the Koszul dual of $B$.
Assume that $\qgr B$ has finite global dimension.
Then the following statements hold.
\begin{enumerate}
\item $A$ is a Koszul Artin-Schelter Gorenstein algebra of dimension $1$ and Gorenstein parameter $2-d$.
\item There exists a duality $F: \Db(\qgr B) \to \uCM^{\ZZ}A$.
\item $\uCM^{\ZZ} A$ has a tilting object $\bigoplus_{i=1}^{d-1} A(i)_{\geq 0}$, and $\Db(\qgr B)$ has a tilting object $\bigoplus_{i=1}^{d-1}\Omega^ik(i)$. Moreover, they correspond to each other via the duality $F$.
\item Let $\Lambda := \End_{\Db(\qgr B)}(\bigoplus_{i=1}^{d-1}\Omega^ik(i))$, and $Q$ the graded total quotient ring of $A$.
Then we have isomorphisms of $k$-algebras
\[
\Lambda \simeq \End_{\uCM^{\ZZ}A}(\bigoplus_{i=1}^{d-1} A(i)_{\geq 0})^{\op}\simeq
		\begin{bmatrix}
			k &0&\cdots  &\cdots  &0\\
			A_1&k &\ddots  &&\vdots \\
			\vdots&\vdots&\ddots  &\ddots&\vdots \\
			A_{d-3}&A_{d-4}&\cdots&k &0\\
			Q_{d-2}&Q_{d-3}&\cdots&Q_1&Q_0
		\end{bmatrix}^{\op},
\]
and we have triangle equivalences
\[\Db(\qgr B) \simeq(\uCM^{\ZZ} A)^{\op} \simeq \Db(\mod \Lambda).\]
\end{enumerate}
\end{thm}

Theorem \ref{thm.nqh in intro} gives a noncommutative generalization of the result that the derived category of a smooth projective quadric hypersurface admits a tilting object \cite{Ka1,Ka2}. 

We end this section by explaining a connection between our results and a work by Smith-Van den Bergh and Mori-Ueyama \cite{BEH, SV, MU3}. By combining Theorem \ref{intro:prop.Q4} and Koszul duality, one can recover the following results immediately.

\begin{cor}\label{tilting in CMB}\cite{SV, MU3}
Under the settings in Theorem \ref{thm.nqh in intro}, the following assertions hold.
\begin{enumerate}
\item There exists a duality $F': \Db(\qgr A)\to\uCM^{\ZZ}B$.
\item $\Db(\qgr A)$ has a tilting object $A(d-1)$, and $\uCM^{\ZZ}B$ has a tilting object $\Omega^{d-1}k(d-1)$. Moreover, they correspond to each other via the duality $F'$.
\item Let $\Lambda' := \End_{\uCM^{\ZZ}B}(\Omega^{d-1}k(d-1))$. Then we have isomorphisms of $k$-algebras 
\[
\Lambda' \simeq \End_{\Db(\qgr A)}(A(d-1))^{\op}\simeq (Q_0)^{\op}
\]
and we have triangle equivalences
\[\uCM^{\ZZ}B \simeq\Db(\qgr A)^{\op} \simeq \Db(\mod \Lambda').\]
\end{enumerate}
\end{cor}

Moreover, through Orlov's semiorthogonal decomposition, the tilting objects given in Corollary \ref{tilting in CMB} becomes a direct summand of our tilting objects given in Theorem \ref{thm.nqh in intro}, as illustrated in Figure \ref{connection of tilting}.

\begin{figure}[h]
\[\xymatrix@R=1pt@C=3pt@M=6.5pt{ 
	&*+[F-:<1pt>]{\text{tilt.\,obj.\,in Thm.\,\ref{thm.nqh}}}
	&&&*+[F-:<1pt>]{\text{tilt.\,obj.\,in Thm.\,\ref{a and tilting}}}
	\\
	&\bigoplus_{i=1}^{d-1}\Omega^{i}k(i) \ar@{|.>}[rrr] \ar@{}|-{\rotatebox{270}{$\in$}}[d]
	&&& \bigoplus_{i=1}^{d-1}  A(i)_{\geq 0}\ar@{}|-{\rotatebox{270}{$\in$}}[d]
	\\
	&\Db(\qgr B) \ar@{<-_)}[ddd]^{\text{Orlov SOD}} \ar[rrr]_{F}^{\text{Koszul duality}}
	&&& \uCM^{\ZZ} A\ar@{<-_)}[ddd]^{\text{Orlov SOD}} 
	&
	\\\\\\
	&\uCM^{\ZZ} B
	&&& \Db(\qgr A) \ar[lll]^{F'}_{\text{Koszul duality}}
	\\
	&\Omega^{d-1}k(d-1) \ar@{}|-{\rotatebox{90}{$\in$}}[u]  \ar@/^4pc/@{|.>}[uuuuu]^-{\oplus}
	&&&\txt{\normalsize$A(d-1)=A(d-1)_{\ge0}$}\ar@{}|-{\rotatebox{90}{$\in$}}[u]  \ar@{|.>}[lll] \ar@/_4pc/@{|.>}[uuuuu]_-{\oplus}
	\\
	&*+[F-:<3pt>]{\text{tilt.\,obj.\,in \cite{SV, MU3}}}
	&&&*+[F-:<3pt>]{\text{tilt.\,obj.\,in Thm.\,\ref{prop.Q4}(4)}}
}\]
\caption{The relationship between several tilting objects}
\label{connection of tilting}
\end{figure}

\subsection{Conventions}
Throughout this paper, $k$ is a field, and all algebras are over $k$.
(In Section \ref{sec.nqh}, we will
additionally assume that $k$ is algebraically closed and of characteristic zero.)
The composition of morphisms $f:L\to M$ and $g:M\to N$ in a category is denoted by $g\circ f$. The composition of arrows $a:i\to j$ and $b:j\to h$ is denoted by $b\circ a$.
Thus each arrow $a:i\to j$ of a quiver $Q$ gives a morphism $(a\cdot):e_i(kQ)\to e_j(kQ)$ of right projective $kQ$-modules over the path algebra $kQ$.

For a ring $A$, we denote by $\Mod A$ (respectively, $\mod A$, $\proj A$) the category of (respectively, finitely generated, finitely generated projective) right $A$-modules.

Let $A$ be a $\ZZ$-graded algebra.
We denote by $\Mod^{\ZZ} A$ (respectively, $\mod^{\ZZ} A$, $\proj^{\ZZ} A$) the category of graded (respectively, finitely generated graded, finitely generated graded projective) 
right $A$-modules.
For a subset $I$ of $\ZZ$, let
\[\mod^IA:=\{M=\bigoplus_{i\in \ZZ}M_i \in\mod^{\ZZ}A\mid\forall i\in\ZZ\setminus I,\ M_i=0\}.\]
In particular, for $i\in\ZZ$, let $\ZZ_{\ge i}:=\{j\in\ZZ\mid j\ge i\}$, $\ZZ_{<i}:=\{j\in\ZZ\mid j<i\}$ and
\begin{equation}\label{define mod>0A}
\mod^{\ge i}A:=\mod^{\ZZ_{\ge i}}A\quad \mbox{ and }\quad \mod^{<i}A:=\mod^{\ZZ_{<i}}A.
\end{equation}
If $A$ is $\NN$-graded, then the canonical surjection $A\to A_0$ gives an equivalence $\mod A_0\simeq\mod^{\{0\}}A$. We regard $\mod A_0$ as a full subcategory of $\mod^{\ZZ}A$.

We denote by $A^{\op}$ the opposite algebra of $A$ and by $A^{\e} = A^{\op} \otimes_k A$ the enveloping algebra.
The category of graded left $A$-modules is identified with $\Mod^{\ZZ} A^{\op}$ and the category of graded $A$-bimodules on which $k$-action acts centrally is identified with $\Mod^{\ZZ} A^{\e}$.
We call $M \in \Mod^{\ZZ} A^{\e}$ \emph{graded invertible} if there exists $L \in \Mod^{\ZZ} A^{\e}$ such that $M \otimes_AL\simeq A\simeq  L \otimes_A M$ in $\Mod^{\ZZ}A^{\e}$.
For $M=\bigoplus_{i\in \ZZ}M_i \in \Mod^{\ZZ}A$ and $n \in \ZZ$, we define the \emph{truncation} $M_{\geq n} := \bigoplus_{i\geq n} M_i$ and the \emph{shift} $M(n) \in \Mod^{\ZZ} A$, which has the same underlying module structure as $M$, but which satisfies $M(n)_i = M_{n+i}$.
For $M, N\in \Mod^{\ZZ} A$, we write
\[\Hom_A^{\ZZ}(M, N)=\Hom_{\Mod^{\ZZ}A}(M, N).\]
It is basic that if $M \in \mod^{\ZZ}A$, then $\Hom_A(M,N)=\bigoplus_{i\in\ZZ}\Hom_A^{\ZZ}(M, N(i))$.
Thus $\End_A(M)$ is a $\ZZ$-graded algebra with $\End_A(M)_i=\Hom_A^{\ZZ}(M,M(i))$ for each $i\in\ZZ$.

Let $\mathscr C$ be an additive category.
For a subcategory (or a collection of objects) $\mathscr B$ of $\mathscr C$, we denote by $\add \mathscr B$ the full subcategory of $\mathscr C$ consisting of direct summands of finite direct sums of objects in $\mathscr B$. The bounded homotopy category is denoted by $\Kb(\mathscr C)$.
A full subcategory of a triangulated category $\mathscr T$ is called \emph{thick} if it is closed under cones, $[\pm 1]$, and direct summands. 
For a subcategory (or a collection of objects) $\mathscr B$ of $\mathscr T$,
we denote by $\thick \mathscr B$ the smallest thick subcategory of $\mathscr T$ which contains $\mathscr B$.
For an abelian category $\mathscr A$,
the bounded (respectively, left bounded, right bounded, unbounded) derived category is denoted by
$\Db(\mathscr A)$
(respectively, $\D^{+}(\mathscr A)$, $\D^{-}(\mathscr A)$, $\D(\mathscr A)$). For a ring $A$, let $\per A:=\thick A$.

\section{Preliminaries}

This section collects a number of preparatory results for our treatment of Artin-Schelter Gorenstein algebras.

\subsection{Preliminaries on graded algebras}
A $\ZZ$-graded algebra $A=\bigoplus_{i \in \ZZ}A_i$ is called \emph{locally finite} if $\dim_k A_i<\infty$ for each $i\in\ZZ$.
Note that a right Noetherian $\ZZ$-graded algebra $A$ with $\dim_k A_0< \infty $ is locally finite.
A $\ZZ$-graded algebra $A$ is called \emph{$\NN$-graded} if $A_i=0$ for all $i<0$, and \emph{connected $\NN$-graded} if $A$ is  $\NN$-graded and $A_0=k$.
A $\ZZ$-graded algebra $A$ is called \emph{bounded below} if there exists $N\in\ZZ$ such that $A_i=0$ holds for each $i<N$. 

\begin{dfn-prop}\cite[Theorem 5.4]{GG}\label{graded Morita equivalent}
Let $A$ and $B$ be $\ZZ$-graded algebras. We say that $A$ and $B$ are \emph{graded Morita equivalent} if the following equivalent conditions are satisfied.
\begin{enumerate}
\item There exists an equivalence
$F:\Mod^{\ZZ}A\simeq\Mod^{\ZZ}B$ satisfying $F\circ(1)\simeq(1)\circ F$.
\item There exists a progenerator $P$ of $\Mod^{\ZZ}A$ such that $B \simeq \End_A(P)$ as graded algebras.
\end{enumerate}
\end{dfn-prop}

Let $A$ be a $\ZZ$-graded algebra and $1=e_1+\cdots+e_n$ orthogonal idempotents of $A_0$.
For given integers $\ell_1, \dots, \ell_n$, let $P=\bigoplus_{i=1}^ne_iA(\ell_i)\in\mod^{\ZZ}A$. We define a new $\ZZ$-graded algebra by $B=\End_A(P)$. Note that $A=B$ holds as an ungraded algebra, but they have different $\ZZ$-gradings. Then $P$ is a graded $(B, A)$-bimodule and a progenerator of $\Mod^{\ZZ}A$. By Definition-Proposition \ref{graded Morita equivalent}, we obtain the following observation.

\begin{prop}\label{graded Morita equivalent 2}
Under the above setting, $A$ is graded Morita equivalent to $B$.
\end{prop}

Clearly locally finite (respectively, bounded below) $\ZZ$-graded algebras are closed under graded Morita equivalence.

\begin{dfn}\label{define I_A}
For a locally finite $\ZZ$-graded algebra $A$, let $1=\sum_{i\in\JJ_A}e_i$ be a complete set of primitive orthogonal idempotents of $A$.
Define an equivalence relation on $\JJ_A$ by
\begin{align*}
i\sim j\ \Leftrightarrow\ e_iA\simeq e_jA\mbox{ as (ungraded) $A$-modules }\ \Leftrightarrow\ Ae_i\simeq Ae_j\mbox{ as (ungraded) $A^{\op}$-modules.}
\end{align*}
We fix a complete set $\II_A$ of representatives of $\JJ_A/\sim$ of $\JJ_A$. Let
\[\Sim A:=\{S_i:=\top^{\ZZ}e_iA\mid i\in\II_A\}\ \mbox{ and }\ \Sim A^{\op}:=\{S^{\op}_i:=\top^{\ZZ}Ae_i\mid i\in\II_A\},\]
where $\top^{\ZZ}$ denotes the top as an object in the abelian category $\mod^{\ZZ}A$ or $\mod^{\ZZ}A^{\op}$.
(Note that $S_i$ does not mean the degree $i$ part of $S$.)
Then $\Sim A$ and $\Sim A^{\op}$ give the sets of isomorphism classes of simple objects in $\mod^{\ZZ}A$ and $\mod^{\ZZ}A^{\op}$ respectively up to degree shift. 

We call $A$ \emph{basic} if $\II_A=\JJ_A$ holds, that is, for each $i\neq j\in\JJ_A$, $e_iA\not\simeq e_jA$ as (ungraded) $A$-modules.
\end{dfn}

For example, let $A$ be a locally finite $\NN$-graded algebra. Then $A$ is basic if and only if $A_0/\rad A_0$ is isomorphic to a product of division algebras. Moreover $\Sim A$ and $\Sim A^{\op}$ give the sets of the isomorphism classes of simple objects in $\mod^{\ZZ} A$ and $\mod A^{\op}$ respectively concentrated in degree zero.

\begin{lem} \label{lem.basic}
Let $A$ be a locally finite $\ZZ$-graded algebra.
\begin{enumerate}
\item $\mod^{\ZZ}A$ is Hom-finite and Krull-Schmidt.
\item $A$ is graded Morita equivalent to a locally finite $\ZZ$-graded algebra $B$ which is basic. If moreover $A$ is $\NN$-graded, then we can choose $B$ to be $\NN$-graded.
\end{enumerate}
\end{lem}

\begin{proof}
(1) is clear. To prove (2), let $e:=\sum_{i\in\II_A}e_i\in A_0$ and $P:=eA\in\Mod^{\ZZ}A$ and $B:=\End_A(P)=eAe$. Then $B$ is basic by our choice of $\II_A$. Moreover, it is a standard argument in Morita theory to show that the functor $\Hom_A(P,-):\Mod^{\ZZ}A\to\Mod^{\ZZ}B$ and $-\otimes_BP:\Mod^{\ZZ}B\to\Mod^{\ZZ}A$ give mutually quasi-inverse equivalences. The latter assertion is clear.
\end{proof}

We will use the following elementary observation.

\begin{lem}\label{connected}
Let $A$ be a ring-indecomposable locally finite $\NN$-graded algebra.
Then for any $S, S'\in\Sim A$ with $S \neq S'$, there exists a sequence of simple modules $S=S^0, S^1, \dots, S^{\ell}=S'$ in $\Sim A$ such that either $\Ext_A^1(S^i, S^{i+1})\neq 0$ or $\Ext_A^1(S^{i+1}, S^i)\neq 0$ holds for each $0\leq i\leq \ell-1$.
\end{lem}

\begin{proof}
Although this is elementary (e.g.\ \cite[Lemma II.2.5]{ASS}), we include a proof for the convenience of the reader.
Without loss of generality, we can assume that $A$ is basic.
Let $J:=\rad A_0+A_{\geq 1}$ where $\rad A_0$ is the Jacobson radical of $A_0$, and let $\overline{A}=A/J$.
It suffices to show that, if $e,f\in A$ are idempotents such that $1=e+f$ and $\Ext_A^1(e\overline{A}, f\overline{A})=0=\Ext_A^1(f\overline{A}, e\overline{A})$, then 
$eAf=0=fAe$, that is, $A=eAe\times fAf$ as rings.
By symmetry, we only show $eAf=0$.

We claim $eAf=eJ^nf$ for each $n\ge1$.
Applying $\Hom_A(-, f\overline{A})$ to the exact sequence $0 \to eJ\to eA \to e\overline{A} \to 0$, we have an exact sequence
\[0=\Hom_A(eA,f\overline{A})\to \Hom_A(eJ,f\overline{A})\to\Ext^1_A(e\overline{A},f\overline{A})=0.\]
Thus $\Hom_A(eJ,f\overline{A})=0$ and hence $e(J/J^2)f=0$. Since $A$ is basic, $eJ^2f=eJf=eAf$ holds.
Inductively, we obtain the claim. In fact,
$eAf=eJ^2f=eJeJf+eJfJf=eJeJ^nf+eJ^nfJf\subset eJ^{n+1}f$.

Since $A_0$ is finite dimensional over $k$, there is an integer $\ell\ge1$ such that $J^\ell \subset A_{\geq 1}$. Thus $eAf=eJ^{\ell n}f\subset eA_{\geq n}f$ holds for each $n$. Thus $eAf=0$ as desired.
\end{proof}

\subsection{Artin-Schelter Gorenstein algebras}\label{section: AS Gorenstein}
In this subsection, we give the definition of $\NN$-graded Artin-Schelter Gorenstein (AS-Gorenstein) algebras, which is the main subject of this paper.

We call a ring $A$ \emph{Iwanaga-Gorenstein} if it is Noetherian and satisfies $\injdim _AA<\infty$ and $\injdim_{A^{\op}}A<\infty$. In this case, $\injdim _AA=\injdim_{A^{\op}}A$ holds \cite{Za, EJ}. 
The following is well-known (e.g.\ \cite[Corollary 2.11]{Mi}).

\begin{prop}\label{A dual}
If $A$ is a $\ZZ$-graded Iwanaga-Gorenstein algebra, then 
\[\RHom_A(-,A) : \Db(\mod^{\ZZ} A) \rightleftarrows \Db(\mod^{\ZZ} A^{\op}) :\RHom_{A^{\op}}(-,A)\]
define a duality by Hom evaluation. 
\end{prop}

AS-Gorenstein algebras are defined by the behavior of simple modules under the duality above. Notice that we do \emph{not} assume that $A$ is connected graded (that is, we do \emph{not} assume $A_0=k$).

\begin{dfn}\label{define AS-Gorenstein}
Let $A$ be a Noetherian locally finite bounded below $\ZZ$-graded algebra with $\II_A$ given in Definition \ref{define I_A}.
We say that $A$ is \emph{Artin-Schelter Gorenstein (AS-Gorenstein) of dimension $d$} 
if it satisfies the following conditions.
\begin{enumerate}
\item $\injdim_A A = \injdim_{A^{\op}} A= d <\infty$.
\item For each $i\in\II_A$, there exists $\nu i:=\nu(i)\in\II_A$ and an integer $p_i$ such that
\[\Ext^\ell_A(S_i,A) \simeq
\begin{cases}
S^{\op}_{\nu i}(p_i) & \text { if } \ell=d,\\
0 & \text { if } \ell\neq d.
\end{cases}\]
\end{enumerate}
We call $p_i$ the \emph{Gorenstein parameter} of $S_i$ and the tuple $p_A=(p_i)_{i\in\II_A}$ the \emph{Gorenstein parameter} of $A$. In this case, $\nu:\II_A\to\II_A$ is a bijection called the \emph{Nakayama permutation}.

We say that $A$ is \emph{Artin-Schelter regular (AS-regular)} if it is AS-Gorenstein, and has finite global dimension. In this case, $\gldim A=d$ holds.
\end{dfn}

Let us give isomorphisms for AS-Gorenstein algebras which are frequently used in this paper. For each $i \in \II_A$, we have 
\begin{equation}\label{Ext-S 1}
\Ext^d_A(S_i,A)\simeq S^{\op}_{\nu i}(p_i)\ \mbox{ or equivalently,}\ D\Ext^d_A(S_i,A)\simeq S_{\nu i}(-p_i),
\end{equation}
where $D=\Hom_k(-,k)$. Applying the duality $\RHom_{A^{\op}}(-,A)$, we have
\begin{equation}\label{Ext-S 2}
\Ext^d_{A^{\op}}(S^{\op}_{\nu i },A)\simeq S_i(p_i)\ \mbox{ or equivalently,}\ D\Ext^d_{A^{\op}}(S^{\op}_{\nu i },A)\simeq S^{\op}_i(-p_i).
\end{equation}
For $i,j\in \II_A$ and $\ell\in\ZZ$, multiplying $e_j$ to \eqref{Ext-S 1} from the left, we obtain
\begin{equation}\label{prop-Ext-S}
\Ext_A^d(S_i, e_jA(\ell))\neq 0\Longleftrightarrow (j,\ell)=(\nu i ,-p_i).
\end{equation}
It follows from \eqref{Ext-S 2} that $A$ is AS-Gorenstein if and only if so is $A^{\op}$.

The following observation is immediate from definition.

\begin{lem}\label{lem.basic 2}
AS-Gorenstein algebras of dimension $d$ are closed under graded Morita equivalence.
Therefore each AS-Gorenstein algebra $A$ of dimension $d$ is graded Morita equivalent to a basic AS-Gorenstein algebra $B$ of dimension $d$. If moreover $A$ is $\NN$-graded, then we can choose $B$ to be $\NN$-graded.
\end{lem}

\begin{proof}
Assume that $A$ and $B$ are graded Morita equivalent. By Definition-Proposition \ref{graded Morita equivalent}, there exists $P\in\Mod^{\ZZ}A$ such that $\End_A(P)\simeq B$ as graded rings and $\Hom_A(P,-):\Mod^{\ZZ}A\to\Mod^{\ZZ}B$ and $-\otimes_BP:\Mod^{\ZZ}B\to\Mod^{\ZZ}A$ give mutually quasi-inverse equivalences. Then for each $i\in\ZZ$, we have a commutative diagram
\[\xymatrix@R=2pc@C=5pc{
\Mod^{\ZZ}A\ar[r]^{\Hom_A(P,-)}_{\sim}\ar[d]_{\Ext^i_A(-,A)}
&\Mod^{\ZZ}B\ar[d]^{\Ext^i_B(-,B)}\\
\Mod^{\ZZ}A^{\op}\ar[r]^{P\otimes_A-}_{\sim}&\Mod^{\ZZ}B^{\op}
}\]
whose horizontal functors are equivalences. This implies that if $A$ is AS-Gorenstein, so is $B$. The remaining assertions follow from Lemma \ref{lem.basic}.
\end{proof}

Thanks to Lemma \ref{lem.basic 2}, to study AS-Gorenstein algebras $A$, we can assume that $A$ is basic.
Throughout this paper, unless otherwise stated, ``a basic AS-Gorenstein algebra'' means  ``a basic AS-Gorenstein algebra of dimension $d$ and Gorenstein parameter $p_A=(p_i)_{i\in\II_A}$ with Nakayama permutation $\nu$''.

\subsection{Gorenstein orders}\label{subsection: order}
In this section, we give a typical example of AS-Gorenstein algebras, called Gorenstein orders.
We start with recalling basic notions in commutative algebra \cite{Mat,BH}.

Let $R$ be a commutative Noetherian ring of Krull dimension $d$.
When $(R,\fm)$ is local, we call $M\in\mod R$ \emph{maximal Cohen-Macaulay} (\emph{CM} for short) if either $\depth M=d$ or $M=0$ hold, where  $\depth M$ is the maximal length of $M$-regular sequences in $\fm$. In general, we call $M\in\mod R$ \emph{maximal Cohen-Macaulay} (\emph{CM} for short) if $M_\fp$ is a CM $R_\fp$-module for each $\fp\in\Spec R$. We call $R$ a \emph{Cohen-Macaulay ring} if it is CM as an $R$-module. 

Let $R=\bigoplus_{i\in\NN}R_i$ be a commutative Cohen-Macaulay $\NN$-graded ring of Krull dimension $d$ with a graded canonical module $\omega_R$. Assume that $R_0$ is an Artinian local ring containing its residue field $k$. Thus $R$ has a unique graded maximal ideal $\m=\rad R_0\oplus(\bigoplus_{i\ge1}R_i)$, and we have $\Ext^d_R(k,\omega_R)\simeq k$ as graded $R$-modules. 
Let $A=\bigoplus_{i\in\NN}A_i$ be an $\NN$-graded $R$-algebra such that the structure morphism $R \to A$ preserves the $\NN$-grading.

We say that $A$ is an \emph{$R$-order} if it is a CM $R$-module. In this case, we call the graded  $A$-bimodule
\[\omega:=\Hom_R(A,\omega_R)\]
a \emph{canonical module} of $A$. An $R$-order $A$ is called a \emph{Gorenstein $R$-order} if $\omega\in\proj A$ holds, or equivalently, $\omega\in\proj A^{\op}$ holds.
In this case, $\omega$ is an invertible $A$-bimodule, and therefore gives an autoequivalence
\begin{align}\label{tensor omega 2}
-\otimes_A\omega:\mod^{\ZZ}A\to\mod^{\ZZ}A.
\end{align}
We call $M\in\mod^{\ZZ}A$ \emph{maximal Cohen-Macaulay} (\emph{CM} for short) if it is CM as an $R$-module. The category of graded CM $A$-modules is denoted by $\CM^{\ZZ}A$.
We consider the full subcategory 
\[\CM^{\ZZ}_0A:=\{M\in\CM^{\ZZ}A\mid
M_{\fp,\ZZ}\in\proj A_{\fp,\ZZ}\; \text{for all}\;  \fp\in\Spec^{\ZZ}R\setminus\{\fm\} \} \]
where $\Spec^{\ZZ}R$ is the set of homogeneous prime ideals of $R$, and $(-)_{\fp,\ZZ}$ is a localization with respect to the multiplicative set consisting of all homogeneous elements in $R\setminus\fp$.
\begin{prop}\label{lem-Gorder-ASG}
Let $A$ be a Gorenstein $R$-order as above.
\begin{enumerate}
\item $A$ is an AS-Gorenstein algebra of dimension $d$.
\item  For each $i\in\mathbb{I}_A$, we have isomorphisms
\[\Ext^d_R(S_i^{\op},\omega_R)\simeq S_i\ \mbox{ and }\ e_i\omega\simeq\Hom_R(Ae_i,\omega_R)\simeq e_{\nu i}A(-p_i)\]
in $\mod^{\ZZ}A$, where $\nu:\mathbb{I}_A\to\mathbb{I}_A$ is a permutation given in \eqref{Ext-S 1}.
\item The category $\uCM_0^{\ZZ}A$ has a Serre functor $-\otimes_A\omega[d-1]$.
\end{enumerate}
\end{prop}

\begin{proof}
(1) Since $A$ is finitely generated as an $R$-module, $A$ is locally finite.
We have $\injdim_A A =\injdim_{A^{\op}} A=d$ by \cite[Proposition 1.1(3)]{GN}.

By \cite[Proposition 3.5 and Theorem 3.7]{IR}, we have a bifunctorial isomorphism
\begin{align}\label{serre}
\Hom_{\D(\Mod A)}(M, N \otimes_A^\mathrm{L} \omega [d]) \simeq D\Hom_{\D(\Mod A)}(N, M)
\end{align}
for any $M \in \Db(\mod_0 A)$ and $N \in \Kb(\proj A)$, where $D=\Hom_k(-,k)$.
Moreover, if $M\in\Db(\mod_0^{\ZZ} A)$ and $N \in \Kb(\proj^{\ZZ} A)$, then both sides of \eqref{serre} have canonical $\ZZ$-gradings which are preserved by the isomorphism.
For each $i\in\II_A$, the isomorphism \eqref{serre} implies $D\Ext_A^d(S_i, A) \simeq S_i\otimes_A\omega$, which is a simple object in $\mod^{\ZZ}A$ by \eqref{tensor omega 2}. Thus there exist $\nu i\in\II_A$ and $p_i\in\ZZ$ such that $\Ext_A^d(S_i, A) \simeq S_{\nu i}(p_i)$,
and hence $A$ is an AS-Gorenstein algebra.

(2) Since $\Ext^d_R(-,\omega_R):\mod_0^{\ZZ}A^{\op}\to\mod_0^{\ZZ}A$ is a duality, $\Ext^d_R(S_i^{\op},\omega_R)$ is a simple object in $\mod^{\ZZ}A$.
For each $j\in\mathbb{I}_A$ with $i\neq j$, we have
\[\Hom_A(e_jA,\Ext^d_R(S_i^{\op},\omega_R))\simeq\Ext^d_R(e_jA\otimes_AS_i^{\op},\omega_R)=0.\]
Thus $\Ext^d_R(S_i^{\op},\omega_R) \simeq S_i(\ell)$ holds for some $\ell\in\ZZ$. Take a surjection $S_i^{\op}\to k$ in $\mod^{\ZZ}R$. Applying $\Ext^d_R(-,\omega_R)$, we obtain an injection $k=\Ext^d_R(k,\omega_R)\to\Ext^d_R(S_i^{\op},\omega_R) \simeq S_i(\ell)$. Thus we obtain $\ell=0$ and $\Ext^d_R(S_i^{\op},\omega_R)\simeq S_i$ in $\mod^{\ZZ}A$, as desired.

Since $\Hom_R(Ae_i,\omega_R)$ belongs to $\proj^{\ZZ}A$, we can write $\Hom_R(Ae_i,\omega_R) \simeq e_jA(\ell)$ for some $j\in\mathbb{I}_A$ and $\ell\in\ZZ$. Let $\pi:Ae_i\to S_i^{\op}$ in $\mod^{\ZZ}A^{\op}$ be a natural surjection. Applying the duality $\RHom_R(-,\omega_R):\Db(\mod^{\ZZ}A^{\op})\to\Db(\mod^{\ZZ}A)$, we obtain a nonzero morphism
\[\RHom_R(\pi,\omega_R):S_i[-d]=\RHom_R(S_i^{\op},\omega_R)\to\RHom_R(Ae_i,\omega_R)=e_jA(\ell).\]
Thus $\Ext^d_A(S_i,e_jA(\ell))\neq0$ holds, and we obtain $(j,\ell)=(\nu i,-p_i)$ by \eqref{prop-Ext-S}. Thus the second claim follows.

(3) This is classical \cite[Chapter I, Proposition 8.8]{Au2}.
\end{proof}

We keep our setting of $R$ and $A$, and assume that $d=1$.
Let $S$ be the set of homogeneous nonzero divisors of $R$, and $Q:=S^{-1}A$.
In this case, the category $\CM_0^{\ZZ}A$ can be described as
\begin{align}\label{CM_0 for order}
\CM^{\ZZ}_0A =\{M \in \CM^{\ZZ}{A} \mid M\otimes_AQ\in\proj^{\ZZ}Q\},
\end{align}
see \cite[Proposition 4.15]{BIY}.

\subsection{Preliminaries on local cohomologies}
In this subsection, we give preliminaries on local cohomologies of graded modules over $\NN$-graded algebras. The results of this subsection will be used in the subsequent subsection.

Let $A$ be an $\NN$-graded algebra and let $\fm=\fm_A=A_{\geq 1}$. We define the functor
\begin{align*}
\Gamma_{\fm} : \Mod^{\ZZ}A \to \Mod^{\ZZ}A;\quad M \mapsto  \lim_{n\to\infty}\Hom_{A}(A/A_{\ge n},M).
\end{align*}
For $\Mod^{\ZZ}A^{\op}$ and $\Mod^{\ZZ}A^{\e}$, we have similar functors, which we denote respectively by $\Gamma_{\fm^{\op}}$ and $\Gamma_{\fm^{\e}}$. Let $\RGamma_{\fm}$ denote the derived functor of $\Gamma_{\fm}$. We define the $i$-th \emph{local cohomology} of $M \in \Mod^{\ZZ}A$ to be 
\[ \H_{\fm}^i(M):= \H^i \RGamma_{\fm}(M). \]
Moreover, the \emph{local cohomological dimension} of $A$ is defined by
\[
\lcd A := \sup\{i \mid \H_{\fm}^i(M) \neq 0 \ \text{for some}\ M \in \Mod^{\ZZ} A \}.
\]

Let $A$ be a Noetherian $\NN$-graded algebra.
An element $x$ of a graded module $M \in \Mod^{\ZZ}A$ is called \emph{$\fm$-torsion} if there exists an integer $m \in\ZZ$ such that $xA_{\ge m}=0$.
The torsion elements in $M$ form a graded submodule of $M$, which coincides with $\Gamma_{\fm}(M)$.
We call $M$ a \emph{$\fm$-torsion-free module} if $\Gamma_{\fm}(M)=0$ while we call $M$ a \emph{$\fm$-torsion module} if $\Gamma_{\fm}(M)=M$.
Moreover, if $M \in \mod^{\ZZ} A$, then $M$ is $\fm$-torsion if and if $M$ has finite length.

The following observation is stated in  \cite[Lemma 4.5]{VdB1} when $A$ and $B$ are Ext-finite connected $\NN$-graded algebras.

\begin{lem}\label{lem.4.5}
Let $A$ and $B$ be Noetherian locally finite $\NN$-graded algebras.  
Then
\[\RGamma_{\fm_{A \otimes B}}(M) \simeq 
\RGamma_{\fm_B} \circ \RGamma_{\fm_A}(M) \simeq 
\RGamma_{\fm_A} \circ \RGamma_{\fm_B}(M)\]
for $M \in \Mod^{\ZZ}A\otimes B$.
\end{lem}

\begin{proof}
Since $(A \otimes B)_{\ge 2n} \subset A\otimes B_{\ge n}+ A_{\ge n}\otimes B \subset (A\otimes B)_{\ge n}$, it follows that
\begin{align*}
&\Gamma_{\fm_{A \otimes B}}(-)=
\lim_{n\to \infty }\Hom_{A \otimes B}
( (A \otimes B)/(A \otimes B)_{\ge n}, -)\\
\simeq&\lim_{n\to \infty }\Hom_{A \otimes B}
( (A \otimes B)/(A\otimes B_{\ge n}+ A_{\ge n}\otimes B), -)
\simeq \lim_{n\to \infty }\Hom_{A \otimes B}
( A/A_{\ge n} \otimes B/B_{\ge n}, -)\\
\simeq&\lim_{n\to \infty }\Hom_{B}
( B/B_{\ge n}, \Hom_{A}(A/A_{\ge n}, -))
\simeq \lim_{n\to \infty }\lim_{m\to \infty }\Hom_{B}
( B/B_{\ge n}, \Hom_{A}(A/A_{\ge m}, -)).
\end{align*}
Since $B$ is Noetherian and $B/B_{\ge n} \in \mod^{\ZZ}B$, we have
\begin{align*}
\lim_{n\to \infty }\lim_{m\to \infty }\Hom_{B}
( B/B_{\ge n}, \Hom_{A}(A/A_{\ge m}, -))
&\simeq \lim_{n\to \infty }\Hom_{B}
( B/B_{\ge n}, \lim_{m\to \infty }\Hom_{A}(A/A_{\ge m}, -))\\
&= \Gamma_{\fm_B}\circ \Gamma_{\fm_A}(-).
\end{align*}
Similarly, $\Gamma_{\fm_{A \otimes B}}(-) \simeq \Gamma_{\fm_A}\circ \Gamma_{\fm_B}(-)$.
An injective resolution in $\Mod^{\ZZ}A\otimes B$ implies the result.
\end{proof}

\begin{dfn}{\cite[Definition 3.2]{AZ}}
Let $A$ be a Noetherian locally finite $\NN$-graded algebra.
For $M\in\mod^{\ZZ}A$, we say that $A$ satisfies the \emph{condition $\chi(M)$} if $\dim_k \Ext^i_A(A_0,M) <\infty$ for every $i \in \ZZ$.
We say that $A$ satisfies the \emph{condition $\chi$} if it satisfies the condition $\chi(M)$ for every $M\in\mod^{\ZZ}A$.
Dually, we define the \emph{condition $\chi^{\op}(N)$} for each $N\in\mod^{\ZZ}A^{\op}$, and also the \emph{condition $\chi^{\op}$}.
\end{dfn}

We have the following basic result.

\begin{lem}\label{lem.bim}
Let $A$ be a Noetherian locally finite $\NN$-graded algebra satisfying  $\chi(A)$ and $\chi^{\op}(A)$. Then
$R\Gamma_{\fm}(A)
\simeq R\Gamma_{\fm^{\op}}(A)$ in $\D(\Mod^{\ZZ} A^{\e})$.
\end{lem}

\begin{proof}
Since $A$ satisfies $\chi(A)$  and $\chi^{\op}(A)$, it follows that cohomologies of $R\Gamma_{\fm}(A)$ and $R\Gamma_{\fm^{\op}}(A)$ are $\fm$-torsion by \cite[Corollary 3.6]{AZ}.
Thus we have 
\[
\RGamma_{\fm}(A) \simeq  \RGamma_{\fm^{\op}}(\RGamma_{\fm}(A)) \stackrel{{\rm Lem.\,\ref{lem.4.5}}}{\simeq} 
\RGamma_{\fm}(\RGamma_{\fm^{\op}}(A))  \simeq
\RGamma_{\fm^{\op}}(A)
\]
in $\D(\Mod^{\ZZ} A^{\e})$.
\end{proof}

\subsection{Canonical modules over Artin-Schelter Gorenstein algebras}

In this subsection, we study the canonical modules over AS-Gorenstein algebras $A$. Specifically, we give the local duality theorem and show that the canonical module is an invertible bimodule.
The results of this subsection are well-known for connected $\NN$-graded algebras (see \cite{Jo, Ye, VdB1} and also \cite{Mar}) and for module-finite algebras over commutative Noetherian rings (see e.g.\ Subsection \ref{subsection: order}). They enable us to treat a much more general class of algebras.

Let $A$ be an $\NN$-graded AS-Gorenstein algebra of dimension $d$.
It is easy to see that $\Ext^i_{A}(A/A_{\ge n}, A)=0$ for all $i \neq d$, so $\H_{\fm}^i(A)=0$ for all $i \neq d$. Similarly, $\H_{\fm^{\op}}^i(A)=0$ for all $i \neq d$. 

\begin{dfn}\label{define omega}
Let $A$ be an $\NN$-graded AS-Gorenstein algebra of dimension $d$.
We call
\begin{equation}\label{define omega 2}
\omega=\omega_A := D\H_{\fm}^d(A)\stackrel{\text{ Lem.\,}\ref{lem.bim}}{\simeq} D\H_{\fm^{\op}}^d(A)  \in \Mod^{\ZZ} A^{\e} 
\end{equation}
the \emph{canonical module} of $A$.
\end{dfn}

We first state a lemma, which is instrumental in the proof of the subsequent  proposition.

\begin{lem}\label{classify injective}
Let $A$ be a Noetherian locally finite $\NN$-graded algebra.
\begin{enumerate}
\item If $I$ is a injective module in $\Mod^{\ZZ} A$, then $I \simeq Q \oplus E$, where $Q$ is a $\fm$-torsion-free injective module and $E$ is a $\fm$-torsion injective module in $\Mod^{\ZZ} A$.
\item Each indecomposable $\fm$-torsion injective module $I\in\Mod^{\ZZ}A$ is isomorphic to $D(Ae_i)(\ell)$ for some $i\in \II_A$ and $\ell\in\ZZ$.
\end{enumerate}
\end{lem}

\begin{proof}
(1) This is \cite[Proposition 7.1]{AZ}.
	
(2) It is clear that nonzero $\fm$-torsion module has nonzero socle.
For each indecomposable $\fm$-torsion injective module $I\in\Mod^{\ZZ}A$, take a simple submodule $S$. Take a projective cover $Ae_i(\ell)\to DS$ with $i\in\II_A$ and $\ell\in\ZZ$. Then $S\to D(Ae_i)(-\ell)$ is an injective envelope in $\Mod^{\ZZ}A$. Thus we have an injection $D(Ae_i)(-\ell)\to I$, which has to be an isomorphism.
\end{proof}

Thanks to the AS-Gorenstein property, we have the following explicit description of $\omega$. Notice that we do not use the condition Definition \ref{define AS-Gorenstein}(1) in the proof.

\begin{prop}\label{prop.lc1}
Let $A$ be a basic $\NN$-graded AS-Gorenstein algebra of dimension $d$.
\begin{enumerate}
\item For each $i\in \II_A$, we have
\[e_i\omega \simeq e_{\nu i }A(-p_{i})\ \mbox{ in }\ \mod^{\ZZ}A\ \mbox{ and }\ \omega e_{\nu i } \simeq Ae_i(-p_i)\ \mbox{ in }\ \mod^{\ZZ}A^{\op}.\]
\item $\omega$ is projective on both sides. More explicitly, we have
\[\omega \simeq \bigoplus_{i\in\II_A}e_{\nu i}A(-p_i)\ \mbox{ in $\mod^{\ZZ}A$ and }\ \omega \simeq\bigoplus_{i\in\II_A}Ae_i(-p_i)\ \mbox{ in $\mod^{\ZZ}A^{\op}$.}\]
\end{enumerate}
\end{prop}

\begin{proof}
It suffices to prove (1). Let $e_{\nu i} A \to I=I^\bullet$ be a minimal injective resolution of $e_{\nu i}A$ in $\K(\Mod^{\ZZ} A)$. We may write $I^\ell=Q^\ell\oplus E^\ell$, where $Q^\ell$ is a $\fm$-torsion-free injective right $A$-module and $E^\ell$ is a $\fm$-torsion injective right $A$-module by Lemma \ref{classify injective}(1). Then there exists an exact sequence of complex $0 \to E^\bullet \to I^\bullet \to Q^\bullet \to 0$.
Since $Q^\bullet$ is $\fm$-torsion-free, we have $\Ext^\ell_A(A/A_{\geq n}, e_{\nu i}A)= \H^\ell(\Hom_ A(A/A_{\geq n}, E^\bullet))$.
Furthermore, since $E^\bullet$ is $\fm$-torsion, we have
\[
\H_{\fm}^\ell(e_{\nu i }A) = \lim_{n\to \infty}\Ext^\ell_A(A/A_{\geq n}, e_{\nu i}A)= \H^\ell(\lim_{n\to \infty}\Hom_ A(A/A_{\geq n}, E^\bullet)) \simeq \H^\ell(E^\bullet).
\]
On the other hand, we have
\[ \Hom_ A(S_j,E^\ell) \simeq
\Hom_ A(S_j,I^\ell) \simeq
\Ext^\ell_ A(S_j,e_{\nu i}A) \simeq 
\begin{cases}
S^{\op}_{\nu j}(p_j) & \ell=d\ \text{and}\ j= i\\
0 & \text{otherwise}
\end{cases}\]
by \eqref{prop-Ext-S}, so $\soc E^\ell=0$ for $\ell \neq d$ and $\soc E^d= S_{i}(p_i)$. Thus $E^\ell=0$ for $\ell \neq d$ and $E^d\simeq D(Ae_{i})(p_{i}) \simeq D(Ae_{i}(-p_{i}))$ by Lemma \ref{classify injective}(2).
Therefore we have
$\H_{\fm}^d(e_{\nu i}A) \simeq \H^d(E^\bullet) \simeq E^d \simeq D(Ae_{i}(-p_{i}))$
in $\Mod^{\ZZ}A$. By \eqref{define omega 2}, we have $\omega e_{\nu i} \simeq D\H_{\fm}^d(e_{\nu i} A) \simeq Ae_{i}(-p_{i})$ in $\Mod^{\ZZ}A^{\op}$.

Similarly, using 
\[
\Ext^\ell_{A^{\op}}(S^{\op}_{\nu j}, Ae_i) \simeq \begin{cases}
S_{j}(p_{j}) & \ell=d\ \text{and}\ j=i\\
0 &\text{otherwise},
\end{cases}
\]
one can obtain $\H_{\fm^{\op}}^d(Ae_i) \simeq D(e_{\nu i}A(-p_{i}))$ in $\Mod^{\ZZ}A^{\op}$. By \eqref{define omega 2}, we have $e_i \omega \simeq D\H_{\fm^{\op}}^d(Ae_i)\simeq e_{\nu i}A(-p_{i})$ in $\Mod^{\ZZ}A$.
\end{proof}

 For $M,N \in \Mod^{\ZZ} A$, we write
\[
\HOM_A(M,N) :=\bigoplus_{i\in\ZZ}\Hom_{A}^{\ZZ}(M,N(i)). \]
Then we have a natural inclusion $\HOM_A(M,N)\subset \Hom_A(M,N)$, and the equality holds
if $M \in \mod^{\ZZ}A$.
We use the similar notation
$\RHOM_A(M, N)$ for $M, N \in  \D(\Mod^{\ZZ} A)$.
We now present the local duality theorem for AS-Gorenstein algebras as a special case of J\o rgensen's theorem \cite[Theorem 2.3]{Jo} (see also \cite[Theorem 5.1]{VdB1}).

\begin{thm}[Local Duality]\label{local duality}
	Let $A$ be a basic $\NN$-graded AS-Gorenstein algebra of dimension $d$ and let $B$ be another graded algebra.
	For $M \in  \D^{-}(\Mod^{\ZZ} B^{\op} \otimes A)$, we have 
	\[  D\RGamma_{\fm}(M) \simeq \RHOM_A(M ,D\RGamma_{\fm}(A)) \simeq
	\RHOM_A(M, \omega[d])\]
	in $\D(\Mod^{\ZZ} A^{\op} \otimes B)$.
	In particular, for $M \in \mod^{\ZZ} A$, we have 
	$D\H_{\fm}^i(M) \simeq \Ext^{d-i}_A(M, \omega)$ in $\Mod A^{\op}$ for all $i \in\ZZ$.
\end{thm}

\begin{proof}
By \cite[Theorem 2.3]{Jo}, it is enough to show that $\lcd A<\infty$.
For a finitely generated module $M \in \mod^{\ZZ}A$, we have 
\begin{align*}
	&\RGamma_{\fm}(M) \simeq \lim_{n\to \infty }\RHom_{A}(A/A_{\ge n}, M)\\
&\stackrel{\text{Lem.\,}\ref{A dual}}{\simeq} \lim_{n\to \infty }\RHom_{A^{\op}}(\RHom_A(M,A), \RHom_A(A/A_{\ge n},A)).
\end{align*}
Since $A$ is Noetherian and $\RHom_A(M,A) \in \Db(\mod^{\ZZ}A^{\op})$, we have
\begin{align*}
	&\lim_{n\to \infty }\RHom_{A^{\op}}(\RHom_A(M,A), \RHom_A(A/A_{\ge n},A))\\
	&\simeq \RHom_{A^{\op}}(\RHom_A(M,A), \lim_{n\to \infty }\RHom_A(A/A_{\ge n},A))\simeq \RHom_{A^{\op}}(\RHom_A(M,A), \RGamma_{\fm}(A))\\
	&\simeq \RHom_{A}(D\RGamma_{\fm}(A), D\RHom_A(M,A))\simeq \RHom_{A}(\omega, D\RHom_A(M,A))[-d].
\end{align*}
Since $\omega$ is projective in $\Mod^{\ZZ}A$ by Proposition \ref{prop.lc1}(2), we get 
$\H_{\fm}^i(M) \simeq \Hom_{A}(\omega, D\Ext_A^{d-i}(M,A))=0$
for any $i >d$.
Since $A$ is Noetherian, local cohomology commutes with direct limits, 
so we obtain $\H_{\fm}^i(M)=0$ for any $M \in \Mod^{\ZZ}A$ and any $i>d$.
Hence $\lcd A\leq d$.
\end{proof}

Now we are able to prove the following basic properties.

\begin{prop} \label{prop.d3}
Let $A$ be a basic $\NN$-graded AS-Gorenstein algebra.
\begin{enumerate}
\item $\RHom_{A}(\omega, \omega) \simeq A$ and $\RHom_{A^{\op}}(\omega, \omega) \simeq A$ in $\D(\Mod^{\ZZ}A^{\e})$.
\item $\omega$ is a graded invertible $A$-bimodule. Thus we have equivalences
\[-\otimes_A\omega:\mod^{\ZZ}A\to\mod^{\ZZ}A\ \mbox{ and }\ \omega\otimes_A-:\mod^{\ZZ}A^{\op}\to\mod^{\ZZ}A^{\op}.\]
\item We have a duality by Hom evaluation:
\[\RHom_A(-,\omega):\Db(\mod^{\ZZ}A)\rightleftarrows\Db(\mod^{\ZZ}A^{\op}):\RHom_{A^{\op}}(-,\omega).\] 
\end{enumerate}
\end{prop}

\begin{proof}
(1) Let $A \to I$ be an injective resolution of $A$ in $\K(\Mod^{\ZZ} A^{\e})$.
Then $\RGamma_{\fm^{\op}}(A) \simeq \Gamma_{\fm^{\op}}(I)$, whose terms are $\fm$-torsion injective $A^{\op}$-modules.
Since cohomologies of $\RGamma_{\fm}(A)$ are $\fm$-torsion modules,
\begin{align}\label{eq.endom}
\RHOM_{A^{\op}}(\RGamma_{\fm}(A), A) \simeq \RHOM_{A^{\op}}(\RGamma_{\fm}(A), \RGamma_{\fm^{\op}}(A)).
\end{align}
Hence we have
\begin{align*}
&\RHom_{A}(\omega, \omega) 
\simeq \RHom_{A}(D\RGamma_{\fm^{\op}}(A), D\RGamma_{\fm}(A)) \simeq \RHOM_{A^{\op}}(\RGamma_{\fm}(A), \RGamma_{\fm^{\op}}(A))\\ 
\stackrel{\eqref{eq.endom}}{\simeq}&\RHOM_{A^{\op}}(\RGamma_{\fm}(A), A)
\simeq \RHOM_{A}(DA, D\RGamma_{\fm}(A))
\stackrel{\text{Thm.\,\ref{local duality}}}{\simeq} D\RGamma_{\fm}(DA)
\simeq DDA
\simeq A,
\end{align*}
where we used the fact that $DA$ is $\fm$-torsion in the second-to-last isomorphism.
The proof of the isomorphism $\RHom_{A^{\op}}(\omega, \omega) \simeq A$ is similar.

(2) We have $\omega \otimes_A \Hom_A(\omega, A) \stackrel{{\rm Prop.\,\ref{prop.lc1}(2)}}{\simeq} \End_A(\omega)\stackrel{{\rm(1)}}{\simeq} A$. 
Dually we have $\Hom_A(\omega, A) \otimes_A \omega \simeq A$.

(3) Since $\RHom_A(-,\omega)=(\omega\otimes_A-)\circ\RHom_A(-,A)$, the assertion follows from Proposition \ref{A dual} and (2).
\end{proof}

\begin{prop}\label{AS is X}
Let $A$ be a basic $\NN$-graded AS-Gorenstein algebra. Then $A$ satisfies the condition $\chi$.
\end{prop}

\begin{proof}
Let $M \in \mod^{\ZZ} A$. Since $\omega$ is bounded below by Proposition \ref{prop.lc1}, so is $\Ext^i_{A}(M,\omega)$ for every $i$. Thus $\H_{\fm}^i(M)$ is bounded above for every $i$ by local duality (Theorem \ref{local duality}), so the assertion follows from \cite[Corollary 3.6]{AZ}. 
\end{proof}

\section{Artin-Schelter Gorenstein algebras of dimension one}
In this section, we investigate AS-Gorenstein algebras of dimension $1$.
More specifically, for a basic $\NN$-graded AS-Gorenstein algebra $A$ of dimension $1$,
we study the graded total quotient ring $Q$ of $A$ and 
the category $\D_{\sg,0}^{\ZZ}(A)$.

\subsection{The graded total quotient ring and the category $\qgr A$}

For a Noetherian $\NN$-graded algebra $A$, let 
\[
\mod^{\ZZ}_0 A:=\{M \in\mod ^{\ZZ}A \mid M\ \text{is finite length}\},\]
and let 
\[\qgr A:=\mod^{\ZZ} A/\mod^{\ZZ}_0 A\]
be the Serre quotient category. 
The category $\qgr A$ is traditionally called the \emph{noncommutative projective scheme} associated to
$A$ \cite{AZ}. We denote by
\[\pi:\mod^{\ZZ}A\to\qgr A\]
the canonical functor.
Let $\Db(\qgr A)$ be the bounded derived category of $\qgr A$ and let $\per(\qgr A)$ be its thick subcategory of $\Db(\qgr A)$ generated by $\proj^\ZZ A$. 

\begin{prop}\label{qgr Hom-finite}
Let $A$ be a basic $\NN$-graded AS-Gorenstein algebra. Then  $\qgr A$ is Hom-finite and Krull-Schmidt.
\end{prop}

\begin{proof}
Immediate from Proposition \ref{AS is X} and \cite[Corollary 7.3(3)]{AZ}. 
\end{proof}

In the rest of this subsection, let $A$ be a basic $\NN$-graded AS-Gorenstein algebra of dimension $1$.

\begin{dfn}
The \emph{graded total quotient ring} $Q$ of $A$ is defined as
\begin{equation}\label{define Q}
	Q:= \bigoplus_{i \in \ZZ} \Hom_{\qgr A}(A, A(i)).
\end{equation}
\end{dfn}
Thanks to Proposition \ref{qgr Hom-finite}, $Q$ is a locally finite $\ZZ$-graded algebra.
We need the following observations.

\begin{lem}\label{tensor Q from qgr}
The following assertions hold.
\begin{enumerate}
\item We have an equivalence $-\otimes_AQ:\pi(\proj^{\ZZ}A)\simeq\proj^{\ZZ}Q$.
\item For each $M\in\mod_0^{\ZZ}A$, we have $M\otimes_AQ=0$.
\item There is a functor $-\otimes_AQ:\qgr A\to\mod^{\ZZ}Q$ (by abuse of notation) making the following diagram commutative.
\[\xymatrix@R=2pc@C=3pc{
&\mod^{\ZZ}A\ar[dl]_\pi\ar[dr]^{-\otimes_AQ}\\
\qgr A\ar[rr]^{-\otimes_AQ}&&\mod^{\ZZ}Q}\]
\end{enumerate}
\end{lem}

\begin{proof}
(1) The assertion follows from 
$\Hom_{\qgr A}(A(i),A(j))=Q_{j-i}=\Hom_Q^{\ZZ}(Q(i),Q(j))$
for each $i,j\in\ZZ$.

(2) Take a projective presentation $P'\xrightarrow{f} P\to M\to0$ in $\mod^{\ZZ}A$. Then $f$ is an epimorphism in $\qgr A$ and hence splits. 
By (1), $f\otimes 1_Q:P'\otimes_AQ\to P\otimes_AQ$ is a split epimorphism in $\mod^{\ZZ}Q$. Since $M\otimes_AQ$ is a cokernel of 
$f\otimes 1_Q$ in $\mod^{\ZZ}Q$, we have $M\otimes_AQ=0$.

(3) Immediate from (2) and the universality of the Serre quotient.
\end{proof}

We prepare the following well-known result.

\begin{prop}{\cite[Proposition 7.2]{AZ}}\label{H_m and qgr}
For a Noetherian $\NN$-graded algebra $A$ and $M\in \mod^{\ZZ}A$,
there exists an exact sequence
\[0\to \H_\fm^0(M)_0\to M_0\to\Hom_{\qgr A}(A,M)\to \H_\fm^1(M)_0\to0\]
and an isomorphism $\Ext_{\qgr A}^i(A,M)\simeq \H_\fm^{i+1}(M)_0$
for each $i\ge1$.
\end{prop}

We show the following important property of the category $\qgr A$.

\begin{prop}\label{prop.Q2}
Let $A$ be a basic $\NN$-graded AS-Gorenstein algebra of dimension $1$.
\begin{enumerate}
\item The category $\qgr A$ has enough projective objects. The category of projective objects in $\qgr A$ is $\add\{A(i)\mid i\in\ZZ\}\simeq\pi(\proj^{\ZZ}A)$, which is equivalent to $\proj^{\ZZ}Q$.
\item We have an equivalence
\[ \bigoplus_{i\in\ZZ}\Hom_{\qgr A}(A(-i),-): \qgr A \to \mod^{\ZZ}Q.\]
\item There is an isomorphism of functors $\qgr A\to\mod^\ZZ Q$:
\[\bigoplus_{i \in \ZZ}\Hom_{\qgr A}(A(-i),-) \simeq -\otimes_A Q.\]\end{enumerate}
\end{prop}

\begin{proof}
(1) For any $M \in \mod^{\ZZ} A$, we have $\H_{\fm}^2(M)=0$ by local duality (Theorem \ref{local duality}) and hence $\Ext_{\qgr A}^1(A,M)=0$ by Proposition \ref{H_m and qgr}. Thus $A(i)$ is a projective object in $\qgr A$ for each $i\in\ZZ$.
Since $\mod^{\ZZ}A$ has enough projective objects $\proj^{\ZZ}A$, its Serre quotient $\qgr A$ has enough projective objects $\pi(\proj^{\ZZ}A)\simeq\add\{A(i)\mid i\in\ZZ\}$.
The last assertion is Lemma \ref{tensor Q from qgr}(1). 

(2) Thanks to (1), Morita theory gives rise to an equivalence
\[\qgr A\simeq\mod(\pi(\proj^{\ZZ}A))\simeq\mod(\proj^{\ZZ}Q)\simeq\mod^{\ZZ}Q.\]
Alternatively, it follows from \cite[Proposition 5.6]{AZ}.

(3) For $M \in \mod^{\ZZ}A$ and $i,j \in \ZZ$, consider the composition
\begin{align*}
M_i \otimes_k Q_j
&\to \Hom_{\mod^{\ZZ} A}(A(j),M(j+i)) \otimes_k \Hom_{\qgr A}(A,A(j))\\
&\to \Hom_{\qgr A}(A(j),M(j+i)) \otimes_k \Hom_{\qgr A}(A,A(j))\\
&\to \Hom_{\qgr A}(A,M(i+j))\to \Hom_{\qgr A}(A(-i-j),M),
\end{align*}
which gives a natural transformation from $-\otimes_k Q$ to $\bigoplus_{i \in \ZZ}\Hom_{\qgr A}(A(-i),-)$. 
By the universality of tensor products, we get a natural transformation $\eta$ from $-\otimes_A Q$ to $\bigoplus_{i \in \ZZ}\Hom_{\qgr A}(A(-i),-)$. 

It is easy to see that $\eta_{A(j)}$ is an isomorphism.
Moreover, for $M \in \mod^{\ZZ}A$, taking a free presentation
$F' \to F \to M \to 0$ in $\mod^{\ZZ}A$, we have 
a commutative diagram
{\small 
\begin{align*}
\xymatrix@R=2pc@C=1pc{
F' \otimes_A Q \ar[r] \ar[d]^{\simeq}_{\eta_{F'}}
&F \otimes_A Q \ar[r] \ar[d]^{\simeq}_{\eta_{F}}
&M \otimes_A Q\ar[r] \ar[d]_{\eta_{M}}
&0
\\
\bigoplus_{i \in \ZZ}\Hom_{\qgr A}(A(-i),F') \ar[r]
&\bigoplus_{i \in \ZZ}\Hom_{\qgr A}(A(-i),F) \ar[r]
&\bigoplus_{i \in \ZZ}\Hom_{\qgr A}(A(-i),M) \ar[r]
&0}
\end{align*}
}
with exact rows, where the exactness of the second row is ensured by the projectivity of $A(i)$ in $\qgr A$. Since $\eta_{F}$ and $\eta_{F'}$ are isomorphisms by \eqref{define Q}, so is $\eta_{M}$.
\end{proof}

Now we give the following fundamental results. In particular, (4) will play a fundamental role in this paper, cf.\ \cite[Proposition 4.16(c)]{BIY}.

\begin{thm}\label{prop.Q4}
Let $A$ be a basic $\NN$-graded AS-Gorenstein algebra of dimension $1$.
\begin{enumerate}
\item There exists $q>0$ such that $\pi(\proj^{\ZZ}A)\simeq\add\bigoplus_{i=1}^{q}A(i)$ and $\proj^{\ZZ}Q=\add\bigoplus_{i=1}^{q}Q(i)$. Moreover, there exists $q'>0$ such that $A(q')\simeq A$ in $\qgr A$ and $Q(q')\simeq Q$ in $\mod^{\ZZ}Q$.
\item $\qgr A$ has a progenerator
\[P:=\bigoplus_{i=1}^{q} A(i).\]
\item Let $\Lambda:=
\End_{\qgr A}(P)$. Then we have an equivalence 
\[\Hom_{\qgr A}(P, -) : \qgr A\to \mod\Lambda.\]
\item $P$ is a tilting object in $\per(\qgr A)$. Therefore we have a triangle equivalence 
\[\per(\qgr A)\simeq\per\Lambda.\]
\end{enumerate}
\end{thm}

Notice that, for each integer $q>0$ satisfying $Q(q)\simeq Q$ in $\mod^{\ZZ}Q$, all the equalities in (1) above hold.

For a Krull-Schmidt category $\mathscr C$, we denote by $\Ind \mathscr C$ a complete set of representatives of isomorphism classes of indecomposable objects of $\mathscr C$.

\begin{proof}
(1) 
Let $q:=\max\{s_1,\dots, s_n \}$ where $s_1,\dots, s_n$ are positive integers such that
$\bigoplus_{i=1}^{n}A(-s_i) \to A \to A_0 \to 0$
is exact in $\mod^{\ZZ} A$. Then we have an exact sequence
\begin{align*}\label{eq.res2}
\bigoplus_{i=1}^{n}A(-s_i) \to A \to 0
\end{align*}
in $\qgr A$.
This splits since $A$ is projective in $\qgr A$ by Proposition \ref{prop.Q2}(1).
Thus we have
\[A \in\mathscr X:=\add\bigoplus_{i=1}^{q}A(-i)\]
in $\qgr A$. In particular, we have $\mathscr X(1)\subseteq \mathscr X$. Since $(1) : \mathscr X \to \mathscr X(1)$ is an equivalence,  we see $\#\Ind\mathscr X = \#\Ind\mathscr X(1)$. Hence $\Ind \mathscr X = \Ind \mathscr X(1)$ and $\mathscr X=\mathscr X(1)$ hold. Inductively, we have $\mathscr X=\mathscr X(i)$.
Consequently, we have $\pi(\proj^{\ZZ}A)\simeq\add\bigoplus_{i=1}^{q}A(-i)$ in $\qgr A$, which gives the first equality.
It also gives the second equality by Proposition \ref{prop.Q2}(2), 

Let $q'$ be the order of the permutation of the finite set $\Ind\mathscr X$ given by the degree shift $(1)$. Then the remaining equalities hold.

(2) and (3) are follows from (1) and Morita theory, and (4) follows from (3).
\end{proof}

As in the commutative case, the $A$-module $Q$ is flat.

\begin{prop} \label{lem.sf2}
Let $A$ be a basic $\NN$-graded AS-Gorenstein algebra of dimension $1$.
\begin{enumerate}
\item $Q$ is a flat $A^{\op}$-module and a flat $A$-module.
\item $Q^{\op} \simeq \bigoplus_{i \in \ZZ} \Hom_{\qgr A^{\op}}(A,A(i))$ as $\ZZ$-graded algebras.
\end{enumerate}
\end{prop}

\begin{proof}
(2) Since $A$ is AS-Gorenstein, $\RHom_A(-,A): \Db(\mod^{\ZZ} A) \to  \Db(\mod^{\ZZ} A^{\op})$
and $\RHom_{A^{\op}}(-,A): \Db(\mod^{\ZZ} A^{\op}) \to  \Db(\mod^{\ZZ} A)$ induces a duality between $\Db(\qgr A)$ and $\Db(\qgr A^{\op})$. Hence we have
\begin{align*}
Q^{\op} &\simeq \bigoplus_{i \in \ZZ} \Hom_{\Db(\qgr A)}(A,A(i))^{\op}
\simeq \bigoplus_{i \in \ZZ} \Hom_{\Db(\qgr A^{\op})}(\RHom_A(A(i),A),\RHom_A(A,A))\\
&\simeq \bigoplus_{i \in \ZZ}\Hom_{\qgr A^{\op}}(A(-i),A)
\simeq \bigoplus_{i \in \ZZ} \Hom_{\qgr A^{\op}}(A,A(i))
\end{align*}
as $\ZZ$-graded algebras.

(1) Since $\bigoplus_{i \in \ZZ}\Hom_{\qgr A}(A(-i),-)$ is an exact functor by Proposition \ref{prop.Q2}(1), so is $-\otimes_A Q$ by Proposition \ref{prop.Q2}(3). Thus $Q$ is a left flat $A$-module.
It remains to prove that $Q$ is a right flat $A$-module, or equivalently, $Q^{\op}$ is a left flat $A^{\op}$-module. This follows from the first claim since (2) shows that $Q^{\op}$ is a $\ZZ$-graded total quotient ring of $A^{\op}$. 
\end{proof}

We also prove the following result on injectivity.

\begin{prop} \label{lem.injective}
Let $A$ be a basic $\NN$-graded AS-Gorenstein algebra of dimension $1$.
\begin{enumerate}
\item $\omega$ and $A$ are injective objects in $\qgr A$.
\item $Q$ is an injective object in $\mod^{\ZZ}Q$.
\end{enumerate}
\end{prop}

\begin{proof}
(1) For an $\fm$-torsion-free module $M\in \mod^{\ZZ} A$, we have
$\Ext^1_{\mod^{\ZZ} A}(M,\omega) \simeq D\H_{\fm}^{0}(M)_0=0$ by local duality (Theorem \ref{local duality}).
An exact sequence 
\[\Ext^1_{\mod^{\ZZ} A}(M,\omega) \to \Ext^1_{\qgr A}(M,\omega) \to 
\lim_{n\to \infty }\Ext_{\mod^{\ZZ} A}^2(M/M_{\ge n}, \omega) =0\]
implies $\Ext^1_{\qgr A}(M,\omega)=0$.
Hence $\omega$ is an injective object in $\qgr A$.
By Proposition \ref{prop.lc1}, $A$ is also an injective object in $\qgr A$.

(2) This assertion follows from (1) and Proposition \ref{prop.Q2}(2).
\end{proof}

As a consequence, we obtain the following observations.

\begin{cor}\label{cor.qql}
Let $A$ be a basic $\NN$-graded AS-Gorenstein algebra of dimension $1$.
\begin{enumerate}
\item $\mod^{\ZZ}Q\simeq \qgr A\simeq\mod\Lambda$ is a Frobenius abelian length category. 
\item $\Lambda$ is a finite dimensional selfinjective algebra.
\item $Q$ is $\ZZ$-graded Artinian, that is, $Q$ has finite length as a $\ZZ$-graded $Q$-module on both sides.
\end{enumerate}
\end{cor}

\begin{proof}
(1)(2) The equivalences follow from Proposition \ref{prop.Q2}(2) and Theorem \ref{prop.Q4}(3).
Then $\Lambda$ is finite dimensional by Proposition \ref{qgr Hom-finite} and selfinjective by Proposition  \ref{lem.injective}. Thus $\mod\Lambda$ is a Frobenius abelian length category. 

(3) This follows from (1).
\end{proof}

\subsection{The singularity category $\D_{\sg,0}^{\ZZ}(A)$}

For a Noetherian $\ZZ$-graded algebra $A$, we define the \emph{singularity category} by
\begin{align*}
\D_{\sg}^{\ZZ}(A):= \Db(\mod^{\ZZ} A)/\Kb(\proj^{\ZZ} A).
\end{align*}
Assume that $A$ is a basic AS-Gorenstein algebra of dimension $d$. Then the category $\D_{\sg}^{\ZZ}(A)$ has a natural enhancement given as follows:
We call $M\in \mod^{\ZZ} A$ graded \emph{maximal Cohen-Macaulay} (\emph{CM} for short) if $\Ext^i_A(M ,A) =0$ for all $i \neq 0$.
By Proposition \ref{prop.lc1}(2) and local duality (Theorem \ref{local duality}), $M\in \mod^{\ZZ} A$ is graded CM if and only if $\Ext^i_A(M, \omega) =0$ for all $i \neq 0$ if and only if $\H_{\fm}^i(M)=0$ for all $i\neq d$.
We write $\CM^{\ZZ}A$ for the full subcategory of $\mod^{\ZZ} A$ consisting of graded CM modules. Then the dualities in Proposition \ref{A dual} and \ref{prop.d3}(3) restrict to the dualities
\begin{align}\notag
\Hom_A(-,A):\CM^{\ZZ}A\rightleftarrows\CM^{\ZZ}A^{\op}:\Hom_{A^{\op}}(-,A),\\ \label{omega dual}
\Hom_A(-,\omega):\CM^{\ZZ}A\rightleftarrows\CM^{\ZZ}A^{\op}:\Hom_{A^{\op}}(-,\omega).
\end{align}

The \emph{stable category} $\umod^{\ZZ}A$ has the same objects as $\mod^{\ZZ}A$ and the morphisms
\[\underline{\Hom}_A^{\ZZ}(M,N):=\Hom_{\umod^{\ZZ}A}(M,N)=\Hom_A^{\ZZ}(M,N)/P^{\ZZ}(M,N),\]
where $P^{\ZZ}(M,N)$ is the subgroup of $\Hom^{\ZZ}_A(M,N)$ consisting of morphisms factoring through objects in $\proj^{\ZZ}A$.
For $M \in \mod^{\ZZ} A$,
we define its \emph{syzygy} $\Omega M$ to be the kernel of the projective cover $P \to M$. This gives the \emph{syzygy functor}
\[
\Omega: \umod^{\ZZ}A \to \umod^{\ZZ}A.
\]
The stable category of $\CM^{\ZZ}A$ is denoted by $\uCM^{\ZZ}A$.
Then $\uCM^{\ZZ}A$ has a canonical structure of a triangulated category with respect to the translation functor $M[-1] := \Omega M$.
It is well-known that we have an equivalence \cite{Bu}
\begin{equation}\label{stable singular}
\uCM^{\ZZ}A \simeq \D_{\sg}^{\ZZ}(A).
\end{equation}

In the rest of this subsection, let $A$ be a basic $\NN$-graded AS-Gorenstein algebra of dimension $1$, and $Q$ the graded total quotient ring of $A$. 

\begin{dfn}
We define
\begin{align*}
&\mathscr C_A := \thick\{\mod^{\ZZ} _{0} A,\ \proj^{\ZZ} A \} \; \subset \Db(\mod^{\ZZ} A)\\
&\D_{\sg, 0}^{\ZZ}(A):= \mathscr C_A /\Kb(\proj^{\ZZ} A).
\end{align*}
Also in view of \eqref{CM_0 for order}, we define
\begin{align*}
\CM^{\ZZ}_0A :=\{M \in \CM^{\ZZ}{A} \mid M\otimes_A Q\in\proj^{\ZZ}Q\}.
\end{align*}
The stable category of $\CM^{\ZZ}_0A$ is denoted by $\uCM^{\ZZ}_0A$.
\end{dfn}

We have the following description of $\CM_0^{\ZZ}A$ in terms of the category $\qgr A$.

\begin{prop}\label{prop.CM0}
Let $A$ be a basic $\NN$-graded AS-Gorenstein algebra of dimension $1$. We have 
\begin{align*}
\CM^{\ZZ}_0A=\{M \in \CM^{\ZZ}{A} \mid M \ \text{is a projective object in $\qgr A$}\}.
\end{align*}
\end{prop}

\begin{proof}
Since the functor $-\otimes_AQ: \qgr A\simeq\mod^{\ZZ}Q$ is an equivalence by Proposition \ref{prop.Q2}(2)(3), the assertion follows.
\end{proof}

We need the following analog of \eqref{stable singular} for $\CM_0^{\ZZ}A$.

\begin{prop}\label{prop.uCM0}
Let $A$ be a basic $\NN$-graded AS-Gorenstein algebra of dimension $1$. We have equivalences
\[\uCM^{\ZZ}_0A \simeq \D_{\sg,0}^{\ZZ}(A)\ \mbox{ and }\ \umod^{\ZZ}Q \simeq \D_{\sg}^{\ZZ}(Q).\]
\end{prop}

\begin{proof}
The second equivalence follows from the following commutative diagram
\[\xymatrix@R=2pc@C=4pc{\umod^{\ZZ}Q\ar[r]\ar[d]^{\wr}&\D_{\sg}^{\ZZ}(Q)\ar[d]^{\wr}\\
\umod \Lambda\ar[r]^{\sim}&\D_{\sg}(\Lambda),
}\]
where the vertical equivalences are induced by the equivalence $\mod^{\ZZ}Q\simeq\mod\Lambda$ given in Corollary \ref{cor.qql}.

In the rest, we prove the first equivalence.
If $M \in \CM^{\ZZ}_0A$, then $M$ is projective in $\qgr A$ by Proposition \ref{prop.CM0}.
Let $P \to M$ be an epimorphism in $\mod^{\ZZ} A$, where $P \in \proj^{\ZZ}A$.
Then this induces a split epimorphism in $\qgr A$. Thus 
$P _{\geq n} \to M_{\geq n}$ is a split epimorphism in $\mod^{\ZZ}A$ for some $n\gg 0$. Since $P, P/P _{\geq n} \in \thick\{\mod^{\ZZ} _{0} A,\ \proj^{\ZZ} A \}$, so is $P _{\geq n}$. Since $M_{\geq n}$ is a direct summand of $P _{\geq n}$,
we have $M_{\geq n} \in \thick\{\mod^{\ZZ} _{0} A,\ \proj^{\ZZ} A \}$.
Moreover, since $M/M_{\geq n} \in \thick\{\mod^{\ZZ} _{0} A,\ \proj^{\ZZ} A \}$, 
so is $M$.

Conversely, we have the following commutative diagram.
\[\xymatrix@R=2pc@C=5pc{
\underline{\CM}^{\ZZ}A\ar[r]^\sim\ar[d]^{-\otimes_AQ}&\D_{\sg}^{\ZZ}(A)\ar[d]^{-\otimes_AQ}\\
\underline{\mod}^{\ZZ}Q\ar[r]^\sim&\D_{\sg}^{\ZZ}(Q)
}\]
Since $(\mod^{\ZZ}_0A)\otimes_AQ=0$ by Lemma \ref{tensor Q from qgr}(2), for each $M\in\D_{\sg,0}^{\ZZ}(A)$, we have $M\otimes_AQ=0$ in $\D_{\sg}^{\ZZ}(Q)$. 
Since $\uCM^{\ZZ}_0A$ consists of all $M\in \uCM^{\ZZ}A$ satisfying $M\otimes_AQ=0$ in $\umod^{\ZZ}Q$, we have $\D_{\sg,0}^{\ZZ}(A)\subset\underline{\CM}^{\ZZ}_0A$. 
\end{proof}

\begin{lem}\label{lem.sf3}
If $N \in \CM^{\ZZ}_0{A}$, then $\bigoplus_{i \in \ZZ} \Hom_{\qgr A}(A,N(i))$
is a flat module in $\Mod^{\ZZ}A$.
\end{lem}

\begin{proof}
Since $N$ is projective in $\qgr A$ by Proposition \ref{prop.CM0},
$N$ is a direct summand of $\bigoplus_{j=1}^m A(s_j)$ in $\qgr A$,
so $\bigoplus_{i \in \ZZ} \Hom_{\qgr A}(A,N(i))$ is a direct summand of
$\bigoplus_{i \in \ZZ} \Hom_{\qgr A}(A, (\bigoplus_{j=1}^m A(s_j))(i))= \bigoplus_{j=1}^m Q(s_j)$.
Since $Q$ is a flat $A$-module by Proposition \ref{lem.sf2}(1), 
so is $\bigoplus_{j=1}^m Q(s_j)$ and so is $\bigoplus_{i \in \ZZ} \Hom_{\qgr A}(A, N(i))$.
\end{proof}

We define the graded \emph{Auslander-Reiten translation} by the composition
\[
\tau:
\xymatrix@C3em{
\uCM^{\ZZ} A \ar[rr]^-{\Omega_{A^{\op}}\Tr(-)}
&& \uCM^{\ZZ}{A^{\op}} \ar[rr]^-{\Hom_{A^{\op}}(-,\omega)}
&& \uCM^{\ZZ} A,}
\]
where $\Tr(-)$ is the Auslander-Bridger transpose.
It is easy to see that 
\begin{align}\label{eq.AR}
\tau \simeq \Omega\Hom_{A^{\op}}(\Hom_A(-,A),\omega)\simeq\Omega(-\otimes_A\omega):
\uCM^{\ZZ}A \to \uCM^{\ZZ}A.
\end{align}
Now we prove the Auslander-Reiten-Serre duality \cite{AR, IT, Yo}, which is already shown if $A$ is a Gorenstein order (see Proposition \ref{lem-Gorder-ASG}) or a connected $\NN$-graded AS-Gorenstein isolated singularity of dimension $d\geq 2$ (see \cite[Theorem 1.3]{Uey1}).

\begin{thm}[Auslander-Reiten-Serre duality for $\uCM^{\ZZ}_0A$] \label{thm.Serref}
Let $A$ be a basic $\NN$-graded AS-Gorenstein algebra of dimension $1$. Then there exists a functorial isomorphism
\[ \Hom_{\uCM^{\ZZ}_0A}(M,N) \simeq D\Ext^1_{\mod^{\ZZ}A}(N,\tau M)\simeq D\Hom_{\uCM^{\ZZ}_0A}(N,M\otimes_A\omega)\]
for $M, N \in \uCM^{\ZZ}_0A$. In other words, $\uCM^{\ZZ}_0A$ has a Serre functor $-\otimes_A\omega$.
\end{thm}

To prove this, we need the following basic observation (e.g.\ \cite[Proposition 7.1]{Au1}, \cite[Lemma 3.3]{IT}).

\begin{lem} \label{lem.sf1}
Let $A$ be an AS-Gorenstein algebra of dimension $d$. Then there is a functorial isomorphism
\[\underline{\Hom}^{\ZZ}_A(M, N) \simeq \Tor^{A}_1(N, \Tr M)_0\]
for each $M,N\in\mod^{\ZZ}A$.
\end{lem}

\begin{proof}[Proof of Theorem \ref{thm.Serref}]
Applying $-\otimes _A \Tr M$ to an exact sequence 
$0 \to N \to \bigoplus_{i \in \ZZ} \Hom_{\qgr A}(A, N(i)) \to \H^1_{\fm}(N) \to 0$ given by Proposition \ref{H_m and qgr}, we have 
\begin{align}\label{eq.flat}
\Tor^A_{2}(\H^1_{\fm}(N), \Tr M)  \simeq \Tor^A_{1}(N, \Tr M)
\end{align}
since $\bigoplus_{i \in \ZZ} \Hom_{\qgr A}(A, N(i))$ is flat by Lemma \ref{lem.sf3}. 
Therefore we obtain
\begin{align*}
&\Hom_{\uCM^{\ZZ}_0A}(M,N)
\stackrel{{\rm Lem.\,\ref{lem.sf1}}}{\simeq} \Tor^{A}_1(N, \Tr M)_0
\stackrel{\eqref{eq.flat}}{\simeq} \Tor^{A}_2(\H^1_{\fm}(N), \Tr M)_0\\
&\stackrel{{\rm Thm.\,\ref{local duality}}}{\simeq}\Tor^{A}_2(D\Hom_A(N, \omega), \Tr M)_0
\simeq D\Ext_{A^{\op}}^2(\Tr M, \Hom_A(N, \omega))_0\\
&\simeq\hspace{1truemm} D\Ext_{A^{\op}}^1(\Omega_{A^{\op}} \Tr M, \Hom_A(N, \omega))_0
\stackrel{\eqref{omega dual}}{\simeq}D\Ext_{A}^1(N, \Hom_{A^{\op}}(\Omega_{A^{\op}}\Tr M, \omega))_0= D\Ext_{A}^1(N, \tau M)_0\\ 
&\simeq\hspace{1truemm}
D\Hom_{\uCM^{\ZZ}_0A}(N, \tau M[1])\stackrel{\eqref{eq.AR}}{\simeq}
D\Hom_{\uCM^{\ZZ}_0A}(N, \Omega(M\otimes_A \omega)[1])
\simeq D\Hom_{\uCM^{\ZZ}_0A}(N, M\otimes_A \omega).\qedhere
\end{align*}
\end{proof}

\section{Gorenstein parameters of Artin-Schelter Gorenstein algebras}

In this section, we study Gorenstein parameters of AS-Gorenstein algebras.

\subsection{Basic properties}
Let $A$ be a basic $\NN$-graded AS-Gorenstein algebra of dimension $d$ and let $\omega$ be the canonical module of $A$.
Since $\omega$ is a $\ZZ$-graded invertible $A$-bimodule by Proposition \ref{prop.d3}(2), we have an auto-equivalence of categories
\begin{align}\label{tensor omega}
-\otimes_A \omega : \mod^{\ZZ}A \longrightarrow \mod^{\ZZ}A,
\end{align}
and this restricts to an auto-equivalence on $\mod_0^{\ZZ}A$.

\begin{lem}\label{* and omega}
We have the following isomorphisms of functors on $\mod_0^{\ZZ}A$ and $\mod_0^{\ZZ}A^{\op}$.
\begin{enumerate}
\item $\Ext_A^d(-, \omega) \simeq D(-)$ and $\Ext_{A^{\op}}^d(-, \omega) \simeq D(-)$.
\item $D\Ext_A^d(-, A) \simeq -\otimes_A\omega$ and $D\Ext_{A^{\op}}^d(-, A) \simeq \omega\otimes_A-$.
\item $\Ext_{A^{\op}}^d(D(-), A) \simeq\Hom_A(\omega,-)$ and $\Ext_A^d(D(-), A) \simeq\Hom_{A^{\op}}(\omega,-)$.
\end{enumerate}
\end{lem}
\begin{proof}
(1) By local duality (Theorem \ref{local duality}), we have $D \simeq D\Gamma_{\fm} \simeq \Ext_A^d(-,\omega)$. The second claim is the opposite version.

(2) We have $D\Ext_A^d(-, A) \stackrel{\eqref{tensor omega}}{\simeq} D\Ext_A^d(-\otimes_A\omega, \omega) \stackrel{\text{(1)}}{\simeq} DD(-\otimes_A\omega) \simeq -\otimes_A\omega$.
The second claim is the opposite version.

(3) This follows from (2).
\end{proof}

We recall the following convention given in Definition \ref{define I_A}.
Let $1=\sum_{i\in \II_A}e_i$ be a complete set of primitive orthogonal idempotents of $A$. For each $i\in \II_A$, consider the simple $A$-module $S_i=\top e_iA_0$ and the simple $A^{\op}$-module $S^{\op}_i=\top A_0e_i$.
The Nakayama permutation  $\nu: \II_A \to \II_A$ 
and the Gorenstein parameters $p_i\in\ZZ$ are defined by
$\Ext^d_A(S_i,A)\simeq S^{\op}_{\nu i}(p_i)$ or equivalently, $D\Ext^d_A(S_i,A)\simeq S_{\nu i}(-p_i)$.

\begin{prop}\label{lem-nu-eAe}
Let $A$ be a basic $\NN$-graded AS-Gorenstein algebra of dimension $d$.
\begin{enumerate}
	\item For each $i\in \II_A$, we have isomorphisms
	\begin{align*}
	&e_i\omega\simeq e_{\nu i}A(-p_i)\ \mbox{ and }\ S_i\otimes_A\omega\simeq S_{\nu i}(-p_i)\ \mbox{ in }\ \mod^{\ZZ}A,\\
	&\omega e_{\nu i} \simeq Ae_i(-p_i)\ \mbox{ and }\ \omega\otimes_AS^{\op}_{\nu i}\simeq S^{\op}_i(-p_i)\  \mbox{ in }\ \mod^{\ZZ}A^{\op}.
	\end{align*}
	\item For each $i,j\in \II_A$ and $\ell\in\ZZ$, we have an isomorphism $e_j A_{\ell}e_i\simeq e_{\nu j} A_{\ell+p_i-p_j}e_{\nu i}$ of $k$-vector spaces.
	\item $p_i\leq 0$ holds for each $i\in\II_A$ if and only if $(\mod^{<0}A)\otimes_A\omega\subset\mod^{<0}A$.
\end{enumerate}
\end{prop}

\begin{proof}
	(1) The left assertions are just Proposition \ref{prop.lc1}(1). Taking the top, we obtain the right assertions. Alternatively, they follow from $S_i\otimes_A\omega\simeq D\Ext_A^d(S_i,A)\simeq S_{\nu i}(-p_i)$ by Lemma \ref{* and omega}(2).
	
	(2) We have the following equations.
	\begin{align*}
		e_j A_{\ell}e_i & = \Hom^{\ZZ}_A(e_iA, e_jA(\ell)) 
		= \Hom^{\ZZ}_A(e_i\omega, e_j\omega(\ell)) \\
		& \stackrel{\text{(1)}}{=} \Hom^{\ZZ}_A(e_{\nu i}A(-p_i), e_{\nu j}A(\ell -p_j)) = e_{\nu j} A_{\ell+p_i -p_j}e_{\nu i},
	\end{align*}
	where for the second equation we use that $-\otimes_A\omega$ induces an auto-equivalence on $\Mod^{\ZZ}A$.

(3) Immediate from (2).
\end{proof}

The following notion plays a key role in this paper.

\begin{dfn}\label{average a-invariant}
Let $A$ be a basic $\NN$-graded AS-Gorenstein algebra of dimension $d$.	Assume that $A$ is ring-indecomposable.  We call 
	\[p^A_\mathrm{av}:=(\#\II_A)^{-1}\sum_{i \in \II_A}p_i\in \QQQ\]
	the \emph{average Gorenstein parameter} of $A$.
\end{dfn}

\begin{prop}\label{prop average a-invariant}
	Let $A$ be a ring-indecomposable basic $\NN$-graded AS-Gorenstein algebra of dimension $d$. Take an integer $g\ge1$ satisfying $\nu^g=1$. For each $i\in \II_A$, the average
	\begin{align*}
		g^{-1}\sum_{\ell=0}^{g-1}p_{\nu^\ell i}\in\QQQ
	\end{align*}
	does not depend on $i\in \II_A$, and hence it is equal to $p^A_\mathrm{av}$.
\end{prop}

\begin{proof}
	Let $p_{[i]}:=g^{-1}\sum_{\ell=0}^{g-1}p_{\nu^\ell i}$. Then $S_i\otimes_A\omega^{\otimes g} \simeq S_i(-gp_{[i]})$.
	Since $A$ is ring-indecomposable and by Lemma \ref{connected}, it is enough to show $p_{[i]}=p_{[j]}$ when $\Ext_A^1(S_i, S_j)\neq 0$.
	Let $S_j\in\Sim A$ with $\Ext_A^1(S_i, S_j)\neq 0$. Take $n$ such that $\Ext_{\mod^{\ZZ}A}^1(S_i,S_j(n))\neq 0$. For each $\ell\ge0$, we have
	\begin{align*}
		&\Ext^1_{\mod^{\ZZ} A}(S_i, S_j(n)) \simeq \Ext^1_{\mod^{\ZZ} A}(S_i\otimes_A\omega^{\otimes \ell g}, S_j\otimes_A\omega^{\otimes\ell g}(n))\\ &\simeq \Ext^1_{\mod^{\ZZ} A}(S_i, S_j(\ell g(p_{[i]}-p_{[j]})+n))
	\end{align*}
	by Proposition \ref{lem-nu-eAe}(1). If $p_{[i]} \neq p_{[j]}$, then this implies that there are infinitely many integers $n'$ with $\Ext_{\mod^{\ZZ} A}^1(S_i, S_j(n'))\neq 0$, which is impossible since $S \in \Sim A$ admits a graded minimal projective resolution by finitely generated graded projective $A$-modules.
\end{proof}

We frequently use the following observations for the case $d=1$.

\begin{prop}\label{lem-nu-eAe2}
Let $A$ be a basic $\NN$-graded AS-Gorenstein algebra of dimension $1$.
\begin{enumerate}
\item   There is an exact sequence 
\begin{align}\label{exAQ}
0\to A\to Q\to D\omega\to0\ \mbox{ in }\ \Mod^{\ZZ} A^{\e}.
\end{align}
\item For each $i\in\II_A$, there is an exact sequence
\[0\to e_{\nu i}A\to e_{\nu i}Q\to D(Ae_i)(p_i)\to0\ \mbox{ in }\ \Mod^{\ZZ}A.\]
\item For each $i\in\II_A$, we have
\begin{align*}
p_i&=\min\{\ell\in\ZZ\mid (e_i\omega)_\ell\neq0\}=\min\{\ell\in\ZZ\mid (\omega e_{\nu i})_\ell\neq0\},\\
-p_i&=\max\{\ell\in\ZZ\mid ((Q/A)e_i)_\ell\neq0\}=\max\{\ell\in\ZZ\mid (e_{\nu i}(Q/A))_\ell\neq0\}.
\end{align*}
Thus $(Ae_i)_{>-p_i}=(Qe_i)_{>-p_i}$ and $(e_{\nu i}A)_{>-p_i}=(e_{\nu i}Q)_{>-p_i}$ hold.
\item For each $i\in\II_A$, we have 
\begin{align*}
(e_i\omega)_{-p_i}\simeq S_{\nu i}(p_i)\ \mbox{ and }\ (e_{\nu i}(Q/A))_{-p_i}\simeq S_i(p_i)\ \mbox{ in }\ \mod^{\ZZ}A,\\
(\omega e_{\nu i})_{-p_i}\simeq S^{\op}_i(p_i)\ \mbox{ and }\ ((Q/A)e_i)_{-p_i}\simeq S^{\op}_{\nu i}(p_i)\ \mbox{ in }\ \mod^{\ZZ}A^{\op}.
\end{align*}
\item For each $i\in\II_A$, we have
\begin{align*}
0&=\max\{\ell\in\ZZ\mid (((Q\otimes_A\omega)/\omega)e_{\nu i})_\ell\neq0\}=\max\{\ell\in\ZZ\mid (e_i((\omega\otimes_AQ)/\omega))_\ell\neq0\}.
\end{align*}
Thus $(Q\otimes_A\omega)_{>0}=\omega_{>0}=(\omega\otimes_AQ)_{>0}$ holds.
\end{enumerate}
\end{prop}

\begin{proof}
(1) Applying Proposition \ref{H_m and qgr} to $A(i)$, we obtain an exact sequence
\[0\to \H_{\fm}^0(A)\to A\to\bigoplus_{i\in\ZZ}\Hom_{\qgr A}(A,A(i))\to\H_{\fm}^1(A)\to0.\]
Since $\H_{\fm}^0(A)=0$ and $\H_{\fm}^1(A)=D\omega$, we obtain the assertion.

(2) Multiplying \eqref{exAQ} by $e_{\nu i}$ from left, we get an exact sequence $0\to e_{\nu i}A\to e_{\nu i}Q\to e_{\nu i}D\omega\to0$ in $\Mod^{\ZZ}A$.
Since $e_{\nu i}D\omega \simeq D(\omega e_{\nu i}) \simeq D(Ae_i)(p_i)$ holds by Proposition \ref{lem-nu-eAe}(1), we obtain the desired sequence. 

(3) (4) These follow from Proposition \ref{lem-nu-eAe}(1) and (2) above.

(5) Since $e_i((Q\otimes_A\omega)/\omega)\simeq e_{\nu i}(Q/A)(-p_i)$ by Proposition \ref{lem-nu-eAe}(1), we have
\[\max\{\ell\in\ZZ\mid (e_i((\omega\otimes_AQ)/\omega))_\ell\neq0\}=\max\{\ell\in\ZZ\mid (e_{\nu i}(Q/A))_\ell\neq0\}+p_i\stackrel{{\rm(3)}}{=}0.\]
The other assertion can be shown similarly.
\end{proof}

\subsection{Gorenstein parameters under graded Morita equivalences}\label{section: GP and GME}
In this subsection, we observe the change of Gorenstein parameters under graded Morita equivalences (see Definition-Proposition \ref{graded Morita equivalent}).
Let $A$ be a ring-indecomposable basic $\NN$-graded AS-Gorenstein algebra of dimension $d$. For given integers $\ell_i \ (i \in \II_A)$, let $P=\bigoplus_{i\in \II_A}e_iA(\ell_i) \in \mod^{\ZZ}A$. As in Proposition \ref{graded Morita equivalent 2}, we obtain a $\ZZ$-graded algebra $B=\End_A(P)$ which is graded Morita equivalent to $A$. The following observation on Gorenstein parameters is the starting point of this subsection.

\begin{prop}\label{prop-eq-a-inv}
Let $A, B$ be as above. Then we have $p^{B}_i = p^A_i -\ell_i+\ell_{\nu i}$ for each $i$.
In particular, $p^{B}_\mathrm{av} = p^{A}_\mathrm{av}$ holds.
\end{prop}

\begin{proof}
By \eqref{prop-Ext-S}, $\Ext_{B}^d(S_i, e_{\nu i }B(-p_i^{B}))\neq 0$ holds.\
We have $S_i\otimes_{B}P\simeq S_i(\ell_i)$ and $e_{\nu i}P=e_{\nu i}A(\ell_{\nu i})$.
Thus by applying an equivalence $-\otimes_{B}P$, we have
\[
0 \neq \Ext_{A}^d(S_i\otimes_{B}P, e_{\nu i}P(-p_i^{B})) = \Ext_A^d(S_i, e_{\nu i }A(-p_i^{B}-\ell_i+\ell_{\nu i})).
\]
Again by \eqref{prop-Ext-S}, we get $-p_i^A=-p_i^{B}-\ell_i+\ell_{\nu i}$.
\end{proof}

The main theorem of this section is the following.

\begin{thm}\label{thm-a-inv}
Let $A$ be a ring-indecomposable basic $\NN$-graded AS-Gorenstein algebra $A$ of dimension $d$. Then there exists a ring-indecomposable basic $\NN$-graded AS-Gorenstein algebra $B$ of dimension $d$ satisfying the following conditions.
\begin{enumerate}
\item $B$ is graded Morita equivalent to $A$.
\item $|p_i^{B}-p^{B}_\mathrm{av}|<1$ holds for each $i\in\II_B$.
\end{enumerate}
In particular, if $p^A_\mathrm{av}\leq 0$ holds, then $p_i^{B}\leq 0$ for each $i\in\II_B$.
\end{thm}

In the rest of this section, we give the proof of Theorem \ref{thm-a-inv}.
For each integer $\ell$, we define $\infty + \ell := \infty$ and regard $\infty > \ell$.
Let $\II=\II_A$ and $\ZZ_{\infty}:=\ZZ\sqcup\{\infty\}$.
To prove Theorem \ref{thm-a-inv}, we study the following map induced from the grading of $A$. Let $m^A :
\II^2 \to \ZZ_{\infty}$ be a map defined as follows:
\[\label{def-m}
m^A(i, j)=\left\{\begin{array}{ll}\min\{\ell\in\ZZ \mid e_j A_{\ell}e_i\neq 0\}&e_jAe_i\neq0,\\
\infty&e_jAe_i=0.
\end{array}\right.\]
Furthermore, let $a^A:\II \to \ZZ$ be a map defined by 
\[a^A(i)=-p_i.\]
By Proposition \ref{lem-nu-eAe}(2), we have
\begin{align}\label{eq-m-a-a}
	m^A(\nu i, \nu j)=m^A(i,j)+p_i-p_j=m^A(i,j)-a^A(i)+a^A(j).
\end{align}
We induce one operation on maps $m^A$ and $a^A$.

\begin{dfn}\label{define conjugation}
Let $m : \II^2 \to \ZZ_{\infty}$, $a:\II\to\ZZ$ be maps.
For a map $s : \II \to \ZZ$, we define maps $sm : \II^2 \to \ZZ_{\infty}$ and $sa:\II\to\ZZ$ by
\[
 (sm)(i, j):=m(i, j)+s(i)-s(j)\ \mbox{ and }\ 
 (sa)(i):=a(i)+s(i)-s(\nu i).
\]
We call this operation a \emph{conjugation} and call the pair $(sm,sa)$ a \emph{conjugate} of the pair $(m, a)$.
\end{dfn}

The following is an easy observation of graded Morita equivalences and conjugations of matrices.

\begin{prop}\label{lem-Morita-conj}
Let $A$ be a basic $\NN$-graded AS-Gorenstein algebra of dimension $d$, and $m^A$ and $a^A$ the maps defined as above. Then the following assertions hold.
\begin{enumerate}
\item If $B$ is graded Morita equivalent to $A$, then the pair $(m^{B},a^B)$ is a conjugate of $(m^A,a^A)$.
\item Any conjugate of $(m^A,a^A)$ comes from an algebra which is graded Morita equivalent to $A$.
\end{enumerate}
\end{prop}

\begin{proof}
For integers $\ell_i \in\ZZ \ (i \in \II)$, let $P=\bigoplus_{i\in\II}e_iA(\ell_i)$ and $B=\End_A(P)$.
Since we have an equivalence $-\otimes_{B}P:\Mod^{\ZZ}B\to\Mod^{\ZZ}A$, we have $\Hom^{\ZZ}_B(e_iB, e_jB(\ell))=\Hom^{\ZZ}_A(e_iA(\ell_i), e_jA(\ell+\ell_j))$ for any integer $\ell$.
This induces a grading of $B$ as follows:
$e_jB_{\ell} e_i = e_j A_{\ell -\ell_i + \ell_j} e_i$.
Thus we have $m^{B}(i, j)=m^A(i, j)+\ell_i-\ell_j$.
Moreover, by Proposition \ref{prop-eq-a-inv}, we have $a^B(i)=a^A(i)+\ell_i-\ell_{\nu i}$.
Both assertions directly follow from these equations. 
\end{proof}

Theorem \ref{thm-a-inv} follows from the following statement.

\begin{lem}\label{almost constant}
In the setting of Theorem \ref{thm-a-inv}, let $m^A$ and $a^A$ be the maps defined as above. Then there exists a conjugate $(m',a')$ of $(m^A,a^A)$ which satisfies the following properties: 
\begin{itemize}
\item[$\mathrm{(C1)}$] for any $i\in\II$, there exists $c \in \ZZ$ such that $a'(\nu^\ell i) \in \{c, c+1\}$ for all $\ell \in \ZZ$.
\item[$\mathrm{(C2)}$] $m'(i, j)\geq 0$ holds for any $i, j\in \II$.
\end{itemize}
\end{lem}

\begin{proof}
This follows from Theorem \ref{thm-nonneg-app-constant} given in Appendix \ref{appendix}. In fact, thanks to Proposition \ref{prop average a-invariant}, \eqref{eq-m-a-a} and Example \ref{extend to m-data} for $N:=0$, one can replace $\infty$ appearing in the image of $m^A:\II^2\to\ZZ_\infty$ by certain non-negative integers to get a map $m':\II^2\to\ZZ_{\ge0}$ satisfying the following conditions.
\begin{itemize}
\item If $m^A(i,j)\neq\infty$, then $m'(i,j)=m^A(i,j)$.
\item For each $(i,j)\in \II^2$, we have $m'(\nu i,\nu j)=m'(i,j)-a^A(i)+a^A(j)$.
\end{itemize}
Then $(m',a^A,\nu)$ satisfies the assumption of Theorem \ref{thm-nonneg-app-constant}. Thus there exists $s:\II\to\ZZ$ such that $(sm',sa^A,\nu)$ is almost constant and non-negative. Then $(m',a'):=(sm',sa^A)$ satisfies the desired conditions.
\end{proof}

Now we are ready to prove Theorem \ref{thm-a-inv}.

\begin{proof}[Proof of Theorem \ref{thm-a-inv}]
By Lemma \ref{almost constant}, there exists a conjugate $(m',a')$ of $(m^A,a^A)$ satisfying the conditions (C1) and (C2). By Proposition \ref{lem-Morita-conj}(2), there exists an algebra $B$ which is graded Morita equivalent to $A$ and satisfies $(m^B,a^B)=(m',a')$. Then the algebra $B$ satisfies the desired conditions in Theorem \ref{thm-a-inv}.
\end{proof}

\section{Tilting theory for the singularity categories}\label{section: tilting theory}

\subsection{Our results}
The aim of this section is to discuss the existence of silting and tilting objects in the triangulated category $\D_{\sg,0}^{\ZZ}(A) \simeq \uCM^{\ZZ}_0A$. 
We start with the following observation.

\begin{prop}\label{non-existence silting}
Let $A$ be an $\NN$-graded AS Gorenstein algebra of dimension $d$.
If $\D_{\sg,0}^{\ZZ}(A)$ admits a silting object, then $\gldim A_0$ is finite.
\end{prop}

To prove this, we need the following simple observation.

\begin{lem}\label{ext ext}
For each $X,Y\in\mod A_0$ and $i>d$, we have $\Ext^i_A(\Omega^d_A(X),\Omega^d_A(Y))_0\simeq\Ext^i_{A_0}(X,Y)$.
\end{lem}

\begin{proof}
Take an exact sequence $0\to\Omega^d_A(Y)\to Q^{1-d}\to\cdots\to Q^0\to Y\to0$ in $\mod^{\ZZ}A$ with $Q^j\in\proj^{\ZZ}A$ for each $j$. 
Since $\Omega^d_A(X)\in\CM^{\ZZ}A$, we have $\Ext^j_A(\Omega^d_A(X),A)=0$ for each $j>0$.
Applying $\Hom_A(\Omega^d_A(X),-)$, we obtain
\[\Ext^i_A(\Omega^d_A(X),\Omega^d_A(Y))\simeq\Ext^{i-d}_A(\Omega^d_A(X),Y)\simeq\Ext^i_A(X,Y).\]
Take a projective resolution
\begin{equation}\label{proj res of X}
\cdots\to P^{-1}\to P^0\to X\to0
\end{equation}
in $\mod^{\ZZ}A$ such that each $P^j\in\proj^{\ZZ}A$ is generated in non-negative degrees. Applying $(-)_0$, we obtain a projective resolution
\begin{equation}\label{proj res of X 2}
\cdots\to (P^{-1})_0\to (P^0)_0\to X\to0
\end{equation}
in $\mod A_0$. Applying $\Hom_A(-,Y)_0$ to \eqref{proj res of X} and $\Hom_{A_0}(-,Y)$ to \eqref{proj res of X 2} and comparing their cohomologies, we obtain
\[\Ext^i_A(X,Y)_0\simeq\Ext^i_{A_0}(X,Y).\]
Thus the assertion follows.
\end{proof}

Now we are ready to prove Proposition \ref{non-existence silting}.

\begin{proof}[Proof of Proposition \ref{non-existence silting}]
Since $\D_{\sg,0}^{\ZZ}(A)$ admits a silting object, for each $M\in\D_{\sg,0}^{\ZZ}(A)$, it follows that $\Hom_{\D_{\sg,0}^{\ZZ}(A)}(M,M[i])=0$
for sufficiently large $i\in\ZZ$ by \cite[Proposition 2.4]{AI}.
Let $S:=A_0/\rad A_0$ and $M:=\Omega^d_A(S)\in\CM_0^{\ZZ}A$. For sufficiently large $i$, we have
\[\Ext^i_{A_0}(S,S)\stackrel{{\rm Lem.}\,\ref{ext ext}}{\simeq}\Ext^i_A(M,M)_0\simeq\Hom_{\D_{\sg,0}^{\ZZ}(A)}(M,M[i])=0.\]
Thus $\gldim A_0$ is finite.
\end{proof}

Our setting in the rest of this section is the following.
\begin{enumerate}
\item[(A1)] $A=\bigoplus_{i\in\NN}A_i$ is a ring-indecomposable basic $\NN$-graded AS-Gorenstein algebra of dimension $1$.
\item[(A2)] $\gldim A_0$ is finite.
\end{enumerate}
Let $Q$ be a graded total quotient ring of $A$, and $q\ge1$ an integer such that $\proj^{\ZZ}Q=\add\bigoplus_{i=1}^{q}Q(i)$ (see Theorem \ref{prop.Q4}(1)).
In some statements, we also assume the following condition.
\begin{enumerate}
\item[(A3)] $p_s\leq 0$ holds for each $s\in\II_A$.
\end{enumerate}
The following is the main result of this paper.

\begin{thm}\label{a and tilting}
Assume that the conditions \textup{(A1)} and \textup{(A2)} are satisfied.
\begin{enumerate}
\item $\D_{\sg,0}^{\ZZ}(A)$ has a silting object
\begin{equation}\label{define V}
V:=\bigoplus_{s\in\II_A}\bigoplus_{i=1}^{-p_s+q}e_{\nu s}A(i)_{\ge 0}.
\end{equation}
\item If the condition \textup{(A3)} is also satisfied, then the object $V$ in \eqref{define V} is a tilting object in $\D_{\sg,0}^{\ZZ}(A)$. Thus for $\Gamma:=\End_{\D_{\sg,0}^{\ZZ}(A)}(V)$, we have a triangle equivalence
\[\D_{\sg,0}^{\ZZ}(A)\simeq\per\Gamma.\]
\item $\D_{\sg,0}^{\ZZ}(A)$ admits a tilting object if and only if either $p^A_\mathrm{av}\leq0$ or $A$ is AS-regular.
\end{enumerate}
\end{thm}
We will prove Theorem \ref{a and tilting} in Subsections \ref{existence} and \ref{non-existence}. We here give more information on the object $V$ and the algebra $\Gamma$. Let 
\begin{align*}
\widetilde{\II}_A&:=\{(s,i)\in\II_A\times\ZZ\mid 1\le i\le -p_s+q\},\\
\widetilde{\II}^1_A&:=\{(s,i)\in\widetilde{\II}_A\mid 1\le i\le -p_s\},\quad
\widetilde{\II}^2_A:=\{(s,i)\in\widetilde{\II}_A\mid\max\{1,-p_s+1\}\le i\le -p_s+q\},\\ 
V^\ell&:=\bigoplus_{(s,i) \in \widetilde{\II}^\ell_A}e_{\nu s}A(i)_{\ge 0}\ \mbox{ for }\ \ell=1,2.
\end{align*}
Then $\widetilde{\II}_A=\widetilde{\II}^1_A\sqcup\widetilde{\II}^2_A$ and $V=V^1\oplus V^2$ hold.

\begin{prop}\label{information on V}
Assume that the conditions \textup{(A1)} and \textup{(A2)} are satisfied. Let $V$ be an object given by \eqref{define V}, and $\Gamma=\End_{\D_{\sg,0}^{\ZZ}(A)}(V)$.
\begin{enumerate}
\item For sufficiently large integer $N$, we have $\add V=\add\bigoplus_{i=1}^NA(i)_{\ge0}$ as subcategories of $\D_{\sg,0}^{\ZZ}(A)$.
\end{enumerate}
Assume that \textup{(A3)} is also satisfied.
\begin{enumerate}
\setcounter{enumi}{1}
\item $\Gamma$ is an Iwanaga-Gorenstein algebra.
\item $\Gamma$ is isomorphic to a subalgebra of the full matrix algebra $\mathrm{ M}_{\widetilde{\II}_A}(Q)$ given by
\begin{align*}
\Hom_{\D_{\sg,0}^{\ZZ}(A)}(e_{\nu s}A(i)_{\ge 0},e_{\nu t}A(j)_{\ge 0})
\simeq \begin{cases}
	e_{\nu t}A_{j-i}e_{\nu s}&(s,i)\in\widetilde{\II}_A^1,\ (t,j)\in\widetilde{\II}_A^1\\
	0&(s,i)\in\widetilde{\II}_A^2,\ (t,j)\in\widetilde{\II}_A^1\\
	e_{\nu t}Q_{j-i}e_{\nu s}&(s,i)\in\widetilde{\II}_A,\ (t,j)\in\widetilde{\II}_A^2.
\end{cases}
\end{align*}
\item We have
\[\Gamma\simeq\left[\begin{array}{cc}\End_A^{\ZZ}(V^1)&0\\
\Hom^{\ZZ}_A(V^1,V^2)&\End^{\ZZ}_A(V^2)\end{array}\right].\]
\end{enumerate}
\end{prop}

We will prove Proposition \ref{information on V} in Subsection \ref{proof of 2 results}.

\begin{ex}\label{End of connected case}
We give an example of the description of  $V\in\D_{\sg,0}^{\ZZ}(A)$ in Proposition \ref{information on V}(3)(4) above.
Assume that $A$ is connected (i.e.\ $A_0=k$) and $q=1$. For $a:=-p_1$, we have
\[\Gamma \simeq
	\End_{\D_{\sg,0}^{\ZZ}(A)}(V) \simeq
	\begin{bmatrix}
		k &0&\cdots&\cdots&0\\
		A_1&k &\ddots & &\vdots \\
		\vdots &\vdots&\ddots &\ddots&\vdots \\
		A_{a-1}&A_{a-2}&\cdots &k &0\\
		Q_{a}&Q_{a-1}&\cdots&Q_1&Q_0
	\end{bmatrix}.
	\]
\end{ex}

As a special case of Theorem \ref{a and tilting}, we obtain the following result.

\begin{cor}\label{a and tilting 2}
Assume that the conditions \textup{(A1)} and \textup{(A2)} are satisfied.
\begin{enumerate}
\item $\qgr A$ is semisimple if and only if $\mod^\ZZ Q$ is semisimple if and only if $\D_{\sg,0}^{\ZZ}(A)=\D_{\sg}^{\ZZ}(A)$.
\end{enumerate}
Assume that {\rm(A3)} and the equivalent conditions in \textup{(1)}  are satisfied. Let $V$ be the object \eqref{define V}.
\begin{enumerate}
\setcounter{enumi}{1}
\item $Q_0$ is semisimple, and $\Gamma=\End_{\D_{\sg}^{\ZZ}(A)}(V)$ has finite global dimension.
\item If the quiver of $A_0$ is acyclic, then there exists an ordering in the isomorphism classes of indecomposable direct summands of $V$, which forms a full strong exceptional collection in $\D_{\sg}^{\ZZ}(A)$.
\end{enumerate}
\end{cor}

We will prove Corollary \ref{a and tilting 2} in Subsection \ref{proof of 2 results}.

\begin{cor}\label{grothendieck}
Under the assumption \textup{(A1)}, \textup{(A2)}, and $p^A_\mathrm{av}\leq0$, the Grothendieck group $K_0(\uCM_0^{\ZZ}A)$ of $\uCM_0^{\ZZ}A$ is a free abelian group of
\[\rank K_0(\uCM_0^{\ZZ}A)=-\sum_{s\in\II_A}p_s+\#\Ind(\proj^\ZZ Q).\]
\end{cor}

We will prove Corollary \ref{grothendieck} in Subsection \ref{proof of 2 results}.

\subsection{Existence of silting and tilting objects}\label{existence}

To simplify the notations, let
\[a_s:=-p_s\ \mbox{ for each }\ s \in \II_A.\]
The first step of the proof of Theorem \ref{a and tilting} is the following.

\begin{lem}\label{ext vanishing}
Assume that the condition \textup{(A1)} is satisfied.
\begin{enumerate}
\item For each $N\ge1$, the object $V':=\bigoplus_{i=1}^{N}A(i)_{\ge0}$  is presilting in $\D_{\sg}^{\ZZ}(A)$.
\item If the condition \textup{(A3)} is also satisfied, then the object $V'$ is pretilting in $\D_{\sg}^{\ZZ}(A)$.
\end{enumerate}
\end{lem}

\begin{proof}
Take a minimal projective resolution
\[\cdots\to P^{-1}\to P^0\to A(i)_{\ge0}\to0.\]
Since $A=A_{\ge0}$, we have $P^{-i}\in\mod^{\ge0}A$ for each $i\ge0$.

(1) We prove $\Hom_{\D_{\sg}^{\ZZ}(A)}(V,V[\ell])=0$ for each $\ell\ge1$. It suffices to show that
\[
\Hom_{\D_{\sg}^{\ZZ}(A)}(A(i)_{\ge0},A(j)_{\ge0}[\ell])=\Ext^\ell_{\mod^{\ZZ}A}(A(i)_{\ge0},A(j)_{\ge0})=0
\]
for each $i,j\ge1$.
Since $\injdim A=1$, we have
$\Ext^\ell_A(A(i)_{\ge0},A)=\Ext^{\ell+1}_A(A(i)/A(i)_{\ge0},A)=0$.
In particular, for each morphism $f:\Omega^\ell(A(i)_{\ge0})\to A(j)_{\ge0}$, there exist morphisms $g$ and $h$ which make the following diagram commutative.
\[\xymatrix@R=2pc@C=2pc{
0\ar[r]&\Omega^\ell(A(i)_{\ge0})\ar[r]\ar[d]^f&P^{1-\ell}\ar[r]\ar[d]^g&\Omega^{\ell-1}(A(i)_{\ge0})\ar[r]\ar[d]^h&0\\
0\ar[r]&A(j)_{\ge0}\ar[r]&A(j)\ar[r]&A(j)/A(j)_{\ge0}\ar[r]&0
}\]
Since $\Omega^{\ell-1}(A(i)_{\ge0})\in\mod^{\ge0}A$ and $A(j)/A(j)_{\ge0}\in\mod^{<0}A$ holds, we have $h=0$. It is routine to check that $f$ factors through $P^{1-\ell}$. Thus we obtain $\Ext^\ell_{\mod^{\ZZ}A}(A(i)_{\ge0},A(j)_{\ge0})=0$, as desired.

(2) We prove $\Hom_{\D_{\sg}^{\ZZ}(A)}(V,V[-\ell])=0$ for each $\ell\ge1$.
It suffices to show that
\[\Hom_{\D_{\sg}^{\ZZ}(A)}(A(i)_{\ge0},A(j)_{\ge0}[-\ell])=0\]
for each $i,j\ge1$.
By Theorem \ref{thm.Serref}, $\D_{\sg,0}^{\ZZ}(A)$ has a Serre functor $-\otimes_A\omega$. Thus we have
\[\Hom_{\D_{\sg}^{\ZZ}(A)}(A(i)_{\ge0},A(j)_{\ge0}[-\ell])\simeq D\Hom_{\D_{\sg}^{\ZZ}(A)}(A(j)_{\ge0},A(i)_{\ge0}\otimes_A\omega[\ell]).\]
By (A3) and Proposition \ref{lem-nu-eAe}(3), we have $(A(j)/A(j)_{\ge0})\otimes_A\omega\in\mod^{<0}A$. Since $\omega\in\proj^{\ZZ}A$, by replacing the above diagram with the following diagram
\[\xymatrix@R=2pc@C=2pc{
	0\ar[r]&\Omega^\ell(A(j)_{\ge0})\ar[r]\ar[d]&P^{1-\ell}\ar[r]\ar[d]&\Omega^{\ell-1}(A(j)_{\ge0})\ar[r]\ar[d]&0\\
	0\ar[r]&A(i)_{\ge0}\otimes_A\omega\ar[r]&\omega(i)\ar[r]&(A(i)/A(i)_{\ge0})\otimes_A\omega\ar[r]&0,
}\]
the same argument in the proof of (1) shows the desired assertion.
\end{proof}

We also need the following technical observations.

\begin{lem}\label{induction sequence}
Assume that the condition \textup{(A1)} is satisfied. 
\begin{enumerate}
\item For each $s\in\II_A$ and $i\in\ZZ$ satisfying $a_s+1\le i\le0$, we have
\[e_{\nu s}Q(i)_{\ge0}\simeq e_{\nu s}A(i)_{\ge0}=e_{\nu s}A(i)\in\proj^{\ZZ}A.\]
\item We have $\add\bigoplus_{i=1}^qQ(i)=\proj^{\ZZ}Q=\add\bigoplus_{i=1}^q(\omega\otimes_AQ)(i)$ in $\mod^{\ZZ}Q$.
\item For each $i\in\ZZ$, we have
\[\add\bigoplus_{i=1}^qQ(i)_{\ge0}=\add V^2\ \mbox{ and }\ Q(i)_{\ge 0},A(i)_{\ge0}\in\add V\ \mbox{ in }\ \D_{\sg}^{\ZZ}(A).\]
\end{enumerate}
\end{lem}

\begin{proof}
(1) By the last assertion of Proposition \ref{lem-nu-eAe2}(3) and $a_s+1\le i$, we have $e_{\nu s}Q(i)_{\ge0}=e_{\nu s}A(i)_{\ge0}$, which equals $e_{\nu s}A(i)\in\proj^{\ZZ}A$ since $i\le0$.

(2) The left equality is the definition of $q$. 
Since $\bigoplus_{i=1}^qA(i)$ is a progenerator of $\qgr A$ and $-\otimes_A\omega:\qgr A\to\qgr A$ is an autoequivalence, $\bigoplus_{i=1}^q\omega(i)$ is a progenerator of $\qgr A$. By Proposition \ref{prop.Q2}(2)(3), $\bigoplus_{i=1}^q(\omega\otimes_AQ)(i)$ is a progenerator of $\mod^{\ZZ}Q$.
Thus the right equality follows.

(3) In $\mod^{\ZZ}A$, we have
\begin{align}\label{V^2 cup}
\bigoplus_{i=1}^{q}(\omega\otimes_AQ)(i)_{\ge0}\stackrel{{\rm Prop.\,\ref{lem-nu-eAe2}(5)}}{\simeq}\bigoplus_{i=1}^{q}\omega(i)_{\ge0}\simeq\bigoplus_{s\in\II_A}\bigoplus_{i=1}^{q}e_s\omega(i)_{\ge0}
\stackrel{{\rm Prop.\,\ref{lem-nu-eAe}(1)}}{\simeq}\bigoplus_{s\in\II_A}\bigoplus_{i=a_s+1}^{a_s+q}e_{\nu s}A(i)_{\ge0}
\end{align}
which is isomorphic to a direct sum of $V^2$ and an object in $\proj^{\ZZ}A$ by (1).
Thus $\add\bigoplus_{i=1}^qQ(i)_{\ge0}\stackrel{{\rm(2)}}{=}\add\bigoplus_{i=1}^q(\omega\otimes_AQ)(i)_{\ge0}\stackrel{\eqref{V^2 cup}}{=}\add V^2$ holds in $\D_{\sg}^{\ZZ}(A)$.

It remains to prove $A(j)_{\ge0}\in\add V$ in $\D_{\sg}^{\ZZ}(A)$ for each $j\in\ZZ$. If $j\le0$, then this is clear from $A(j)_{\ge0}=A(j)\in\proj^{\ZZ}A$. In the rest, fix $j\ge1$ and $s\in\II_A$. If $j\le a_s+q$, then $e_{\nu s}A(j)_{\ge0}\in\add V$ holds. If $j\ge a_s+1$, then Proposition \ref{lem-nu-eAe2}(3) and the first assertion imply $e_{\nu s}A(j)_{\ge0}\simeq e_{\nu s}Q(j)_{\ge0}\in\add V$ in $\D_{\sg}^{\ZZ}(A)$. Thus $A(j)_{\ge0}\simeq\bigoplus_{s\in\II_A}e_{\nu s}A(j)_{\ge0}\in\add V$ always holds.
\end{proof}

Now we prove the following key observation.

\begin{prop}\label{generate}
Assume that the conditions \textup{(A1)} and \textup{(A2)} are satisfied.
\begin{enumerate}
\item We have $\D_{\sg,0}^{\ZZ}(A)=\thick V$. The object $V$ in \eqref{define V} is silting in $\D_{\sg,0}^{\ZZ}(A)$.
\item If the condition \textup{(A3)} is also satisfied, then the object $V$ is tilting in $\D_{\sg,0}^{\ZZ}(A)$.
\end{enumerate}
\end{prop}

\begin{proof}
Thanks to Lemma \ref{ext vanishing}, it suffices to prove $\D_{\sg,0}^{\ZZ}(A)=\thick V$. We only have to show $(\mod A_0)(i)\subset\thick V$ for each $i\in\ZZ$ since $\D_{\sg,0}^{\ZZ}(A)$ is generated by these subcategories.

By Lemma \ref{induction sequence}(3), we have
\begin{align}\label{Q(i) A(i) add V}
Q(i)_{\ge0}\oplus A(i)_{\ge0}\in\add V\ \mbox{ for each }\ i\in\ZZ.
\end{align}
Using induction on $i$, we prove $(\mod A_0)(i)\subset\thick V$ for each $i\ge1$.
By our assumption (A2), it suffices to show $A_0(i)\in\thick V$. For $i\ge1$,
assume that $(\mod A_0)(j)\in\thick V$ holds for each $1\le j<i$. Take an exact sequence
\[0\to M\to  A(i)/A(i)_{\ge0}\to A_0(i)\to0,\]
with $M\in\mod^{[1-i,-1]}A$. Since $M$ has a finite filtration by $(\mod A_0)(j)$ with $1\le j\le i-1$, it belongs to $\thick V$ by our induction hypothesis. Also $A(i)/A(i)_{\ge0}\simeq A(i)_{\ge0}[1]\in\thick V$ holds by \eqref{Q(i) A(i) add V}. Thus $A_0(i)\in\thick V$ holds, as desired.

Using induction on $i$, we prove $(\mod A_0)(-i)\subset\thick V$ for each $i\ge0$.
Again by our assumption (A2), it suffices to show $(DA_0)(-i)\in\thick V$. Assume that $(\mod A_0)(-j)\subset\thick V$ holds for each $j<i$.
For each $s\in\II_A$, Proposition \ref{lem-nu-eAe2}(2) gives an exact sequence
\[0\to e_{\nu s}A\to e_{\nu s}Q\to (D(Ae_s))(-a_s)\to0\]
of $\ZZ$-graded $A$-modules. Applying $(a_s-i)$ and $(-)_{\ge0}$, we obtain an exact sequences
\[0\to e_{\nu s}A(a_s-i)_{\ge0}\to e_{\nu s}Q(a_s-i)_{\ge0}\to (D(Ae_s))(-i)_{\ge0}\to0.\]
By \eqref{Q(i) A(i) add V}, the left and middle terms belong to $\add V$, and hence $(D(Ae_s))(-i)_{\ge0}\in\thick V$. Consider an exact sequence
\[0\to (D(A_0e_s))(-i)\to (D(Ae_s))(-i)_{\ge0}\to N\to0\]
with $N\in\mod^{[0,i-1]}A$. By our induction hypothesis, $N\in\thick V$ holds, and hence $(D(A_0e_s))(-i) \in\thick V$. Thus $(DA_0)(-i)=\bigoplus_{s\in\II_A}e_s(DA_0)(-i) =\bigoplus_{s\in\II_A}(D(A_0e_s))(-i)\in\thick V$ holds, as desired.
\end{proof}

\subsection{Non-existence of tilting objects}\label{non-existence}
In this subsection, we complete our proof of Theorem \ref{a and tilting} by showing the ``only if'' part of (2).
For $M\in \Mod^{\ZZ}A$, let
\[\infd M:= \inf\{i\in\ZZ \mid M_i \neq 0\}.\]
We need the following easy observations.

\begin{lem}\label{lem.aneg1}
Let $A$ be a Noetherian locally finite $\NN$-graded algebra.
\begin{enumerate}
\item For each $M\in\mod^{\ZZ}A$, we have $\infd M\le\inf\Omega M$.
\item Let $A$ be a ring-indecomposable basic $\NN$-graded AS-Gorenstein algebra with $p^A_\mathrm{av}>0$. Then $\lim_{i\to\infty}\infd \omega^{\otimes i}=\infty$.
\end{enumerate}
\end{lem}

\begin{proof}
(1) Taking a projective cover $0\to\Omega M\to P\to M\to0$ in $\mod^{\ZZ}A$, we have $\infd M=\infd P\le\infd\Omega M$.

(2) Take any $s\in\II_A$. Since $e_s\omega  \simeq e_{\nu s}A(-p_s)$ by Proposition \ref{lem-nu-eAe}, we have 
$e_s\omega^{\otimes i}  \simeq e_{\nu^is}A(-\sum_{j=0}^{i-1}p_{\nu^js})$.
Since $p^A_\mathrm{av}>0$, we have $\sum_{j=0}^\infty p_{\nu^js}=\infty$. Thus the assertion holds.
\end{proof}

The following is a noncommutative version of \cite[Theorem 1.6(c)]{BIY}.

\begin{prop}\label{non-existence tilting}
Assume that the condition \textup{(A1)} holds. If $p^A_\mathrm{av}>0$ and $\uCM^{\ZZ}_0A \simeq \D_{\sg,0}^{\ZZ}(A)$ has a tilting object, then $A$ is AS-regular.
\end{prop}

\begin{proof}
Assume that $\uCM^{\ZZ}_0A$ has a tilting object $T$. Let $\Gamma= \End_{\uCM^{\ZZ}_0 A}(T)$.
By Theorem \ref{thm.Serref} and \cite[Proposition 4.4]{BIY}, there is a triangle equivalence $\uCM^{\ZZ}_0 A \xrightarrow{\sim} \per \Gamma$ sending $T$ to $\Gamma$ and making the following diagram commutative.
\[\xymatrix@R=2pc@C=4pc{
\uCM^{\ZZ}_0 A\ar[r]^\sim\ar[d]_{-\otimes_A \omega}&\per \Gamma\ar[d]^{\nu_\Gamma}\\
\uCM^{\ZZ}_0 A\ar[r]^\sim&\per \Gamma,
}\]
where $\nu_\Gamma:=-\Lotimes_{\Gamma} D\Gamma$. For all $i \geq 0$, $\nu_\Gamma^{i}(\Gamma)\in \D^{\leq 0}(\mod \Gamma)$ holds clearly. Thus we have
\begin{equation}\label{>0}
\H^j(\nu_\Gamma^{i}(\Gamma))=0\ \mbox{ for each $i \geq 0$ and  $j>0$.}
\end{equation}
On the other hand, take an epimorphism $f:F \to T$ in $\mod^{\ZZ}A$, where $F$ is free of finite rank. For each $i,j\ge0$, since $f \otimes_A \omega^{\otimes i}: F \otimes_A\omega^{\otimes i} \to T\otimes_A\omega^{\otimes i}$ is an epimorphism, we have
\[\infd F + \infd \omega^{\otimes i} \leq \infd (T\otimes_A\omega^{\otimes i})\stackrel{\text{ Lem.\,\ref{lem.aneg1}(1)}}{\leq} \infd(\Omega^j(T \otimes_A \omega^{\otimes i})).\]
By Lemma \ref{lem.aneg1}(2), we can take $i \gg 0$ such that $\infd(T \otimes_A \omega^{\otimes i})$ is greater than all the degrees of the minimal generators of $T$.
Then for each $j\ge0$, we have $\Hom^{\ZZ}_A(T,\Omega^j(T \otimes_A \omega^{\otimes i}))=0$ and hence
\begin{align*} \H^{-j}(\nu_\Gamma^{i}(\Gamma))
&\simeq \Hom_{\Db(\mod \Gamma)}(\Gamma,\nu_\Gamma^{i}(\Gamma)[-j])\\
&\simeq \Hom_{\uCM^{\ZZ}_0 A}(T, (T \otimes_A \omega^{\otimes i})[-j])\\
&\simeq \Hom_{\uCM^{\ZZ}_0 A}(T, \Omega^j(T \otimes_A \omega^{\otimes i}))=0.
\end{align*}
This together with \eqref{>0} shows that $\nu_\Gamma^{i}(\Gamma)$ is acyclic and hence zero in $\Db(\mod\Gamma)$. Since $\nu_\Gamma$ is an autoequivalence, we have $T=0$ and hence $\underline{\CM}_0^{\ZZ}A=0$. Thus $A$ is AS-regular.
\end{proof}

\begin{proof}[Proof of Theorem \ref{a and tilting}]
(1) Proposition \ref{generate}(1) implies the assertion.

(2) Proposition \ref{generate}(2) implies the assertion.

(3) We prove the ``if'' part. If $A$ is AS-regular, then $\D_{\sg,0}^{\ZZ}(A)$ is zero and hence has a tilting object. If $p^A_\mathrm{av}\leq 0$, then Theorem \ref{thm-a-inv} shows that there exists $B$ which is graded Morita equivalent and satisfies the conditions (A1), (A2), and (A3). By (2), $\D_{\sg,0}^{\ZZ}(A)\simeq\D_{\sg,0}^{\ZZ}(B)$ has a tilting object.
The ``only if'' part follows from Proposition \ref{non-existence tilting}.
\end{proof}

\subsection{Proof of Proposition \ref{information on V} and Corollaries \ref{a and tilting 2} and \ref{grothendieck}}\label{proof of 2 results}

The following observation gives a description of the endomorphism algebra of $V$.

\begin{prop}\label{j-i}
Assume that the conditions \textup{(A1)} and \textup{(A2)} are satisfied. For each $i,j\in\ZZ$ and $s,t\in\II_A$, the following assertions hold.
\begin{enumerate}
\item We have
\[Q_{j-i}\simeq\Hom_A^{\ZZ}(A(i)_{\ge 0},Q(j))\simeq\Hom_A^{\ZZ}(A(i)_{\ge 0},Q(j)_{\ge 0}).\]
\item We have
\begin{align*}
e_{\nu t}A_{j-i}&\simeq\Hom_A^{\ZZ}(A(i),e_{\nu t}A(j))\subset\Hom_A^{\ZZ}(A(i)_{\ge 0},e_{\nu t}A(j))
\simeq\Hom_A^{\ZZ}(A(i)_{\ge 0},e_{\nu t}A(j)_{\ge 0}).
\end{align*}
If $j\le a_t$, then the middle inclusion is an isomorphism 
\item Assume $j\le a_t$. Then 
$\Hom_A^{\ZZ}(Q(i)_{\ge 0},e_{\nu t}A(j)_{\ge 0})=0$. If moreover $i>a_s$, then
\[\Hom_A^{\ZZ}(e_{\nu s}A(i),e_{\nu t}A(j))\simeq\Hom_A^{\ZZ}(e_{\nu s}A(i)_{\ge 0},e_{\nu t}A(j)_{\ge 0})=0.\]
\item If $i\le a_s$, $j\le a_t$ and $e_{\nu s}A(i)_{\geq 0} \simeq e_{\nu t}A(j)_{\geq 0}$, then $s=t$ and $i=j$ hold.
\item If \textup{(A3)} and $i\ge1$ hold, then
\[\Hom_A^{\ZZ}(A(i)_{\ge 0},A(j)_{\ge 0})=\underline{\Hom}_A^{\ZZ}(A(i)_{\ge 0},A(j)_{\ge 0}).\]
\end{enumerate}
\end{prop}

\begin{proof}
(1) The right isomorphism is clear. We prove the left one.
We have an equivalence $-\otimes_AQ:\qgr A\to\mod^{\ZZ}Q$. Since $A(i)_{\ge0}\simeq A(i)$ in $\qgr A$, we have $A(i)_{\ge0}\otimes_AQ\simeq Q(i)$ in $\mod^{\ZZ}Q$. Thus 
\[\Hom_A^{\ZZ}(A(i)_{\ge 0},Q(j))\simeq \Hom_Q^{\ZZ}(A(i)_{\ge0}\otimes_AQ,Q(j))\simeq\Hom_Q^{\ZZ}(Q(i),Q(j))=Q_{j-i}.\]

(2) The left and right isomorphisms are clear. We prove the middle inclusion. There exists an exact sequence $0\to A(i)_{\ge0}\to A(i)\to A(i)/A(i)_{\ge0}\to0$ in $\mod^{\ZZ}A$.
Applying $\Hom_A^{\ZZ}(-,e_{\nu t}A(j))$ to it, we obtain an exact sequence
\begin{align}\notag
&\Hom_A^{\ZZ}(A(i)/A(i)_{\ge0},e_{\nu t}A(j))\to\Hom_A^{\ZZ}(A(i),e_{\nu t}A(j))\to\Hom_A^{\ZZ}(A(i)_{\ge 0},e_{\nu t}A(j))\\
\to&\Ext^1_{\mod^{\ZZ}A}(A(i)/A(i)_{\ge0},e_{\nu t}A(j)),\label{(-,A(j))}
\end{align}
where the first term is zero since $A(i)/A(i)_{\ge0}\in\mod^{\ZZ}_0A$. Thus we obtain the middle inclusion. If $j\le a_t$, then by \eqref{prop-Ext-S}, we have $\Ext^1_{\mod^{\ZZ}A}(S_s(k),e_{\nu t}A(j))=0$ for each $s\in\II$ and $k\ge1$. Thus the last term of \eqref{(-,A(j))} vanishes, and the middle inclusion is an isomorphism.

(3) We prove the first assertion. By Theorem \ref{prop.Q4}(1), there exists $q''\in\ZZ$ such that $Q\simeq Q(q'')$ in $\mod^{\ZZ}Q$ and $i+q''>a_s$ for each $s\in\II_A$. Then we have $Q(i)_{\ge0}\simeq Q(i+q'')_{\ge0}=A(i+q'')_{\ge0}$ and hence
\[\Hom_A^{\ZZ}(Q(i)_{\ge 0},e_{\nu t}A(j)_{\ge 0})\simeq \Hom_A^{\ZZ}(A(i+q'')_{\ge 0},e_{\nu t}A(j)_{\ge 0})\stackrel{\text{(2)}}{\simeq}e_{\nu t}A_{j-i-q''}=0\]
as desired. The second assertion follows from
\begin{align*}
\Hom_A^{\ZZ}(e_{\nu s}A(i),e_{\nu t}A(j))\stackrel{\text{(2)}}{\simeq}\Hom_A^{\ZZ}(e_{\nu s}A(i)_{\ge 0},e_{\nu t}A(j)_{\ge 0})
\stackrel{i>a_s}{\simeq}\Hom_A^{\ZZ}(e_{\nu s}Q(i)_{\ge 0},e_{\nu t}A(j)_{\ge 0})=0
\end{align*}
where the last equality follows from the first assertion.

(4) By (2), we have 
$\Hom_A^{\ZZ}(e_{\nu s}A(i)_{\ge 0},e_{\nu t}A(j)_{\ge 0})=e_{\nu t}A_{j-i}e_{\nu s}$ and
$\Hom_A^{\ZZ}(e_{\nu t}A(j)_{\ge 0},e_{\nu s}A(i)_{\ge 0})=e_{\nu s}A_{i-j}e_{\nu t}$.
Since they are nonzero, we have $i=j$. Since the image of the multiplication map $e_{\nu t}A_0e_{\nu s}\times e_{\nu s}A_0e_{\nu t}\to e_{\nu t}Ae_{\nu t}$ contains $e_{\nu t}$, we have $s=t$.

(5) It suffices to show that, for each $k\in\ZZ$, at least one of $\Hom_A^{\ZZ}(A(i)_{\ge0},A(k))$ and $\Hom_A^{\ZZ}(A(k),A(j)_{\ge0})$ is zero.
If $k>0$, then we have $\Hom_A^{\ZZ}(A(k),A(j)_{\ge0})\simeq (A(j)_{\ge0})_{-k}=0$. Assume $k\le0$. Then for each $s\in\II_A$, we have $k\le a_s$ by (A3) and hence
\[\Hom_A^{\ZZ}(A(i)_{\ge0},e_{\nu s}A(k))\stackrel{\text{(2)}}{\simeq}e_{\nu s}A_{k-i}\stackrel{k-i<0}{=}0.\]
Thus $\Hom_A^{\ZZ}(A(i)_{\ge0},A(k))=0$.
\end{proof}

We are ready to prove Proposition \ref{information on V}.

\begin{proof}[Proof of Proposition \ref{information on V}]
(1) If $N\ge\max\{-p_s+q\mid s\in\II_A\}$, then $\add V\subset\add\bigoplus_{i=1}^NA(i)_{\ge0}$ holds. The reverse inclusion follows from Lemma \ref{induction sequence}(3).

(2) By Theorem \ref{a and tilting}(2), we have a triangle equivalence $\per\Gamma\simeq\D^{\ZZ}_{\sg,0}(A)$. Thus $\per\Gamma$ has a Serre functor by Theorem \ref{thm.Serref}. Hence $\Gamma$ is Iwanaga-Gorenstein by \cite[Proposition 4.4]{BIY}.

(3) For each $(s,i),(t,j)\in\widetilde{\II}_A$, by Proposition \ref{j-i}(5), we have
\begin{align}\label{entry}
\Hom_{\D_{\sg,0}^{\ZZ}(A)}(e_{\nu s}A(i)_{\ge 0},e_{\nu t}A(j)_{\ge 0})=\Hom_A^{\ZZ}(e_{\nu s}A(i)_{\ge 0},e_{\nu t}A(j)_{\ge 0}).
\end{align}
If $(t,j)\in\widetilde{\II}_A^2$, then \eqref{entry} is $e_{\nu t}Q_{j-i}e_{\nu s}$ by Proposition \ref{j-i}(1). 
If $(t,j)\in\widetilde{\II}_A^1$, then \eqref{entry} is $e_{\nu t}A_{j-i}e_{\nu s}$ by Proposition \ref{j-i}(2), which is zero if moreover $(s,i)\in\widetilde{\II}_A^2$ by Proposition \ref{j-i}(3).

(4) This follows immediately from (3).
\end{proof}

We prepare the following observation.

\begin{lem}\label{End U}
Assume that the conditions \textup{(A1)} and \textup{(A2)} are satisfied. Let
\[W^1:=\bigoplus_{s\in\II_A}\bigoplus_{1\le i\le -p_s}e_{\nu s}A(i).\]
Then $\End_A^{\ZZ}(W^1)$ has finite global dimension.
\end{lem}

\begin{proof}
Let $a:=\max\{-p_s\mid s\in\II_A\}$ and
$W^2:=\bigoplus_{s\in\II_A}\bigoplus_{i=-p_s+1}^{a}e_{\nu s}A(i)$.
Since $W^1\oplus W^2=\bigoplus_{i=1}^{a}A(i)_{\ge0}$, we have 
\[\End^{\ZZ}_A(W^1\oplus W^2)\simeq\End^{\ZZ}_A(\bigoplus_{i=1}^{a}A(i))\simeq\begin{bmatrix}A_0&0&\cdots&0\\ A_1&A_0&\cdots&0\\ \vdots&\vdots&\ddots&\vdots\\ A_a&A_{a-1}&\cdots&A_0\end{bmatrix}.\]
Since $A_0$ has finite global dimension, so does $\End^{\ZZ}_A(W^1\oplus W^2)$.
Since $\Hom_A^{\ZZ}(W^2,W^1)=0$ holds by Proposition \ref{j-i}(3), we have
\[\End^{\ZZ}_A(W^1\oplus W^2)\simeq\left[\begin{array}{cc}\End_A^{\ZZ}(W^1)&0\\
\Hom_A^{\ZZ}(W^1,W^2)&\End^{\ZZ}_A(W^2)\end{array}\right].\]
Thus $\End^{\ZZ}_A(W^i)$ also has finite global dimension for $i=1,2$.
\end{proof}

We prove Corollary \ref{a and tilting 2}.

\begin{proof}[Proof of Corollary \ref{a and tilting 2}]
(1) The first equivalence follows from Proposition \ref{prop.Q2}(2). The second equivalence follows from Proposition \ref{prop.CM0} since each object in $\qgr A$ is isomorphic to an object in $\CM^{\ZZ}A$.

(2) Using Proposition \ref{information on V}(4)(3), we have 
\begin{align*}
\Gamma=\End_{\D_{\sg,0}^{\ZZ}(A)}(V)&\simeq\left[\begin{array}{cc}\End_A^{\ZZ}(V^1)&0\\
\Hom_A^{\ZZ}(V^1,V^2)&\End^{\ZZ}_A(V^2)\end{array}\right]
\simeq\left[\begin{array}{cc}\End_A^{\ZZ}(W^1)&0\\
\Hom_A^{\ZZ}(V^1,V^2)&\End^{\ZZ}_Q(V^2\otimes_AQ)\end{array}\right].
\end{align*}
Moreover we see that $\End^{\ZZ}_A(W^1)$ has finite global dimension by Lemma \ref{End U}, and $\End^{\ZZ}_Q(V^2\otimes_AQ)$ is semisimple by (1). Thus $\Gamma$ has finite global dimension.

(3) Since the quiver of $\End_{\D_{\sg,0}^{\ZZ}(A)}(V)$ is also acyclic by Proposition \ref{information on V}(3)(4), there is an ordering in the isomorphism classes of the indecomposable projective $\Gamma$-modules $\per\Gamma$ which forms a full strong exceptional collection.
\end{proof}

We end this section with proving Corollary \ref{grothendieck}.

\begin{proof}[Proof of Corollary \ref{grothendieck}]
Thanks to Theorem \ref{thm-a-inv}, we can assume $p_s\le0$ for each $s\in\II_A$.
By Theorem \ref{a and tilting}(2), $\uCM_0^{\ZZ}A$ admits a tilting object $V$, and hence $K_0(\uCM_0^{\ZZ}A)$ is a free abelian group of rank $\#\Ind(\add V)$.
By Proposition \ref{information on V}(4), we have $\#\Ind(\add V)=\#\Ind(\add V^1)+\#\Ind(\add V^2)$.
We have $\#\Ind(\add V^1)=-\sum_{s\in\II_A}p_s$ by Proposition \ref{j-i}(4) and $\#\Ind(\add V^2)=\#\Ind(\proj^{\ZZ}Q)$ by Lemma \ref{induction sequence}(3).
Thus the assertion follows.
\end{proof}

\section{Orlov-type semiorthogonal decompositions}\label{section: orlov}

Let $\XX$ and $\YY$ be full subcategories in a triangulated category $\TT$.
We denote by $\XX*\YY$ the full subcategory of $\TT$ whose objects consisting of $Z\in\TT$ such that there is a triangle $X\to Z\to Y\to X[1]$ with $X\in\XX$ and $Y\in\YY$.
When $\Hom_{\TT}(\XX,\YY)=0$ holds, we write $\XX*\YY=\XX\perp\YY$. For full subcategories $\XX_1,\ldots,\XX_n$, we define $\XX_1*\cdots*\XX_n$ and $\XX_1\perp\cdots\perp\XX_n$ inductively.
If $\TT=\XX_1\perp\cdots\perp\XX_n$ for thick subcategories $\XX_1,\ldots,\XX_n$ of $\TT$, we say that $\TT=\XX_1\perp\cdots\perp\XX_n$
is a (weak) \emph{semiorthogonal decomposition} of $\TT$ \cite{Or}.

We start this section with the following elementary distributive law of thick subcategories.

\begin{lem} \label{distributive}
Let $\TT$ be a triangulated category, and $\AA$, $\BB$, $\CC$ and $\DD$ thick subcategories of $\TT$.
\begin{enumerate}
\item If $\AA\subset\CC$, then $(\AA*\BB)\cap\CC=\AA*(\BB\cap\CC)$. Similarly, if $\BB\subset\CC$, then $(\AA*\BB)\cap\CC=(\AA\cap\CC)*\BB$.
\item If $\TT=\AA*\BB=\CC*\DD$ with $\AA\subset\CC$ and $\BB\supset\DD$, then
\[\TT=\AA*(\BB\cap\CC)*\DD.\]
\item If $\AA\subset\CC\subset\AA*\BB$ and $\BB\subset\DD\subset\AA*\BB$, then
\[\CC\cap\DD=(\AA\cap\DD)*(\BB\cap\CC).\]
\end{enumerate}
\end{lem}

\begin{proof}
(1) We only show the former assertion. Clearly ``$\supset$'' holds. To prove ``$\subset$'', for each $C\in(\AA*\BB)\cap\CC$, take a triangle $A\to C\to B\to A[1]$ with $A\in\AA$ and $B\in\BB$. Since $A\in\AA\subset\CC$, we have $B\in\BB\cap\CC$, and hence $C\in \AA*(\BB\cap\CC)$.

(2) Applying $-\cap\CC$ to $\TT=\AA*\BB$ and using (1), we obtain $\CC=\AA*(\BB\cap\CC)$. Thus $\TT=\CC*\DD=\AA*(\BB\cap\CC)*\DD$.

(3) We still have $\CC=\AA*(\BB\cap\CC)$. Applying $-\cap\DD$ and using (1) again, we have  the assertion.
\end{proof}

Throughout this section, we assume the following.
\begin{enumerate}
\item[\rm(B1)]  $A$ is a $\ZZ$-graded Iwanaga-Gorenstein ring such that $A=\bigoplus_{i\ge0}A_i$.
\end{enumerate}
The aim of this section is to realize Verdier quotients
\begin{align*}
\D_{\sg}^{\ZZ}(A):= \Db(\mod^{\ZZ} A)/\Kb(\proj^{\ZZ} A)\ \mbox{ and }\ \Db(\qgr A)=\Db(\mod^{\ZZ} A)/\Db(\mod_0^{\ZZ} A)
\end{align*}
of $\Db(\mod^{\ZZ}A)$ as thick subcategories of $\Db(\mod^{\ZZ}A)$. 
By (B1), we have a duality
\[(-)^*=\RHom_A(-,A):\Db(\mod^{\ZZ}A)\to \Db(\mod^{\ZZ}A^{\op}).\]
Let $i\in\ZZ$. Consider full subcategories $\mod^{\ge i}A$ and $\mod^{<i}A$ of $\mod^{\ZZ}A$ given in \eqref{define mod>0A}, and let
\[\DD^{\ge i}_A:=\Db(\mod^{\ge i}A)\ \mbox{ and }\ \DD^{<i}_A:=\Db(\mod^{<i}A).\]
Define full subcategories of $\proj^{\ZZ}A$ by
\begin{align*}
\proj^{<i}A:=\add\{A(j)\mid j>-i\}&\ \mbox{ and }\ \proj^{\ge i}A:=\add\{A(j)\mid j\le-i\}.
\end{align*}
Then $(\proj^{<i}A,\proj^{\ge i}A)$ gives a torsion pair in $\proj^{\ZZ}A$. Also define full subcategories of $\Kb(\proj^{\ZZ}A)$ by
\begin{align*}
\PP^{<i}_A:=\Kb(\proj^{<i}A)
&\ \mbox{ and }\ \PP^{\ge i}_A:=\Kb(\proj^{\ge i}A).
\end{align*}
Then $(\PP^{<i}_{A},\PP^{\ge i}_{A})$ gives a semiorthogonal decomposition of $\Kb(\proj^{\ZZ}A)$.

Now we consider the following condition.
\begin{enumerate}
\item[\rm(B2)] $\gldim A_0$ is finite.
\end{enumerate}
Under this condition, we can realize the $\ZZ$-graded singularity category $\D_{\sg}^{\ZZ}(A)$ in $\Db(\mod^{\ZZ}A)$.

\begin{thm}\label{SOD of sg}
Assume that the assumptions \textup{(B1)} and \textup{(B2)} hold. For each $i\in\ZZ$, we have a semiorthogonal decomposition
\begin{align}\label{SOD0'}
\Db(\mod^{\ZZ}A)&=\PP^{<i}_A\perp(\DD^{\ge i}_A\cap(\DD^{>-i}_{A^{\op}})^*)\perp\PP^{\ge i}_A.
\end{align}
Thus we have a triangle equivalence
\begin{align*}
F_i:\D_{\sg}^{\ZZ}(A)&\simeq\DD^{\ge i}_A\cap(\DD^{>-i}_{A^{\op}})^*\subset\Db(\mod^{\ZZ}A).
\end{align*}
\end{thm}

\begin{proof}
Although this is exactly \cite[Corollary 2.4]{IY}, we include a complete proof which will also give an idea for other proofs. By (B1) and (B2), we have semiorthogonal decompositions
\begin{align}\label{SOD1'}
\Db(\mod^{\ZZ}A)&=\PP^{<i}_A\perp \DD^{\ge i}_A,\\ \label{SOD2'}
\Db(\mod^{\ZZ}A^{\op})&=\PP^{\le -i}_{A^{\op}}\perp \DD^{>-i}_{A^{\op}}.
\end{align}
Applying $(-)^*$ to \eqref{SOD2'}, we obtain a semiorthogonal decomposition
\begin{equation}\label{SOD3'}
\Db(\mod^{\ZZ}A)=(\DD^{>-i}_{A^{\op}})^*\perp(\PP^{\le -i}_{A^{\op}})^*=(\DD^{>-i}_{A^{\op}})^*\perp\PP^{\ge i}_A.
\end{equation}
Since $\DD^{\ge i}_A\supset\PP^{\ge i}_A$, we can apply Lemma \ref{distributive}(2) to \eqref{SOD1'} and \eqref{SOD3'} to obtain \eqref{SOD0'}.
\end{proof}

Our next results require the following assumptions.
\begin{enumerate}
\item[\rm(B3)] The $\ZZ$-graded ring $A$ is a $\ZZ$-graded $k$-algebra over a field $k$ such that $\dim_kA_0<\infty$.
\item[\rm(B4)] There exists $W\in\Db(\mod^{\ZZ}A^{\e})$ such that $-\Lotimes_AW:\Db(\mod^{\ZZ}A)\simeq\Db(\mod^{\ZZ}A)$ is an autoequivalence and $\RHom_A(-,W)\simeq D$ as functors $\Db(\mod^{\ZZ}_0A)\simeq\Db(\mod^{\ZZ}_0A^{\op})$ for the $k$-dual $D$.
\end{enumerate}
In this case, $\RHom_A(W,W)=A$ holds, a quasi-inverse of $-\Lotimes_AW$ is given by $-\Lotimes_AW^{-1}$ for $W^{-1}:=\RHom_A(W,A)\simeq\RHom_{A^{\op}}(W,A)\in\Db(\mod^{\ZZ}A^{\e})$,
and we have equivalences $-\otimes_AW:\Kb(\proj^{\ZZ}A)\simeq\Kb(\proj^{\ZZ}A):-\Lotimes_AW^{-1}$. Also we have an isomorphism of functors
\begin{equation}\label{* omega commute}
(-)^*\otimes_AW\simeq(W^{-1}\otimes_A-)^*.
\end{equation}
Let $i\in\ZZ$. We define full subcategories by
\begin{align*}
&\mod_0^{\ge i}A:=\mod^{\ge i}A\cap\mod_0^{\ZZ}A\subset\mod^{\ZZ}A\ \mbox{ and }\ \SS^{\ge i}_A:=\Db(\mod_0^{\ge i}A)\subset\Db(\mod^{\ZZ}A).
\end{align*}
Then $(\SS^{\ge i}_{A},\DD^{<i}_{A})$ gives a semiorthogonal decomposition of $\Db(\mod_0^{\ZZ}A)$.
For a subcategory $\CC$ of $\Db(\mod^{\ZZ}A^{\op})$, we write $W^{-1}\Lotimes_A\CC:=\{W^{-1}\Lotimes_AX\mid X\in\CC\}$,
and define full subcategories of $\Db(\mod^{\ZZ}A^{\op})$ by
\begin{align*}
\MM^{>i}_{A^{\op}}:=W^{-1}\Lotimes_A\DD^{>i}_{A^{\op}}
&\ \mbox{ and }\ \MM^{\le i}_{A^{\op}}:=W^{-1}\Lotimes_A\DD^{\le i}_{A^{\op}}.
\end{align*}
Then $(\MM^{>i}_{A^{\op}},\MM^{\le i}_{A^{\op}})$ gives a semiorthogonal decomposition of $\Db(\mod^{\ZZ}A^{\op})$.
Clearly the $k$-duality $D$ gives a duality
\[D:\mod^{<i}A\to \mod^{>-i}_0A^{\op},\]
which gives the following descriptions of the categories $\MM^{>-i}_{A^{\op}}$ and $\MM^{\le-i}_{A^{\op}}$.

\begin{lem}
We have 
\begin{align}\label{X^*}
(\DD^{<i}_A)^*&=\MM^{>-i}_{A^{\op}}\cap\Db(\mod_0^{\ZZ}A^{\op}),\\ \label{Y^*}
(\SS^{\ge i}_A)^*&=\MM^{\le-i}_{A^{\op}}.
\end{align}
\end{lem}

\begin{proof}
Since $D\simeq W\otimes_A(-)^*$ holds on $\Db(\mod_0^{\ZZ}A)$ by (B4), we have
\[(\DD^{<i}_A)^*=W^{-1}\Lotimes_AD(\DD^{<i}_A)=W^{-1}\Lotimes_A\SS^{>-i}_{A^{\op}}=\MM^{>-i}_{A^{\op}}\cap\Db(\mod_0^{\ZZ}A^{\op}).\]
Thus the first assertion follows. The proof of the second one is completely parallel.
\end{proof}

The derived category $\Db(\qgr A)$ can be realized as a thick subcategory of $\Db(\mod^{\ZZ}A)$ as follows.

\begin{thm}\label{SOD of qgr}
Assume that the assumptions \textup{(B1)}, \textup{(B3)} and \textup{(B4)} hold. For each $i\in\ZZ$, we have a semiorthogonal decomposition
\begin{align}\label{SOD0}
\Db(\mod^{\ZZ}A)=\SS^{\ge i}_A\perp(\DD^{\ge i}_A\cap(\MM^{>-i}_{A^{\op}})^*)\perp\DD^{<i}_A.
\end{align}
Thus we have a triangle equivalence
\[G_i:\Db(\qgr A)\simeq\DD^{\ge i}_A\cap(\MM^{>-i}_{A^{\op}})^*\subset\Db(\mod^{\ZZ}A).\]
\end{thm}

\begin{proof}
We have semiorthogonal decompositions
\begin{align}\label{SOD1}
\Db(\mod^{\ZZ}A)&=\DD^{\ge i}_A\perp \DD^{<i}_A,\\ \label{SOD2}
\Db(\mod^{\ZZ}A^{\op})&=\DD^{>-i}_{A^{\op}}\perp \DD^{\le-i}_{A^{\op}}.
\end{align}
Applying $W^{-1}\Lotimes_A-$ to \eqref{SOD2}, we have a semiorthogonal decomposition
\begin{align}\label{SOD3}
\Db(\mod^{\ZZ}A^{\op})=\MM^{>-i}_{A^{\op}}\perp\MM^{\le-i}_{A^{\op}}.
\end{align}
Applying $(-)^*$, we obtain a semiorthogonal decomposition
\begin{align}\label{SOD4}
\Db(\mod^{\ZZ}A)=(\MM^{\le-i}_{A^{\op}})^*\perp(\MM^{>-i}_{A^{\op}})^*\stackrel{\eqref{Y^*}}{=}\SS^{\ge i}_A\perp(\MM^{>-i}_{A^{\op}})^*.
\end{align}
Since $\DD^{\ge i}_A\supset\SS^{\ge i}_A$, we can apply Lemma \ref{distributive}(2) to \eqref{SOD1} and \eqref{SOD4} to obtain \eqref{SOD0}.
The last assertion is clear from \eqref{SOD0} and
\[\Db(\qgr A)=\Db(\mod^{\ZZ}A)/(\SS^{\ge i}_A\perp\DD^{<i}_A).\qedhere\]
\end{proof}

Now we give a semiorthogonal decomposition which gives a direct connection between $\D_{\sg}^{\ZZ}(A)$ and $\Db(\qgr A)$. Let
\[\QQ^{<i}_A:=\PP^{<i}_A\Lotimes_AW
\ \mbox{ and }\ \QQ^{\ge i}_A:=\PP^{\ge i}_A\Lotimes_AW.\]
Then $(\QQ^{<i}_{A},\QQ^{\ge i}_{A})$ gives a semiorthogonal decomposition of $\Kb(\proj^{\ZZ}A)$.

\begin{thm}\label{sg vs qgr}
Assume that \textup{(B1)}, \textup{(B2)}, \textup{(B3)} and \textup{(B4)} hold. Let $i\in\ZZ$. \begin{enumerate}
\item If $W^{-1}\in\DD^{\ge0}_A$, then we have semiorthogonal decompositions:
\begin{align}\label{sg=qgr*-}
F_i(\D_{\sg}^{\ZZ}(A))&=(\DD^{\ge i}_A\cap\MM^{<i}_A)\perp G_i(\Db(\qgr A))\\ \label{sg=-*qgr}
&=(G_i(\Db(\qgr A))\Lotimes_AW^{-1})\perp(\DD^{\ge i}_A\cap\MM^{<i}_A).
\end{align}
\item If $W^{-1},W\in\DD^{\ge0}_A$, then we have
\begin{align*}
F_i(\D_{\sg}^{\ZZ}(A))=G_i(\Db(\qgr A)).
\end{align*}
\item If $W\in\DD^{\ge0}_A$, then we have semiorthogonal decompositions:
\begin{align}\label{qgr=-*sg}
G_i(\Db(\qgr A))&=(\PP^{\ge i}_A\cap\QQ^{<i}_A)\perp(F_i(\D_{\sg}^{\ZZ}(A))\otimes_AW)\\ \label{qgr=sg*-}
&=F_i(\D_{\sg}^{\ZZ}(A))\perp(\PP^{\ge i}_A\cap\QQ^{<i}_A).
\end{align}
\end{enumerate}
\end{thm}

\begin{proof}
Notice that $W^{-1}\in\DD^{\ge0}_A$ is equivalent to $\MM^{\ge0}_A\subset\DD^{\ge0}_A$, and $W\in\DD^{\ge0}_A$ is equivalent to $\MM^{\ge0}_A\supset\DD^{\ge0}_A$.

(1) By \eqref{X^*}, we have  $\MM^{<i}_A\subset(\DD^{>-i}_{A^{\op}})^*$. Applying Lemma \ref{distributive}(3) to $\Db(\mod^{\ZZ}A)=\MM^{\ge i}_A\perp\MM^{<i}_A$, we obtain \eqref{sg=-*qgr}:
\begin{align*}
\DD^{\ge i}_A\cap(\DD^{>-i}_{A^{\op}})^*=(\MM^{\ge i}_A\cap(\DD^{>-i}_{A^{\op}})^*)\perp(\DD^{\ge i}_A\cap\MM^{<i}_A)=(G_i(\Db(\qgr A))\Lotimes_AW^{-1})\perp(\DD^{\ge i}_A\cap\MM^{<i}_A).
\end{align*}
Replacing $A$ by $A^{\op}$ and $i$ by $1-i$, we obtain
\begin{align*}
\DD^{>-i}_{A^{\op}}\cap(\DD^{\ge i}_A)^*&=(\MM^{>-i}_{A^{\op}}\cap(\DD^{\ge i}_A)^*)\perp(\DD^{>-i}_{A^{\op}}\cap\MM^{\le-i}_{A^{\op}}).
\end{align*}
Applying $(-)^*$, we obtain \eqref{sg=qgr*-}:
\begin{align*}
\DD^{\ge i}_A\cap(\DD^{>-i}_{A^{\op}})^*=(\DD^{>-i}_{A^{\op}}\cap\MM^{\le-i}_{A^{\op}})^*\perp(\DD^{\ge i}_A\cap(\MM^{>-i}_{A^{\op}})^*)
\stackrel{\eqref{X^*}\eqref{Y^*}}{=}(\DD^{\ge i}_A\cap\MM^{<i}_A)\perp G_i(\Db(\qgr A)).
\end{align*}

(2) This is immediate from
\[F_i(\D_{\sg}^{\ZZ}(A))=\DD^{\ge i}_A\cap(\DD^{>-i}_{A^{\op}})^*=\DD^{\ge i}_A\cap(\MM^{>-i}_{A^{\op}})^*=G_i(\Db(\qgr A)).\]

(3) Let $\WW^{\ge i}_A:=\DD^{\ge i}_A\Lotimes_AW \subset\Db(\mod^{\ZZ}A)$.
Applying $-\otimes_AW$ to \eqref{SOD1'}, we obtain a semiorthogonal decomposition
\begin{align}\label{SOD5}
\Db(\mod^{\ZZ}A)=\QQ_A^{<i}\perp\WW_A^{\ge i}.
\end{align}
By our assumption, $\WW_A^{\ge i}\subset\DD^{\ge i}_A$. Also we have $\QQ^{<i}_A=(W^{-1}\Lotimes_A\PP^{>-i}_{A^{\op}})^*\subset(\MM^{>-i}_{A^{\op}})^*$. Applying Lemma \ref{distributive}(3) to \eqref{SOD5}, we obtain \eqref{qgr=-*sg}:
\begin{align*}
\DD^{\ge i}_A\cap(\MM^{>-i}_{A^{\op}})^*=(\DD^{\ge i}_A\cap\QQ^{<i}_A)\perp(\WW_A^{\ge i}\cap(\MM^{>-i}_{A^{\op}})^*)
=(\PP^{\ge i}_A\cap\QQ^{<i}_A)\perp(F_i(\D_{\sg}^{\ZZ}(A))\otimes_AW).
\end{align*}
Replacing $A$ by $A^{\op}$ and $i$ by $1-i$, we obtain
\begin{align*}
\DD^{>-i}_{A^{\op}}\cap(\MM^{\ge i}_A)^*&=(\PP^{>-i}_{A^{\op}}\cap\QQ^{\le-i}_{A^{\op}})\perp(\WW_{A^{\op}}^{>-i}\cap(\MM^{\ge i}_A)^*).
\end{align*}
Applying $W^{-1}\Lotimes_A-$, we obtain
\begin{align*}
\MM^{>-i}_{A^{\op}}\cap(\DD^{\ge i}_A)^*&=((W^{-1}\Lotimes_A\PP^{>-i}_{A^{\op}})\cap\PP^{\le-i}_{A^{\op}})\perp(\DD^{>-i}_{A^{\op}}\cap(\DD^{\ge i}_A)^*).
\end{align*}
Applying $(-)^*$, we obtain \eqref{qgr=sg*-}:
\begin{align*}
\DD^{\ge i}_A\cap(\MM^{>-i}_{A^{\op}})^*=(\DD^{\ge i}_A\cap(\DD^{>-i}_{A^{\op}})^*)\perp((\PP^{\le-i}_{A^{\op}})^*\cap(W^{-1}\Lotimes_A\PP^{>-i}_{A^{\op}})^*)=F_i(\D_{\sg}^{\ZZ}(A))\perp(\PP^{\ge i}_A\cap\QQ^{<i}_A).&\qedhere
\end{align*}
\end{proof}

In the rest of this subsection, we consider thick subcategories
\begin{align*}
{\mathscr C}_A &= \thick\{\mod^{\ZZ} _{0} A,\ \proj^{\ZZ} A \} \; \subset \Db(\mod^{\ZZ} A),\\
\D_{\sg, 0}^{\ZZ}(A)&= {\mathscr C}_A /\Kb(\proj^{\ZZ} A)\subset\D_{\sg}^{\ZZ}(A),\\
\per(\qgr A)&={\mathscr C}_A/\Db(\mod_0A)\subset\Db(\qgr A).
\end{align*}
The following is immediate from Theorems \ref{SOD of sg}, \ref{SOD of qgr}, and \ref{sg vs qgr}.

\begin{cor}\label{sg vs qgr 2}
Let $i\in\ZZ$.
\begin{enumerate}
\item Assume that the assumptions \textup{(B1)} and \textup{(B2)} hold. Then we have semiorthogonal decompositions
\begin{align*}
{\mathscr C}_A&=\PP^{<i}_A\perp(\DD^{\ge i}_A\cap(\DD^{>-i}_{A^{\op}})^*\cap{\mathscr C}_A)\perp\PP^{\ge i}_A.
\end{align*}
Thus we have a triangle equivalence
\begin{align*}
F_i:\D_{\sg,0}^{\ZZ}(A)&\simeq\DD^{\ge i}_A\cap(\DD^{>-i}_{A^{\op}})^*\cap{\mathscr C}_A\subset{\mathscr C}_A.
\end{align*}
\item Assume that the assumptions \textup{(B1)}, \textup{(B3)} and \textup{(B4)} hold. Then we have a semiorthogonal decomposition
\begin{align*}
{\mathscr C}_A=\SS^{\ge i}_A\perp(\DD^{\ge i}_A\cap(\MM^{>-i}_{A^{\op}})^*\cap{\mathscr C}_A)\perp\DD^{<i}_A.
\end{align*}
Thus we have a triangle equivalence
\[G_i:\per(\qgr A)\simeq\DD^{\ge i}_A\cap(\MM^{>-i}_{A^{\op}})^*\cap{\mathscr C}_A\subset{\mathscr C}_A.\]
\item Assume that the assumptions \textup{(B1)}, \textup{(B2)}, \textup{(B3)} and \textup{(B4)} hold.
\begin{enumerate}
\item If $W^{-1}\in\DD^{\ge0}_A$, then we have semiorthogonal decompositions:
\begin{align*}
F_i(\D_{\sg,0}^{\ZZ}(A))&=(\DD^{\ge i}_A\cap\MM^{<i}_A)\perp G_i(\per(\qgr A))\\
&=(G_i(\per(\qgr A))\Lotimes_AW^{-1})\perp(\DD^{\ge i}_A\cap\MM^{<i}_A).
\end{align*}
\item If $\MM^{\ge0}_A=\DD^{\ge0}_A$, then we have
\begin{align*}
F_i(\D_{\sg,0}^{\ZZ}(A))=G_i(\per(\qgr A)).
\end{align*}
\item If $W\in\DD^{\ge0}_A$, then we have semiorthogonal decompositions:
\begin{align*}
G_i(\per(\qgr A))&=(\PP^{\ge i}_A\cap\QQ^{<i}_A)\perp(F_i(\D_{\sg,0}^{\ZZ}(A))\otimes_AW)\\
&=F_i(\D_{\sg,0}^{\ZZ}(A))\perp(\PP^{\ge i}_A\cap\QQ^{<i}_A).
\end{align*}
\end{enumerate}
\end{enumerate}
\end{cor}

\section{Tilting theory via semiorthogonal decompositions}\label{section: tilting theory 2}

The aim of this subsection is to give another proof of Theorem \ref{a and tilting}(2) as an application of a semiorthogonal decomposition given in Theorem \ref{sg vs qgr}.
Throughout this section, let $A$ be a ring-indecomposable basic $\NN$-graded AS-Gorenstein algebra of dimension $1$ and Gorenstein parameter $(p_i)_{i\in\II_A}$.
Recall that $Q= \bigoplus_{i \in \ZZ} \Hom_{\qgr A}(A, A(i))$ is the total quotient ring, $q$ is a positive integer such that $\proj^{\ZZ}Q=\add\bigoplus_{i=1}^qQ(i)$ (see Theorem \ref{prop.Q4}(1)), and $a_s:=-p_s$ for each $s \in \II_A$.
For $M \in \mod^{\ZZ}A$, let $M_{<0}:=M/M_{\geq 0}$.
We consider objects
\begin{align*}
T:=\bigoplus_{i\ge0}\omega(-i)_{<0}[-1]\in(\mod^{\ZZ}A)[-1]\ \mbox{ and }\ U:=\bigoplus_{i=1}^{q} Q(i)_{\ge 0}\in\mod^{\ZZ}A,
\end{align*}
where the direct sum is clearly finite. We state our main result in the following form.

\begin{thm}\label{another proof}
Assume that \textup{(A1)}, \textup{(A2)} and \textup{(A3)} are satisfied.
\begin{enumerate}
\item $F_0(\D_{\sg,0}^{\ZZ}(A))\otimes_A\omega=\WW_A^{\ge0}\cap(\MM^{>0}_{A^{\op}})^*\cap{\mathscr C}_A$ has a tilting object $T\oplus U$.
\item $\D_{\sg,0}^{\ZZ}(A)$ has a tilting object $T\oplus U$ such that  $\add(T\oplus U)=\add V$ holds for $V$ in \eqref{define V}.
\end{enumerate}
\end{thm}

We start with the following elementary observation.

\begin{lem}\label{lem.inj}
$Q(i)_{\ge 0}$ is an injective object in $\mod^{\ge 0}A$ for any $i \in \ZZ$
\end{lem}

\begin{proof}
We have isomorphisms of functors on $\mod^{\ge0} A$:
\[\Hom^{\ZZ}_A(-,Q(i)_{\ge0}) \simeq
\Hom^{\ZZ}_A(-,Q(i)) \simeq
\Hom^{\ZZ}_Q(-\otimes_A Q, Q(i)).\]
This is an exact functor since $Q$ is a flat $A$-module by Proposition \ref{lem.sf2}(1) and $Q(i)$ is an injective object in $\mod^{\ZZ}Q$ by Proposition \ref{lem.injective}(2). Thus $Q(i)_{\ge0}$ is injective in $\mod^{\ge0}A$.
\end{proof}

The following assertion is a slight modification of Theorem \ref{prop.Q4}(4).

\begin{lem}\label{U is tilting}
$U$ is a tilting object in $G_0(\per(\qgr A))=\DD^{\ge0}_A\cap(\MM^{>0}_{A^{\op}})^*\cap{\mathscr C}_A$.
\end{lem}

\begin{proof}
We check that $U$ belongs to $G_0(\per(\qgr A))$.
Clearly $U\in \DD_A^{\geq 0}$ holds. For each $i,j\ge0$, we have 
\[\Hom^{\ZZ}_A(Q(i)_{\ge0},e_s\omega(-j))=\Hom^{\ZZ}_A(Q(i)_{\ge0},e_{\nu s}A(a_s-j))\stackrel{\text{Prop.\,\ref{j-i}(3)}}{=}0.\]
Thus $\omega\otimes_A(Q(i)_{\ge 0})^*=\Hom_A(Q(i)_{\ge0},\omega)\in\DD^{>0}_{A^{\op}}$ holds, and hence $(Q(i)_{\ge 0})^*\in\MM^{>0}_{A^{\op}}$. Thus $U\in G_0(\per(\qgr A))$.

Since $U$ is a tilting object in $\per(\qgr A)$ by Theorem \ref{prop.Q4}(4), it is also a tilting object in $G_0(\per(\qgr A))$.
\end{proof}

By our assumption (A3) and Corollary \ref{sg vs qgr 2}(3)(a), we have a semiorthogonal decomposition
\begin{align}\label{sg vs qgr apply}
F_0(\D_{\sg, 0}^{\ZZ}(A))\otimes_A\omega=G_0(\per(\qgr A))\perp(\WW_A^{\ge0}\cap\DD^{<0}_A),
\end{align}
which plays a key role in the proof of Theorem \ref{another proof}.
Another crucial step is the following.

\begin{prop}\label{W is tilting}
Let $\AA:=((\mod^{\ge0}A)\otimes_A\omega)\cap\mod^{<0}A$.
\begin{enumerate}
\item $T$ is a progenerator in $\AA$ .
\item $E:=\End^{\ZZ}_A(T)$ has finite global dimension.
\item $T$ is a tilting object in $\Db(\AA)=\WW_A^{\ge0}\cap\DD^{<0}_A$.
\end{enumerate}
\end{prop}

\begin{proof}
(1) We have $\omega(-i)=A(-i)\otimes_A\omega\in(\mod^{\ge0}A)\otimes_A\omega$ and hence its factor object $\omega(-i)_{<0}$ belongs to $((\mod^{\ge0}A)\otimes_A\omega)\cap\mod^{<0}A$.
Since $\add\{A(-i)\mid i\ge0\}$ is a progenerator of $\mod^{\ge0}A$, the subcategory $\add\{\omega(-i)\mid i\ge0\}$ is a progenerator of $(\mod^{\ge0}A)\otimes_A\omega$. Since we have an isomorphism of functors on $\mod^{<0}A$:
\[\Hom^{\ZZ}_A(\omega(-i)_{<0},-)\simeq\Hom^{\ZZ}_A(\omega(-i),-),\]
$\add\{\omega(-i)_{<0}\mid i\ge0\}$ is a progenerator of $((\mod^{\ge0}A)\otimes_A\omega)\cap\mod^{<0}A=\AA$, as desired.

(2) The proof is completely parallel to that of Lemma \ref{End U}. Instead of $W^1$ and $W^2$, we need to consider 
\[T^1:=\bigoplus_{s\in\II_A}\bigoplus_{i=0}^{a_s}e_s\omega(-i)\ \mbox{ and }\ T^2:=\bigoplus_{s\in\II_A}\bigoplus_{i=a_s+1}^{a}e_s\omega(-i).\]
Then the parallel argument works.

(3) By (1), we have an equivalence $H:\AA\simeq\mod E$, and $H(T)$ is a progenerator in $\mod E$. By (2), $H(T)$ is a tilting object in $\per E=\Db(\mod E)\simeq\Db(\AA)=\WW_A^{\ge0}\cap\DD^{<0}_A$, and the assertion follows.
\end{proof}

Now we are ready to prove Theorem \ref{another proof}.

\begin{proof}[Proof of Theorem \ref{another proof}]
(1) We use the semiorthogonal decomposition \eqref{sg vs qgr apply}.
By Lemma \ref{U is tilting} and Proposition \ref{W is tilting},  $G_0(\per(\qgr A))$ and $\WW_A^{\ge0}\cap\DD^{<0}_A$ have tilting objects
\[U=\bigoplus_{i=1}^{q} Q(i)_{\ge 0} \in G_0(\per(\qgr A))\ \mbox{ and }\ T=\bigoplus_{i\ge0}\omega(-i)_{<0}[-1]\in\WW_A^{\ge0}\cap\DD^{<0}_A\]
respectively. In particular, we have
\[\thick(T\oplus U)=F_0(\D_{\sg,0}^{\ZZ}(A))\otimes_A\omega.\]
Moreover, for each $\ell\in\ZZ$, we have
\begin{align*}
&\Hom_{\Db(\mod^{\ZZ}A)}(U,U[\ell])=0=\Hom_{\Db(\mod^{\ZZ}A)}(T,T[\ell])\ \mbox{ if $\ell\neq0$,}\\
&\Hom_{\Db(\mod^{\ZZ}A)}(U,T[\ell])=0.
\end{align*}
It remains to check
\[\Hom_{\Db(\mod^{\ZZ}A)}(\omega(-i)_{<0}[-1], Q(j)_{\ge0}[\ell])=0\]
for each $i,j$ and $\ell\neq0$. If $\ell<-1$, then this is clear since $\omega(-i)_{<0}$ and $Q(j)_{\ge0}$ are modules.
If $\ell=-1$, then this is also clear since $\omega(-i)_{<0}\in\mod^{<0}A$ and $Q(j)_{\ge0}\in \mod^{\ge0}A$.
If $\ell>0$, then
\begin{align*}
\Hom_{\Db(\mod^{\ZZ}A)}( \omega(-i)_{<0}[-1], Q(i)_{\ge 0}[\ell]) &\simeq \Ext^{\ell+1}_{\mod^{\ZZ}A}(\omega(-i)_{<0}, Q(j)_{\ge 0})\\
&\simeq\Ext^{\ell}_{\mod^{\ZZ}A}(\omega(-i)_{\ge0}, Q(j)_{\ge 0})\stackrel{\text{Lem.\,\ref{lem.inj}}}{=}0.
\end{align*}
Thus the assertion follows.

(2) By (1), $T\oplus U$ is a tilting object in $\D_{\sg,0}^{\ZZ}(A)$.
It remains to show $T\oplus U\simeq V$ in $\D_{\sg,0}^{\ZZ}(A)$. For each $s\in\II_A$, in $\D_{\sg,0}^{\ZZ}(A)$, we have
\begin{align*}
&T=\bigoplus_{i\ge0}e_s\omega(-i)_{<0}[-1]\simeq\bigoplus_{i\ge0}e_s\omega(-i)_{\ge0}\simeq\bigoplus_{i\ge0}e_{\nu s}A(a_s-i)_{\ge0}\simeq V^1,\\
&\add U=\add\bigoplus_{i=1}^qQ(i)_{\ge0}\stackrel{{\rm Lem.\,\ref{induction sequence}(3)}}{=}\add V^2.
\end{align*}
Thus we have $\add(T\oplus U)=\add V$.
\end{proof}

\section{Gorenstein tiled orders}

\subsection{Tilting theory for Gorenstein tiled orders}

We recall the definition of tiled orders \cite{ZK}.
We refer to \cite{KKMPZ} for basic properties of Gorenstein tiled orders (see also \cite{Kirichenko}).
We fix an integer $n\geq 1$ and let $\mathbb{I}=\{1, 2, \dots, n\}$.

\begin{dfn}\label{dfn-tiled-order}
Let $R=k[x]$ be the polynomial ring in one variable over a field $k$ with $\deg x=1$. Then $R$ has a unique graded maximal ideal $Rx$, and we denote by $K=k[x,x^{-1}]$ the Laurent polynomial ring.
A \emph{tiled order} over $R$ is an $R$-subalgebra of $Q:=\rM_n(K)$ of the form
\begin{align*}
A=\begin{bmatrix}
Rx^{m(1,1)} & Rx^{m(1,2)} & \cdots & Rx^{m(1,n)} \\
Rx^{m(2,1)} & Rx^{m(2,2)} & \cdots & Rx^{m(2,n)} \\
\vdots & \vdots & \ddots & \vdots \\
Rx^{m(n,1)} & Rx^{m(n,2)} & \cdots & Rx^{m(n,n)}
\end{bmatrix}
\subset Q=\rM_n(K)
\end{align*}
for $m(i,j)\in\ZZ$, where $m(i,i)=0$ for each $1\le i\le n$.
We regard $Q$ as a $\ZZ$-graded algebra by $Q_i=\rM_n(k)x^i$, and $A$ as a $\ZZ$-graded algebra with respect to the induced $\ZZ$-grading.
\end{dfn}

Under the setting above, $Q=A\otimes_R K$ is a graded total quotient ring of $A$.
Notice that $A$ is an $R$-subalgebra of $Q$ if and only if
\begin{align}\label{eq dfn order}
m(i,k)+m(k,j)\geq m(i,j)
\end{align}
holds for any $i,j,k\in \mathbb{I}$.
Clearly $A$ can be recovered from the integer matrix
\[\mathrm{v}(A):=(m(i,j))\in \rM_n(\ZZ).\]

Let $E_{ij}$ be a matrix unit, $e_i=E_{ii}$ and $S_i=\top e_i A$.
We can see that $A$ is basic if and only if $m(i,j) + m(j, i) > 0$ holds for any $i\neq j$. A basic tiled order $A$ is a Gorenstein order if and only if there exists a (unique) permutation $\nu : \II \to \II$ and $(\ell_i)\in\ZZ^{\mathbb{I}}$ such that
\begin{align}\label{R dual for tiled order}
\Hom_R(Ae_i,R)\simeq e_{\nu i}A(\ell_i)
\end{align}
in $\mod^{\ZZ}A$ holds for each $i\in\mathbb{I}$. The isomorphism \eqref{R dual for tiled order} is equivalent to the equality
\begin{equation}\label{gorenstein tiled}
m(\nu i, j)+m(j, i)=\ell_i\ \mbox{ for each }\ j\in\mathbb{I}.
\end{equation}
Evaluating $j=i$, we have $\ell_i=m(\nu i,i)$. In this case $A$ is an AS-Gorenstein algebra of dimension $1$ by Proposition \ref{lem-Gorder-ASG}(1).
Since $\omega_R=R(-1)$, by applying $(-1)$ to \eqref{R dual for tiled order}, we obtain
\[\Hom_R(Ae_i,\omega_R)\simeq e_{\nu i}A(m(\nu i,i)-1).\]
By this isomorphism and Proposition \ref{lem-Gorder-ASG}(2), we have that $\nu$ is the Nakayama permutation and that the Gorenstein parameter $p_i$ of $S_i$ satisfies
\begin{align}\label{gto a inv}
p_i = 1 - m(\nu i, i).
\end{align}
Since the category $\mod^{\ZZ}Q$ is semisimple, we have $\CM^{\mathbb{Z}}_0A = \CM^{\mathbb{Z}}A$ and $\D_{\sg,0}^{\ZZ}(A) = \D_{\sg}^{\ZZ}(A)$ by Corollary \ref{a and tilting 2}(1).

The \emph{rank} of $M\in\CM^{\mathbb{Z}}A$ is the length of the object $M\otimes_RK$ in $\mod^{\ZZ}Q$. It is elementary that the rank of $M$ is one if and only if there exists $v=(v_1,\ldots,v_n)\in\ZZ^n$ such that
\[M\simeq\rL(v):=[Rx^{v_1}\, \cdots\, Rx^{v_n}]\ \mbox{ in }\ \CM^{\ZZ}A.\]
In this case, we call
\[\mathrm{v}(M):=v=(v_1, \dots, v_n)\in\ZZ^\II\]
the \emph{exponent vector} of $M$.
Thus each $M\in\CM^{\ZZ}A$ with rank one is uniquely determined by its exponent vector $\mathrm{v}(M)$.
For instance, for each $i\in\II$, we have
\[\mathrm{v}(e_iA)=(m(i,1),\ldots,m(i,n)).\]
For each $M\in\CM^{\mathbb{Z}}A$ with rank one and $j\in\ZZ$, we have
\begin{align*}
\mathrm{v}(M(j)) = \mathrm{v}(M)-j\mathbf{1}\ \mbox{ and }\ \mathrm{v}(M_{\ge0})=\max\{\mathrm{v}(M),\bm{0}\},
\end{align*}
where $\bm{1}:=(1,\ldots,1)$ and $\max\{v,w\}=(\max\{v_1,w_1\},\ldots,\max\{v_n,w_n\})$.
In particular, $\mathrm{v}(M(j)_{\geq 0})=\max\{\mathrm{v}(M)-j\bm{1},\bm{0}\}$.

In the rest of this subsection, let $A$ be a basic $\NN$-graded Gorenstein tiled order with $\mathrm{v}(A)=(m(i, j))\in \rM_n(\NN)$.
Then we have $Q(1) \simeq Q$ in $\mod^{\ZZ}Q$, and we can choose $q=1$ in Theorem \ref{prop.Q4}(1).
Assume that $p_i\leq 0$ for any $i\in\II$.
By Theorem \ref{a and tilting}, $\D_{\sg}^{\ZZ}(A)$ admits a tilting object $V=\bigoplus_{i\in\II}\bigoplus_{j\geq 1}^{1-p_i} e_{\nu i}A(j)_{\geq 0}$.
To give a description of $\End_{\D_{\sg}^{\ZZ}(A)}(V)$, we consider a poset $(\ZZ^\mathbb{I},\le)$, where for $v, w\in\ZZ^\mathbb{I}$, we write $v\le w$ if $v_i \leq w_i$ for any $i\in \mathbb{I}$. Then the subposet
\begin{align}\label{define V_A}
\mathbb{V}_A &:=\{\mathrm{v}(e_i A(j)_{\geq 0}) \mid i\in\mathbb{I},\ 1\le j\}\subset\ZZ^\II
\end{align}
plays a key role.
Notice that $e_iA(j)_{\ge0}=[R\, \cdots\, R]$ and $\mathrm{v}(e_iA(j)_{\ge0})=\bm{0}$ hold for $j\gg0$.

Let $(P, \le)$ a finite poset.
We denote by $P^{\op}$ the opposite poset of $P$ and $[x, y]:=\{ z \in P \mid x \le z \le y\}$.
The \emph{incidence algebra} $kP$ of $(P, \le)$ is a $k$-algebra whose underlying $k$-vector space is $kP=\bigoplus_{x\le y}k[x, y]$ with a product $[x, y]\cdot [x', y']:=\delta_{y, x'}[x, y']$, where $\delta$ is the Kronecker delta. Clearly the Gabriel quiver of $kP$ is the Hasse quiver of $(P, \le)$ which is acyclic. Thus the global dimension of $kP$ is finite.

The main result of this section is the following.
\begin{thm}\label{thm-GTO-Inc}
Let $A$ be a basic $\NN$-graded Gorenstein tiled order given in Definition \ref{dfn-tiled-order} such that $p_i\leq 0$ for any $i\in\mathbb{I}$, and
$(\mathbb{V}_A, \le)$ the poset as above.
Then the following statements hold.
\begin{enumerate}
\item $V=\bigoplus_{i\in\II}\bigoplus_{j=1}^{1-p_i}e_{\nu i}A(j)_{\ge0}$ is a tilting object in $\D_{\sg}^{\ZZ}(A)$.
\item $|V|=1-\sum_{i\in\II}p_i$ holds.
\item $\End_{\D_{\sg}^{\ZZ}(A)}(V)$ is Morita equivalent to $k\mathbb{V}_A^{\op}$. In particular, the global dimension of $\End_{\D_{\sg}^{\ZZ}(A)}(V)$ is finite, and we have a triangle equivalence
\[\D_{\sg}^{\ZZ}(A)\simeq\per k\mathbb{V}_A^{\op} \simeq \Db(\mod k\mathbb{V}_A^{\op}).\]
\end{enumerate}
\end{thm}

To prove this, we give preparations.
The following result induces the part (1) of the theorem.

\begin{prop}\label{prop-basic-V}
Let $A$ be a basic Gorenstein tiled order given in Definition \ref{dfn-tiled-order} with $\mathrm{v}(A)=(m(i,j))\in\rM_n(\ZZ)$ and the Nakayama permutation $\nu$.
\begin{enumerate}
\item For each $i\in\II$, $\max\{ m(\nu i, j) \mid j\in\II\}=m(\nu i, i)=1-p_i$ holds.
\item For $i, j\in\II$ and $1\leq \ell, \ell' \leq \min\{-p_{i}, -p_{j}\}$, the following conditions are equivalent.
	\begin{enumerate}
	\item[(i)] $\mathrm{v}(e_{\nu i}A(\ell)_{\geq 0}) = \mathrm{v}(e_{\nu j}A(\ell')_{\geq 0})$
	\item[(ii)] $e_{\nu i}A(\ell)_{\geq 0} \simeq e_{\nu j}A(\ell')_{\geq 0}$ in $\mod^{\ZZ}A$
	\item[(iii)] $(i, \ell)=(j,\ell')$
	\end{enumerate}
\end{enumerate}
\end{prop}

\begin{proof}
(1)
By Proposition \ref{lem-nu-eAe2}(3), we have
\[p_i=-\max\{\ell\in\ZZ\mid (e_{\nu i}(\rM_n(K)/A))_\ell\neq0\} = -\max\{m(\nu i, j)-1 \mid j\in\II\}.\]
So the assertion holds by \eqref{gto a inv}.

(2) (i)$\Leftrightarrow$(ii) is clear, and (ii)$\Leftrightarrow$(iii) was shown in Proposition \ref{j-i}(4).
\end{proof}

To prove (2) of the theorem, we need more preparation.

\begin{prop}\label{iso with incidence}
Let $A$ be a Gorenstein tiled order given in Definition \ref{dfn-tiled-order}, $\mathbb{V}$ a finite subset of $\ZZ^n$ such that $\rL(v)\in\CM^{\ZZ}A$ holds for each $v\in\mathbb{V}$
and $V:=\bigoplus_{v\in\mathbb{V}}\rL(v)$.
Then we have an isomorphism of $k$-algebras
\[\End_A^{\ZZ}(V)\simeq k\mathbb{V}^{\op}.\]
\end{prop}

\begin{proof}
(i) Let $v=(v_1,\ldots,v_n),w=(w_1,\ldots,w_n)\in\mathbb{V}$, $M:=\rL(v)$ and $N:=\rL(w)$. We prove
\begin{align*}
\Hom_A^{\ZZ}(M,N)\simeq\left\{\begin{array}{ll}
k & v \ge w \\
0 & \mbox{else}.
\end{array}.\right.
\end{align*}
By definition, both $M$ and $N$ are submodules of a simple $Q$-module $S:=[K\,\cdots\, K]$.
Thus we have an identification $\Hom_A(M,N) = \{ f \in \End_{Q}(S) \mid f(M) \subseteq N\}$.
Then the isomorphism $K\simeq\End_{Q}(S)$; $\alpha\mapsto(\alpha\cdot)$ gives an isomorphism
\begin{align*}
\Hom_A(M, N) \simeq \{ \alpha\in K \mid \alpha M\subseteq N\}
=\{\alpha\in K\mid \forall i\in[1,n],\ \alpha(Rx^{v_i})\subseteq Rx^{w_i}\}=Rx^\ell
\end{align*}
for $\ell := \max\{w_i - v_i \mid i\in\mathbb{I}\}$.
By taking the degree $0$ part, we have that $\Hom_A^{\ZZ}(M, N)\neq 0$ if and only if $\ell\leq 0$ if and only if $v\geq w$.
In this case, $\Hom_A^{\ZZ}(M, N) = (Rx^{\ell})_0 = k$ holds.

(ii) The assertion (i) gives an isomorphism $\End^{\ZZ}_A(V)\simeq k\mathbb{V}^{\op}$ of $k$-vector spaces. It suffices to show that this commutes with the multiplications.
For each $u,v,w\in\mathbb{V}$ satisfying $u \geq v \geq w$, the diagram
\[
\begin{tikzcd}[ampersand replacement=\&]
\Hom_A^{\ZZ}(\rL(v),\rL(w)) \times \Hom_A^{\ZZ}(\rL(u),\rL(v)) \arrow[rr, "\text{comp.}"]  \arrow[d, "\wr"] \& \& \Hom_A^{\ZZ}(\rL(u),\rL(w)) \arrow[d, "\wr"] \\
k \times k \arrow[rr, "\text{mult.}"] \& \& k
\end{tikzcd}
\]
commutes. Thus the assertion follows.
\end{proof}

We are ready to prove Theorem \ref{thm-GTO-Inc}. 

\begin{proof}[Proof of Theorem \ref{thm-GTO-Inc}]
The statement (1) directly follows from Theorem \ref{a and tilting} since $q=1$.
We show (2) and (3).
By Proposition \ref{prop-basic-V}(1), $e_{\nu i}A(j)_{\geq 0} \simeq \rL(\bm{0})$ holds for any $i\in\II$ and any $j\geq 1-p_i$.
Thus, $\mathbb{V}_A$ is as follows, and let $T$ be the following $A$-module:
\begin{align*}
\mathbb{V}_A =\{\bm{0}\}\sqcup\bigsqcup_{i\in\II}\bigl\{\mathrm{v}(e_{\nu i} A(j)_{\geq 0}) \mid 1\leq j \leq -p_{i}\bigr\}, \qquad
T:=\rL(\bm{0})\oplus\bigoplus_{i\in\II} \bigoplus_{j=1}^{-p_{i}} e_{\nu i}A(j)_{\geq 0}.
\end{align*}
Then $T$ is basic by Proposition \ref{prop-basic-V}(2).
Therefore we have $|V|=|T| = |\mathbb{V}_A|=1-\sum_{i\in\II} p_i$.

Since $\add V = \add T$, $\underline{\End}^{\ZZ}_A(V)$ is Morita equivalent to $\underline{\End}^{\ZZ}_A(T)$.
By Proposition \ref{j-i}(5), we have $\underline{\End}^{\ZZ}_A(T)=\End^{\ZZ}_A(T)$.
By Proposition \ref{iso with incidence}, we have an isomorphism $\End^{\ZZ}_A(T)\simeq k\mathbb{V}_A^{\op}$.
Thus the assertions follow.
\end{proof}

\subsection{Cyclic Gorenstein tiled orders}

In this subsection, we give a family of Gorenstein tiled orders with cyclic Nakayama permutations, and describe the endomorphism algebras of the tilting objects given by Theorem \ref{thm-GTO-Inc}.

We start with non-negative integers $m_1,\ldots,m_n$ such that $\sum_{i=1}^nm_i\ge1$.
Let 
\begin{empheq}[left={m(i, j)=\empheqlbrace}]{alignat=2}
&\sum_{k=i}^{j-1}m_k &\quad&~i<j \\
&\sum_{k=i}^{n}m_k + \sum_{k=1}^{j-1}m_k&\quad&~i>j.  
\end{empheq}
Then $m(i,j)$ satisfy \eqref{eq dfn order}, and therefore $A=(Rx^{m(i,j)})\subset \rM_n(K)$ is a basic tiled order.
Moreover, $A$ is a Gorenstein tiled order with the Nakayama permutation $\nu=(1~2~\dots~n)$.
In fact, one can check that $m(i, j)$ satisfy \eqref{gorenstein tiled} for $\ell_i:=\sum_{k\in\II}m_k-m_i$.
The Gorenstein parameter of $A$ is given by
\begin{align}\label{p_i here}
p_i \stackrel{\eqref{gto a inv}}{=} 1-m(i+1, i)=  1+m_i-\sum_{k\in\II}m_k.
\end{align}
We refer to Example \ref{ex-gto-cyclic} below for an explicit form of $A$ in the case $n=4$.

Assume that $p_i\leq 0$ holds for each $i\in\II$.
Thanks to Theorem \ref{thm-GTO-Inc}, there exists a tilting object $V$ in $\D_{\sg}^{\ZZ}(A)$.
We consider the endomorphism algebra of $V$.

Let $(i, j):=\mathrm{v}(e_iA(j)_{\geq 0})$.
Then the Hasse quiver of $(\mathbb{V}_A, \leq)^{\op}$ is given as follows:
\begin{enumerate}
\item The set of vertices is $\bigsqcup_{i\in\mathbb{I}}\{ (i+1, j) \mid 1\leq j \leq -p_{i} \} \sqcup\{\bm{0}\}$.
\item We draw an arrow from $(i, j)$ to $(k, \ell)$ if one of the following conditions hold.
	\begin{enumerate}
	\item $k=i$ and $\ell=j+1$.
	\item $i=k+1$, $1 \leq j \leq -p_{k-1}-m_k$ and $\ell = j + m_k$.
	\item $j=-p_{i-1}$ and $(k, \ell)=\bm{0}$.
	\end{enumerate}
\end{enumerate}
The arrows between vertices of $\{ (k+1, j) \mid 1\leq j \leq -p_{k} \} \sqcup \{ (k, j) \mid 1\leq j \leq -p_{k-1} \} \sqcup \{\bm{0}\}$ look like the following:
\[
\begin{tikzcd}
(k+1, 1) \arrow[r] \arrow[rrd] & \cdots \arrow[r] & (k+1, -p_{k-1}-m_k) \arrow[r] \arrow[rrd] 
\arrow[d, "\dots\dots",  phantom, description] & \cdots \arrow[r] & (k+1, -p_{k}) \arrow[rd] \\
(k, 1) \arrow[r] & \cdots \arrow[r] & (k, 1+m_k) \arrow[r] & \cdots \arrow[r] & (k, -p_{k-1}) \arrow[r] & \bm{0}
\end{tikzcd}.
\]
The horizontal arrows are of type (a), the two arrows going to $\bm{0}$ are of type (c), the others are of type (b).

\begin{ex}\label{ex-gto-cyclic}
Let $n=4$ and put $a=m_1$, $b=m_2$, $c=m_3$ and $d=m_4$. Then $\mathrm{v}(A) = (m(i, j))\in\rM_4(\ZZ)$ is as follows.
\[
\begin{bmatrix}
0 & a & a+b & a+b+c  \\
b+c+d & 0 & b & b+c  \\
c+d & c+d+a & 0  & c \\
d& d+a & d+a+b & 0 
\end{bmatrix}
\]
Assume that $a,b,c,d\geq 0$.
Moreover assume that all Gorenstein parameters are non-positive, that is, $a+b+c, b+c+d, c+d+a, d+a+b \geq 1$ by \eqref{p_i here}.
Then we write the Hasse quiver of $(\mathbb{V}_A, \leq)^{\op}$ in Figure \ref{V for n=4}.
In the picture, each number represents the number of vertices in each indicated area.
The north line (respectively, east line, south line, and west line) presents the vertices of the form $\mathrm{v}(e_1A(j)_{\geq 0})$ (respectively, $\mathrm{v}(e_2A(j)_{\geq 0})$, $\mathrm{v}(e_3A(j)_{\geq 0})$ and $\mathrm{v}(e_4A(j)_{\geq 0})$).

\begin{figure}[h]
\[
\scalebox{0.85}[0.85]{
\begin{tikzpicture}[descr/.style={fill=white},text height=1.5ex, text depth=0.25ex, scale=1]
\node(0)at(0,0){$\bm{0}$};

\node(11)at(0,7.5){$\bullet$};
\node(1s)at(0,6){$\bullet$};
\node(1s1)at(0,5){$\bullet$};
\node(1sr1)at(0,3.5){$\bullet$};
\node(1sr)at(0,2.5){$\bullet$};
\node(1srq1)at(0,1){$\bullet$};
\node(1dot)at(-2.7,2.5){$\cdots$};
\node(1dot)at(-2.2,2.5){$\cdots$};
\node(1dot)at(-0.9,2.5){$\cdots$};
\path[->, thick, font=\scriptsize ,>=angle 45] (11) edge node[descr]{$\vdots$} (1s);
\path[->, thick, font=\scriptsize ,>=angle 45] (1s) edge (1s1);
\path[->, thick, font=\scriptsize ,>=angle 45] (1s1) edge node[descr]{$\vdots$} (1sr1);
\path[->, thick, font=\scriptsize ,>=angle 45] (1sr1) edge (1sr);
\path[->, thick, font=\scriptsize ,>=angle 45] (1sr) edge node[descr]{$\vdots$} (1srq1);
\path[->, thick, font=\scriptsize ,>=angle 45] (1srq1) edge (0);
\draw[thick,decorate,decoration={brace,raise=5pt}] (11.north) -- (1s.south) node[midway,right=10pt] {$a$};
\draw[thick,decorate,decoration={brace,raise=5pt}] (1s1.north) -- (1sr1.south) node[pos=0.7,right=10pt] {$b-1$};
\draw[thick,decorate,decoration={brace,raise=5pt}] (1sr.north) -- (1srq1.south) node[midway,right=10pt] {$c$};

\node(21)at(-7.5,0){$\bullet$};
\node(2p)at(-6,0){$\bullet$};
\node(2p1)at(-5,0){$\bullet$};
\node(2ps1)at(-3.5,0){$\bullet$};
\node(2ps)at(-2.5,0){$\bullet$};
\node(2psr1)at(-1,0){$\bullet$};
\node(2dot)at(-2.6,-2){$\rotatebox{45}{$\cdots$}$};
\node(2dot)at(-3,-2.4){$\rotatebox{45}{$\cdots$}$};
\node(2dot)at(-1.8,-1.2){$\rotatebox{45}{$\cdots$}$};
\path[->, thick, font=\scriptsize ,>=angle 45] (21) edge node[descr]{$\cdots$} (2p);
\path[->, thick, font=\scriptsize ,>=angle 45] (2p) edge (2p1);
\path[->, thick, font=\scriptsize ,>=angle 45] (2p1) edge node[descr]{$\cdots$} (2ps1);
\path[->, thick, font=\scriptsize ,>=angle 45] (2ps1) edge (2ps);
\path[->, thick, font=\scriptsize ,>=angle 45] (2ps) edge node[descr]{$\cdots$} (2psr1);
\path[->, thick, font=\scriptsize ,>=angle 45] (2psr1) edge (0);
\draw[thick,decorate,decoration={brace,raise=5pt}] (21.west) -- (2p.east) node[midway,above=10pt] {$d$};
\draw[thick,decorate,decoration={brace,raise=5pt}] (2p1.west) -- (2ps1.east) node[pos=0.7,above=10pt] {$a-1$};
\draw[thick,decorate,decoration={brace,raise=5pt}] (2ps.west) -- (2psr1.east) node[midway,above=10pt] {$b$};

\node(31)at(0,-7.5){$\bullet$};
\node(3q)at(0,-6){$\bullet$};
\node(3q1)at(0,-5){$\bullet$};
\node(3qp1)at(0,-3.5){$\bullet$};
\node(3qp)at(0,-2.5){$\bullet$};
\node(3qps1)at(0,-1){$\bullet$};
\node(3dot)at(2.7,-2.5){$\dots$};
\node(3dot)at(2.2,-2.5){$\dots$};
\node(3dot)at(0.9,-2.5){$\dots$};
\path[->, thick, font=\scriptsize ,>=angle 45] (31) edge node[descr]{$\vdots$} (3q);
\path[->, thick, font=\scriptsize ,>=angle 45] (3q) edge (3q1);
\path[->, thick, font=\scriptsize ,>=angle 45] (3q1) edge node[descr]{$\vdots$} (3qp1);
\path[->, thick, font=\scriptsize ,>=angle 45] (3qp1) edge (3qp);
\path[->, thick, font=\scriptsize ,>=angle 45] (3qp) edge node[descr]{$\vdots$} (3qps1);
\path[->, thick, font=\scriptsize ,>=angle 45] (3qps1) edge (0);
\draw[thick,decorate,decoration={brace,raise=5pt}] (31.south) -- (3q.north) node[midway,left=10pt] {$c$};
\draw[thick,decorate,decoration={brace,raise=5pt}] (3q1.south) -- (3qp1.north) node[pos=0.7,left=10pt] {$d-1$};
\draw[thick,decorate,decoration={brace,raise=5pt}] (3qp.south) -- (3qps1.north) node[midway,left=10pt] {$a$};

\node(41)at(7.5,0){$\bullet$};
\node(4r)at(6,0){$\bullet$};
\node(4r1)at(5,0){$\bullet$};
\node(4rq1)at(3.5,0){$\bullet$};
\node(4rq)at(2.5,0){$\bullet$};
\node(4rqp1)at(1,0){$\bullet$};
\node(4dot)at(3,2.4){$\rotatebox{45}{$\cdots$}$};
\node(4dot)at(2.6,2){$\rotatebox{45}{$\cdots$}$};
\node(4dot)at(1.8,1){$\rotatebox{45}{$\cdots$}$};
\path[->, thick, font=\scriptsize ,>=angle 45] (41) edge node[descr]{$\cdots$} (4r);
\path[->, thick, font=\scriptsize ,>=angle 45] (4r) edge (4r1);
\path[->, thick, font=\scriptsize ,>=angle 45] (4r1) edge node[descr]{$\cdots$} (4rq1);
\path[->, thick, font=\scriptsize ,>=angle 45] (4rq1) edge (4rq);
\path[->, thick, font=\scriptsize ,>=angle 45] (4rq) edge node[descr]{$\cdots$} (4rqp1);
\path[->, thick, font=\scriptsize ,>=angle 45] (4rqp1) edge (0);
\draw[thick,decorate,decoration={brace,raise=5pt}] (41.east) -- (4r.west) node[midway,below=10pt] {$b$};
\draw[thick,decorate,decoration={brace,raise=5pt}] (4r1.east) -- (4rq1.west) node[pos=0.7,below=10pt] {$c-1$};
\draw[thick,decorate,decoration={brace,raise=5pt}] (4rq.east) -- (4rqp1.west) node[midway,below=10pt] {$d$};

\path[->, thick, font=\scriptsize ,>=angle 45] (11) edge (2p1);
\path[->, thick, font=\scriptsize ,>=angle 45] (1s) edge (2ps);
\path[->, thick, font=\scriptsize ,>=angle 45] (1sr1) edge (2psr1);

\path[->, thick, font=\scriptsize ,>=angle 45] (21) edge (3q1);
\path[->, thick, font=\scriptsize ,>=angle 45] (2p) edge (3qp);
\path[->, thick, font=\scriptsize ,>=angle 45] (2ps1) edge (3qps1);

\path[->, thick, font=\scriptsize ,>=angle 45] (31) edge (4r1);
\path[->, thick, font=\scriptsize ,>=angle 45] (3q) edge (4rq);
\path[->, thick, font=\scriptsize ,>=angle 45] (3qp1) edge (4rqp1);

\path[->, thick, font=\scriptsize ,>=angle 45] (41) edge (1s1);
\path[->, thick, font=\scriptsize ,>=angle 45] (4r) edge (1sr);
\path[->, thick, font=\scriptsize ,>=angle 45] (4rq1) edge (1srq1);
\end{tikzpicture}
}
\]
\caption{The endomorphism algebra of $V$ for $n=4$}
\label{V for n=4}
\end{figure}

\end{ex}

\section{Noncommutative quadric hypersurfaces}\label{sec.nqh}

We discuss an application of our main result (Theorem \ref{a and tilting}) to the study of noncommutative quadric hypersurfaces. Let $B$ be a noncommutative quadric hypersurface (see Definition \ref{dfn.qh}).  
Smith-Van den Bergh \cite{SV} (and Mori-Ueyama \cite{MU3} in a more general setting) proved that $\uCM^{\ZZ} B$ has a tilting object using
the method originally developed by Buchweitz-Eisenbud-Herzog \cite{BEH}. In this subsection, we prove that if $\qgr B$ has finite global dimension, then $\Db(\qgr B)$ has a tilting object. 
The key point is that the opposite algebra of the Koszul dual of $B$ is an AS-Gorenstein algebra of dimension $1$.

Throughout this section, let $k$ be an algebraically closed field of characteristic zero.

\subsection{Preliminaries on noncommutative quadric hypersurfaces}

A connected $\NN$-graded algebra $A$ is called \emph{quadratic} if it is isomorphic to the quotient algebra $T(V)/(R)$, where $T(V)=\bigoplus_{i\in\NN} V^{\otimes i}$ is the tensor algebra on a finite-dimensional vector space $V$, and $R$ is a subspace of $T(V)_2=V\otimes_k V$. For a quadratic algebra $A=T(V)/(R)$, the \emph{quadratic dual} $A^!$ of $A$ is defined by 
$T(V^*)/(R^\perp)$, where $V^*$ is the $k$-linear dual of $V$, and $R^\perp$ is the subspace of $T(V^*)_2=V^*\otimes_k V^*$ consisting of elements which are orthogonal to any element of $R$.

Let $A$ be a connected $\NN$-graded algebra.
A graded module $M \in \Mod^{\ZZ}A$ has a \emph{linear free resolution} if the $i$-th
term in its minimal free resolution is a direct sum of copies of $A(-i)$ for each $i$ or, equivalently, if $\Ext^i_A(M, k)_j = 0$ for $i+j=0$. The full subcategory of $\mod^{\ZZ}A$ consisting of modules having a linear free resolution is denoted by $\lin A$. 
We say that  $A$ is \emph{Koszul} if $k=A/A_{\geq 1} \in \lin A$.

Let $B$ be a Koszul algebra. Then it is well-known that $B$ is a quadratic algebra, $B^!$ is also Koszul, and $B^!$ is isomorphic to the Yoneda algebra $\bigoplus_{i\in\NN}\Ext_B^i(k,k)$
(in this case, $B^!$ is also called the \emph{Koszul dual} of $B$). 
Let $A=({B^!})^{\op}$. Then there exists a duality
\[
E_B:= \bigoplus_{i\in\NN}\Ext_B^i(-,k): \lin B \to \lin A.
\]
\begin{lem}\label{lem.ED}{\cite[Lemma 3.6]{MU1}}
	Let $B$ be a Koszul algebra and let $A=({B^!})^{\op}$. If $M\in \lin B$, then $E_B(\Omega ^iM(i))\simeq E_B(M)(i)_{\geq 0}$ in $\Mod^{\ZZ} A$ and $E_B(M(i)_{\geq 0})\simeq \Omega^iE_B(M)(i)$ in $\Mod^{\ZZ} A$ for all $i\in \NN$.  
\end{lem} 

The next proposition is necessary for the proof of the main theorem of this section.

\begin{prop} \label{prop.Kd}
	If $B$ and $A=({B^!})^{\op}$ are both connected Koszul AS-Gorenstein algebras, then we have a duality
	\[F: \Db(\qgr B) \to \uCM^{\ZZ}A\]
	such that $F(\Omega ^ik(i)) \simeq A(i)_{\geq 0}$ for any $i\in \NN$.
\end{prop}

\begin{proof}
	The duality $E_B: \lin B \to \lin A$ extends to a duality $\overline{E}_B: \Db(\mod^{\ZZ} B) \to \Db(\mod^{\ZZ} A)$
	by \cite[Proposition 4.5]{Mo1}.
	Furthermore,  $\overline{E}_B$ induces a duality $F: \Db(\qgr B) \to \uCM^{\ZZ} A$ by \cite[Theorem 5.3 and Lemma 5.1]{Mo2}.
	The last isomorphism follows from Lemma \ref{lem.ED}.
\end{proof}

Recall that an element $f$ of a ring $S$ is called \emph{normal} if $Sf=fS$, and is called \emph{regular} is the multiplication maps $\cdot f:S\to S$ and $f\cdot:S\to S$ are injective.

Let $S=T(V)/(R)$ be a Koszul algebra and let $f \in S_2$ be a homogeneous regular normal element. Then $B:=S/(f) =T(V)/(R+kf)$ and there is the canonical surjection $\pi_S: S=T(V)/(R)\to T(V)/(R+kf)=B$. Moreover we have $S^!=T(V^*)/(R^\perp), B^!=T(V^*)/(R^\perp \cap f^\perp)$, so there is the canonical surjection $\pi_{B^!}:  B^!=T(V^*)/(R^\perp \cap f^\perp)\to T(V^*)/(R^\perp)=S^!$.
Then there is an element $w\in T(V^*)_2$ for which $R^\perp= (R^\perp \cap f^\perp)+kw$. By abuse of notation, let $w:=w+(R^\perp \cap f^\perp) \in B^!_2$. The following is known.

\begin{prop} \label{prop.KD} Let $S,f,B$ be as above.
	\begin{enumerate}
		\item  \textnormal{\cite[Theorem 1.2]{ST}} $B=S/(f)$ is Koszul.
		\item \textnormal{\cite[Corollary 1.4]{ST}} $w \in B^!$ is regular and normal such that $B^!/(w)\simeq S^!$.
	\end{enumerate}
\end{prop}

\begin{dfn}\label{dfn.qh}
	An $\NN$-graded algebra $B$ is called a \emph{noncommutative quadric hypersurface} of dimension $d-1$ if $B$ has a form $B=S/(f)$, where
	\begin{itemize}
		\item $S$ is a connected Koszul AS-regular algebra of dimension $d$, and 
		\item $f \in S$ is a homogeneous regular normal element of degree $2$.
	\end{itemize}
\end{dfn}

Notice that we do \emph{not} assume that $f$ is central as in \cite{SV}.

\begin{prop}\label{prop.KdGor} Let $B=S/(f)$ be a noncommutative quadric hypersurface of dimension $d-1$, and let $A=(B^{!})^{\op}$.
	\begin{enumerate}
		\item  $B$ is a connected Koszul AS-Gorenstein algebra of dimension $d-1$ and Gorenstein parameter $d-2$.
		\item $(S^!)^{\op}$ is a finite dimensional connected Koszul AS-Gorenstein algebra of dimension $0$ and  Gorenstein parameter $-d$
		\item There exists a regular normal element $w \in A$ of degree $2$ such that $A/(w)=(S^!)^{\op}$.
		\item  $A$ is a connected Koszul AS-Gorenstein algebra of dimension $1$ and  Gorenstein parameter $2-d$. 
	\end{enumerate}
\end{prop}

\begin{proof}
	(1) follows from Rees-Lemma (e.g.\ \cite[Proposition 3.4(b)]{Lev}).
	(2) follows from \cite[Proposition 5.10]{Sm}.
	(3) follows from Proposition \ref{prop.KD}.
	(4) follows from (3) and Rees-Lemma.
\end{proof}

Let $B=S/(f)$ be a noncommutative quadric hypersurface and let $A=(B^!)^{\op}$.
Since $A$ has a regular normal element $w \in A_2$, there exists a unique graded algebra automorphism $\psi_w$ of $A$ such that $aw = w\psi_w (a)$ for $a \in  A$.
Thus the multiplicative subset $\{w^i\mid i\ge0\}$ of $A$ satisfies the Ore condition, and we have a localization $A[w^{-1}]$.
An element of  $A[w^{-1}]$ is denoted as
$aw^{-i}$ with $a\in A, i\in \NN$.
Note that the $\ZZ$-graded algebra structure of $A[w^{-1}]$ is given by the following,
where $a,a'\in A$ and $i,j \in \NN$.

\begin{itemize}
	\item (addition) $aw^{-i}+a'w^{-j}=(aw^j+a'w^i)w^{-i-j}$,
	\item (multiplication) $(aw^{-i})(a'w^{-j})=a\psi_w^i(a')w^{-i-j}$,
	\item (grading) $\deg(aw^{-i})=\deg a-2i$.
\end{itemize}

By \cite[Proof of Proposition 4.6]{MU3}, the functor
$\mod^{\ZZ} A \to \mod^{\ZZ} A[w^{-1}]; M \mapsto M\otimes_A A[w^{-1}]$ induces an equivalence $\qgr A \simeq \mod^{\ZZ} A[w^{-1}]$. Thus the graded total quotient ring $Q_A$ defined in \eqref{define Q} is isomorphic to $A[w^{-1}]$.
We define
\begin{equation}\label{define C(A)}
 C(B):=A[w^{-1}]_0\simeq (Q_A)_0.
 \end{equation}
Since $A$ is generated by elements in degree $1$, we can choose $q=1$ in Theorem \ref{prop.Q4}(1). Thus $Q_A\simeq A[w^{-1}]$ is strongly graded, and we obtain the equivalences
\begin{align} \label{eq.qgrKos}
	\qgr A \simeq \mod^{\ZZ} A[w^{-1}] \simeq \mod C(B)
\end{align}
(see Corollary \ref{cor.qql}(1), \cite[Lemma 4.13]{MU3}). In particular, they induce an equivalence
\[G: \Db(\qgr A) \xrightarrow{\sim} \Db(\mod C(B)).\]
In addition, the following holds, where the equivalence of (3), (4), and (5) also follows from Corollary \ref{cor.qql}(1).
\begin{prop}\label{prop.gldimq}
	{\cite[Theorem 5.5]{MU3}}
	Let $B=S/(f)$ be a noncommutative quadric hypersurface of dimension $d-1$ with $d\geq2$, and let $A=(B^!)^{\op}$. 	
	Then the following are equivalent.
	\begin{enumerate}
		\item $\qgr B$ has finite global dimension.
		\item $\gldim(\qgr B)=d-2$.
		\item $C(B)$ is a semisimple algebra.
		\item $\qgr A$ has finite global dimension.
		\item $\gldim(\qgr A)=0$.
	\end{enumerate}
\end{prop}  

\begin{ex} \label{ex.cqh}
	Let $S = k[x_1, \dots, x_n]$, $f=x_1^2 + \cdots + x_n^2 \in S_2$, $B=S/(f)$, and $A=(B^!)^{\op}$.
	Then $A$ is isomorphic to $k\langle x_1, \dots, x_n\rangle/(x_jx_i+x_ix_j, x_n^2-x_i^2 \mid 1\leq i, j\leq n, i\neq j)$ and $w=x_1^2=\cdots=x_n^2 \in A_2$ is a regular central element such that $A/(w) \simeq (S^!)^{\op}$.
	Furthermore, one can check that 
	\begin{align*}
C(B) &\simeq k\langle t_1, \dots, t_{n-1}\rangle/(t_jt_i+t_it_j,\, t_i^2-1 \mid 1\leq i, j\leq n-1, i\neq j)\simeq \begin{cases} \rM_{2^{(n-1)/2}}(k) &\text{if $n$ is odd},\\ \rM_{2^{(n-2)/2}}(k)^2 &\text{if $n$ is even} \end{cases}
\end{align*}
	(see e.g.\ \cite{Lee}). By \eqref{eq.qgrKos}, we have
	\[  \qgr A \simeq \mod C(B) \simeq
	\begin{cases} \mod k &\text{if $n$ is odd},\\ \mod k^2 &\text{if $n$ is even}. \end{cases}
	\]
\end{ex}

\subsection{Tilting theory for noncommutative quadric hypersurfaces}
The following is the main result of this section. 

\begin{thm}\label{thm.nqh}
	Let $B=S/(f)$ be a noncommutative quadric hypersurface of dimension $d-1$ with $d\geq 2$, and let $A=(B^!)^{\op}$.
	Assume that $\qgr B$ has finite global dimension (or, equivalently, $C(B)$ is semisimple).
	\begin{enumerate}
	\item There exists a duality $F: \Db(\qgr B) \to \uCM^{\ZZ}A$ such that $F(\Omega ^ik(i)) \simeq A(i)_{\geq 0}$ for any $i\in \NN$.
	\item $\uCM^{\ZZ} A$ has a tilting object $\bigoplus_{i=1}^{d-1} A(i)_{\geq 0}$, and $\Db(\qgr B)$ has a tilting object $\bigoplus_{i=1}^{d-1}\Omega^ik(i)$. Moreover, they correspond to each other via the duality $F$ in (1).
	\item Take an arbitrary direct sum decomposition $\Omega^{d-1}k(d-1)=\bigoplus_{j=1}^{\ell}X_j^{\oplus m_j}$ in $\qgr B$, where $X_j$'s are pairwise non-isomorphic indecomposable. Then we have a full strong exceptional collection in $\Db(\qgr B)$\[(X_{\ell}, X_{\ell-1}, \dots, X_1, \Omega^{d-2}k(d-2), \dots, \Omega^1k(1)).\]
	\item Let $\Lambda := \End_{\Db(\qgr B)}(\bigoplus_{i=1}^{d-1}\Omega^ik(i))$. Let
	$Q=Q_{A}$ be the graded total quotient ring of $A$. 
	Then we have isomorphisms of $k$-algebras
		\[
		\Lambda \simeq \End_{\uCM^{\ZZ}A}(\bigoplus_{i=1}^{d-1} A(i)_{\geq 0})^{\op}
		\simeq
		\begin{bmatrix}
			k &0&\cdots  &\cdots  &0\\
			A_1&k &\ddots  &&\vdots \\
			\vdots&\vdots&\ddots  &\ddots&\vdots \\
			A_{d-3}&A_{d-4}&\cdots&k &0\\
			Q_{d-2}&Q_{d-3}&\cdots&Q_1&Q_0
		\end{bmatrix}^{\op},
		\]
		and we have triangle equivalences
		\[\Db(\qgr B) \simeq(\uCM^{\ZZ} A)^{\op} \simeq \Db(\mod \Lambda).\]
	\end{enumerate}
\end{thm}

\begin{proof}
	(1) This follows from Propositions \ref{prop.Kd} and \ref{prop.KdGor}.
	
	(2) Since $\qgr B$ has finite global dimension, we have
	$\gldim(\qgr A)=0$ by Proposition \ref{prop.gldimq}. Thus we have $\uCM^{\ZZ} A = \uCM^{\ZZ}_0 A$ by Proposition \ref{prop.CM0}.
	Since Gorenstein parameter of $A$ is $2-d \leq 0$,
	and $\proj^{\ZZ}Q=\add Q$ holds (i.e.\ we can choose $q=1$), it follows that $\uCM^{\ZZ}A$ has a tilting object $\bigoplus_{i=1}^{d-1} A(i)_{\geq 0}$ by Theorem \ref{a and tilting}. 
	Since $F(\Omega^ik(i)) \simeq A(i)_{\geq 0}$ holds for any $i\in \NN$ by (1), $\Db(\qgr B)$ has a tilting object $\bigoplus_{i=1}^{d-1}\Omega^ik(i)$.
	
	(3) This follows from Corollary \ref{a and tilting 2}(3).
	
	(4) The isomorphisms follow from (2) and Example \ref{End of connected case}.
Since $\Lambda$ has finite global dimension, we get $\Db(\qgr B) \simeq (\uCM^{\ZZ} A)^{\op} \simeq \per \Lambda=\Db(\mod \Lambda)$.
\end{proof}

\begin{ex}\label{ex.coco}
	We consider the case $n=3$ in Example \ref{ex.cqh}, that is,
	the case where $S=k[x,y,z]$, $f=x^2+y^2+z^2 \in S_2$, and $B=S/(f)$.
	Then $S$ is a connected Koszul AS-regular algebra of dimension $3$ and $f$ is a regular central element of $S$. It is easy to see that $A=(B^!)^{\op} \simeq k\langle x, y, z\rangle/(yx+xy, zx+xz, zy+yz, x^2-y^2, x^2-z^2)$ 
	is a connected Koszul AS-Gorenstein algebra of dimension $1$
	and Gorenstein parameter $-1$, and $w=x^2=y^2=z^2 \in A_2$ is a regular central element such that $A/(w)\simeq (S^!)^{\op}$.
	As seen in Example \ref{ex.cqh}, $C(B)\simeq \rM_2(k)$ is a semisimple algebra, so $\gldim(\qgr B)<\infty$ by Proposition \ref{prop.gldimq}.
	Thus $\Db(\qgr B) \simeq(\uCM^{\ZZ} A)^{\op}$ has a tilting object by Theorem \ref{thm.nqh}. Now $V=A(1)_{\geq 0} \oplus A(2)_{\geq 0}$ is a tilting object of $\uCM^{\ZZ} A$.
	By Theorem \ref{information on V}(3), we have $\dim_k\End_{\uCM^{\ZZ} A}(A(1)_{\geq 0})=1$.
	One can calculate that
	\[ A(2)_{\geq 0} \simeq L^1\oplus L^2, \]
	where 
	\begin{align*}
		L^1 = ((yz+\xi xz)A+(xy-\xi x^2)A) (2),\quad  L^2 =  ((yz-\xi xz)A+(-xy-\xi x^2)A) (2),&&\xi=\sqrt{-1}.
	\end{align*}
	Moreover, we see that $L^1$ and $L^2$ are isomorphic via 
	\[
	L^1 \to L^2;\quad yz+\xi xz\mapsto -xy-\xi x^2,\quad xy-\xi x^2\mapsto yz-\xi xz.
	\]
	It follows from Theorem \ref{information on V}(3) that
	$\dim_k\End_{\uCM^{\ZZ} A}(L^1)=1$ and $\dim_k \Hom_{\uCM^{\ZZ} A}(A(1)_{\geq 0},L^1)=2$,
	so $\End_{\uCM^{\ZZ} {A}} (V)$ is Morita equivalent to the path algebra $kK_2$ of the following quiver $K_2$:
	\[\xymatrix@C2em@R1em{
		&A(1)_{\geq 0} \ar@<-0.6ex>[r]\ar@<0.6ex>[r] &L^1 &\text{($2$-Kronecker quiver)}.
	}\]
	Hence we have a triangle equivalence
	\[ \Db(\qgr B) \simeq (\uCM^{\ZZ} {A})^{\op} \simeq \Db(\mod kK_2)^{\op}\simeq \Db(\mod kK_2).  \]
	Let $X =\Projs B$. Notice that this example reproves the fact that $\Db(\coh X)\simeq \Db(\mod kK_2) \simeq \Db(\coh \mathbb{P}^1)$.
\end{ex}

Let $B=S/(f)$ be a commutative quadric hypersurface of dimension greater than or equal to $2$, and let $X=\Proj B$. Then it is well-known that
$\Db(\qgr B) \simeq \Db(\coh X)$ has a full strong exceptional collection with the Spinor bundles by Kapranov's theorem \cite{Ka1, Ka2}. 
In this case, our exceptional collection for $\Db(\qgr B)$ obtained in Theorem \ref{thm.nqh} can be realized as an iterated left mutation of Kapranov's exceptional collection.

As a comparison with Example \ref{ex.coco} on a commutative conic, we now give an example using a noncommutative conic; see \cite{HMM, Uey} for similar results.

\begin{ex} 
	Let $S=k\langle x, y, z\rangle/(xy-yx, xz-zx, yz+zy+x^2)$, $f=y^2+z^2 \in S_2$, and $B=S/(f)$.
	Then 
	$S$ is a connected Koszul AS-regular algebra of dimension $3$
	and $f$ is a regular central element of $S$. It is easy to see that $A=(B^!)^{\op} \simeq k\langle x, y, z\rangle/(yx+xy, zx+xz, zy-yz, x^2-zy, y^2-z^2)$ 
	is a connected Koszul AS-Gorenstein algebra of dimension $1$
	and Gorenstein parameter $-1$, and $w=y^2=z^2 \in A_2$ is a regular central element such that $A/(w)\simeq (S^!)^{\op}$.
	Since $C(B)\, (=A[w^{-1}]_0) \to k[t]/(t^4-1); xyw^{-1} \mapsto t$ is an isomorphism, $C(B)$ is a semisimple algebra isomorphic to $k^4$.
	Thus $\gldim(\qgr B)<\infty$ by Proposition \ref{prop.gldimq}.
	It follows from Theorem \ref{thm.nqh} that $\Db(\qgr B) \simeq(\uCM^{\ZZ} A)^{\op}$ has a tilting object. Now $V=A(1)_{\geq 0} \oplus A(2)_{\geq 0}$ is a tilting object of $\uCM^{\ZZ} A$.
	By Theorem \ref{information on V}(3), we have $\dim_k\End_{\uCM^{\ZZ}A}(A(1)_{\geq 0})=1$.
	One can calculate that
	\[ A(2)_{\geq 0} \simeq L^1\oplus L^2 \oplus L^3 \oplus L^4,  \]
	where 
	\begin{align*}
		&L^1 = (xy+xz+\xi yz+\xi y^2)A(2),&&L^2 = (xy+xz-\xi yz-\xi y^2)A(2),\\
		&L^3 = (xy-xz-yz+y^2)A(2),&&L^4 = (xy-xz+yz-y^2)A(2), &&\xi=\sqrt{-1}.
	\end{align*}
	Since $\dim_k\End_{\uCM^{\ZZ} A}(A(2)_{\geq 0})=4$
	by Theorem \ref{information on V}(3), we have $\dim_k\End_{\uCM^{\ZZ} A}(L^i)=1$ and $L^i \not\simeq L^j$ for $i \neq j$.
	Since $0\to L^i \to L^i(1) \to k(1)\to 0$ is a non-split exact sequence, we have $\Ext^1_{\mod^{\ZZ} A}(k(1), L^i)\neq 0$.
	 Applying $\Hom_{\mod^{\ZZ} A}(-,L^i)$ to the exact sequence $0\to A_{\ge1}(1)\to A(1)\to k(1)\to0$, we obtain
	\begin{align*}
	\Hom_{\uCM^{\ZZ} A}(A(1)_{\geq 0},L^i)
	\stackrel{\text{Prop.\,}\ref{j-i}(5)}{=}\Hom_{\mod^{\ZZ} A}(A_{\geq 1}(1), L^i)
	\simeq \Ext^1_{\mod^{\ZZ} A}(k(1), L^i)\neq 0.
	\end{align*}
	 Since $\dim_k \Hom_{\uCM^{\ZZ} A}(A(1)_{\geq 0}, A(2)_{\geq 0})=4$ by Theorem \ref{information on V}(3), it follows that $\dim_k \Hom_{\uCM^{\ZZ}A}(A(1)_{\geq 0},L^i)=1$.
	Therefore $\End_{\uCM^{\ZZ} A} (V)$ is isomorphic to the path algebra $kQ$ of the following quiver $Q$:
	\[\xymatrix@C1em@R0em{
		&&A(1)_{\geq 0} \ar[lldd]\ar[ldd]\ar[rdd]\ar[rrdd]\\
		&&&&&\text{(type $\widetilde{D_4}$)}.\\
		L^1 &L^2 &&L^3 &L^4 
	}\]
	Hence we have a triangle equivalence
	\[ \Db(\qgr B) \simeq (\uCM^{\ZZ} A)^{\op} \simeq \Db(\mod kQ)^{\op}\simeq \Db(\mod kQ).  \]
\end{ex}

\appendix\section{Combinatorics}\label{appendix}

\subsection{Non-negativity of matrices}
In this subsection we give combinatorial results for a matrix.
For a finite set $I$, let $I^2=I \times I$ and $\Aut(I)=\Aut_\mathrm{set}(I)$.

\begin{dfn}
Let $I$ be a finite set and $m:I^2\to\ZZ$ a map.
\begin{enumerate}
\item We call $m$ \emph{non-negative} if $m(i,j)\ge0$ holds for any $i,j\in I$. We denote by $\Sq I$ the set of sequences $\fs=(i_1, i_2,\dots, i_n)$ in $I$ with $n\geq 1$. For $\fs=(i_1, i_2,\dots, i_n)\in \Sq I$, let
\[m(\fs):=\sum_{k=1}^nm(i_k, i_{k+1}),\ \mbox{where $i_{n+1}:=i_1$.}\]
We say that $m$ is \emph{$\Sigma$-non-negative} if $m(\fs)\geq 0$ holds for any $\mathfrak{s}\in\Sq I$.
\item For a map $s : \II \to \ZZ$, we define a map $sm : \II^2 \to \ZZ$ called the \emph{conjugate} of $m$ by
\[
 (sm)(i, j):=m(i, j)+s(i)-s(j).
\]
\end{enumerate}
\end{dfn}

The aim of this subsection is to prove the following result.

\begin{thm}\label{lem-mat-pos}
Let $m : I^2 \to \ZZ$ be a map. Then $m$ admits a non-negative conjugate if and only if $m$ is $\Sigma$-non-negative.
\end{thm}

To prove Theorem \ref{lem-mat-pos}, we fix some notations.
For $\fs=(i_1, i_2,\dots, i_n)\in \Sq I$, let $\#\fs:=\#\{i_1, i_2,\dots, i_n\}$ and
\[
m_{\min}:=\min\{m(\fs)\mid \fs\in\Sq I,\ \#\fs\ge2\}.
\]
Call a sequence $\fs=(i_1, i_2,\dots, i_n)$ \emph{multiplicity-free} if $i_1,\ldots,i_n$ are pairwise distinct. The proof of the following easy observations is left to the reader.

\begin{lem}\label{easy fact of m}
Let $I$ be a finite set and $m:I^2\to\ZZ$ a map.
\begin{enumerate}
\item $m(\fs)=(sm)(\fs)$ holds for each $s:I\to\ZZ$ and $\fs\in\Sq I$. Therefore $\Sigma$-non-negativity is preserved by taking conjugations.
\item $m_{\min}=m(\fs)$ holds for some multiplicity-free $\mathfrak{s}\in\Sq I$ with $\#\fs\ge2$. Hence $m$ is $\Sigma$-non-negative if and only if $m_{\min}\ge0$ and $m(i,i)\ge0$ for each $i\in I$.
\item $m$ is $\Sigma$-non-negative if and only if $\sum_{i\in I}m(i,\sigma(i))\ge0$ holds for each $\sigma\in\Aut(I)$.
\end{enumerate}
\end{lem}

The ``only if'' part of Theorem \ref{lem-mat-pos} is immediate from Lemma \ref{easy fact of m}(1).

In the rest, we prove the ``if'' part of Theorem \ref{lem-mat-pos}.
The following result confirms it in a special case.

\begin{lem}\label{lem-cyc-pos}
Let $I=[0, n-1]$ and $m : I^2 \to \ZZ$ a $\Sigma$-non-negative map. Assume that a sequence $\mathfrak{i}=(0, 1, \dots, n-1)$ satisfies $m(\mathfrak{i})=m_{\min}$. Then there exists a conjugate $m'$ of $m$ satisfying the following statements.
\begin{itemize}
\item[(1)] $m'$ is non-negative.
\item[(2)] $m'(i, i+1)=0$ for $i\neq n-1$ and $m'(n-1, 0)=m_{\min}$. (In this case, we call $m'$ \emph{normalized}.)
\end{itemize}
\end{lem}

\begin{proof}
Let $s : I \to \ZZ$ be a map such that $s(0)=0$ and $s(i)=\sum_{j\in [0, i-1]}m(j, j+1)$ for $i\in [1, n-1]$.
We show that $m'=sm$ satisfies the desired conditions.
We have $m'(i, i+1) = m(i, i+1) +s(i) - s(i+1) = 0$ for any $i\in[0,n-2]$, and $m'(n-1, 0)=m(n-1, 0) + s(n-1) = m_{\min}$. Thus (2) is satisfied.

To prove (1), fix $i,j\in I$. If $i>j$, let $I'=[j, i]$ and $\fs=(j, j+1, \dots, i-1, i)\in\Sq I$. We have
\begin{align*}
m'(i, j) & = m(i, j) +s(i) -s(j) = m(i, j) + \sum_{k\in I'}m(k, k+1)= m(\fs)\geq 0.
\end{align*}
This is illustrated as follows.
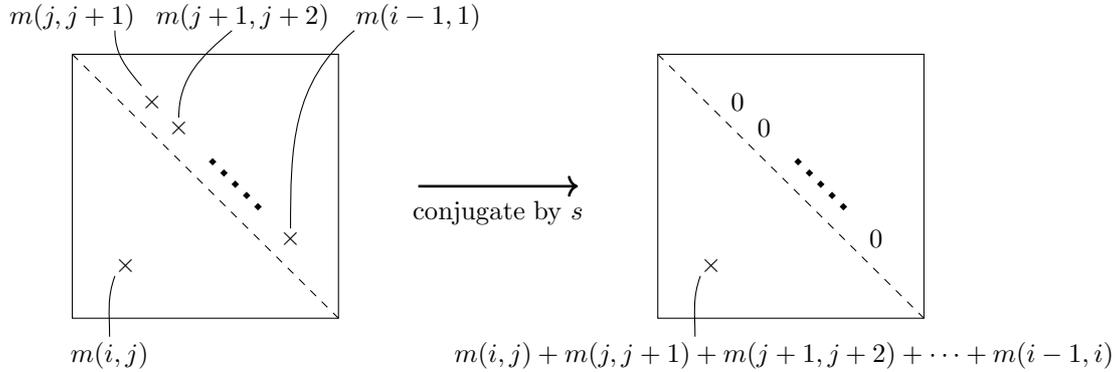
\begin{center}
\begin{tikzpicture}[scale=0.7]
\draw (0,0)--(0,5)--(5,5)--(5,0)--cycle; 
\draw[dashed] (0,5)--(5,0);

\node at (1,1) {$\times$};
\node at (1.5,4.1) {$\times$};
\node at (2,3.6) {$\times$};
\node at (4.1,1.5) {$\times$};
\draw[loosely dotted, line width=2pt] (2.6,3.0)--(3.55,2.05);

\draw (0.8,0.8) to [out=250, in=90] (0.7,-0.35);
\node at (0.7,-0.7) {$m(i,j)$}; 

\draw (1.3,4.4) to [out=130, in=270] (0.8,5.45);
\node at (0,5.7) {$m(j,j+1)$}; 

\draw (2, 3.8) to [out=90, in=220] (3,5.45);
\node at (3.2,5.7) {$m(j+1, j+2)$};

\draw (4.1,1.7) to [out=90, in=220] (5.3, 5.45);
\node at (6.5, 5.7) {$m(i-1,i)$}; 

\draw[->, line width=1pt] (6.5,2.5)--(9.5,2.5);
\node at (8, 2) {conjugate by $s$}; 

\draw (11,0)--(11,5)--(16,5)--(16,0)--cycle; 
\draw[dashed] (11,5)--(16,0);

\node at (12.5,4.1) {$0$};
\node at (13,3.6) {$0$};
\node at (15.1,1.5) {$0$};
\draw[loosely dotted, line width=2pt] (13.6,3.0)--(14.55,2.05);

\node at (12,1) {$\times$};
\draw (11.8,0.8) to [out=250, in=90] (11.7,-0.35); 
\node [anchor=west] at (5,-0.7) {$m(i,j)+ m(j,j+1) + m(j+1,j+2) + \cdots + m(i-1,i)$};
\end{tikzpicture}
\end{center}
If $i<j$, let $I'=[i, j-1]$ and $\fs=(0, 1, \dots, i, j, \dots, n-1)\in\Sq I$. Then $m_{\min}=m(\fs)+\sum_{k\in I'}m(k, k+1)-m(i, j)$ holds. Therefore we have
\begin{align*}
m'(i, j) & = m(i,j) + s(i)-s(j) = m(i, j) - \sum_{k\in I'} m(k, k+1)
= m(\fs) - m_{\min}\geq 0.\qedhere
\end{align*}
\end{proof}

We are ready to prove Theorem \ref{lem-mat-pos}.

\begin{proof}[Proof of Theorem \ref{lem-mat-pos}]
We show the assertion by an induction on $\# I$. Take multiplicity-free $\ft\in \Sq I$ with $m(\ft) =m_{\min}$ and $\#\ft\ge2$ (Lemma \ref{easy fact of m}(2)). Label the elements $\ft=(0, 1, \dots, \ell-1)$, and let $\ft=\{0, 1, \dots, \ell-1\}\subset I$ by abuse of notations.
Let $m_\ft:\ft^2\to\ZZ$ be a restriction of $m$.
Applying Lemma \ref{lem-cyc-pos} to $(m,\mathfrak{i}):=(m_\ft,\ft)$, we obtain a map $s : \ft \to \ZZ$ such that $sm_\ft$ is non-negative and normalized. Extending $s$ to a map $s:I\to\ZZ$ by $s|_{I\setminus \ft}=0$ and replacing $m$ by $s m$, we may assume that $m|_{\ft^2}$ is non-negative and normalized, that is,
\begin{equation}\label{normalized}
\mbox{$m(i, i+1)=0$ for $i\in \ft \setminus\{\ell-1\}$ and $m(\ell-1, 0)=m_{\min}$.}
\end{equation}
Let $\overline{I}=(I\setminus \ft) \sqcup\{ t \}$.
Then $\#\overline{I}<\#I$ holds. We define a map $\overline{m} : \overline{I}^2 \to \ZZ$ by
\begin{itemize}
\item $\overline{m}(i, j)=m(i, j)$ for $i, j \in I\setminus \ft$,
\item $\overline{m}(i, t)=\min\{ m(i, j) \mid j\in \ft\}$ for each $i\in I\setminus \ft$,
\item $\overline{m}(t, j)=\min\{m(i, j) \mid i \in \ft\}$ for each $j\in I\setminus \ft$,
\item $\overline{m}(t, t)=0$.
\end{itemize}
We claim that $\overline{m}$ is $\Sigma$-non-negative. By Lemma \ref{easy fact of m}(2), it suffices to prove $\overline{m}(\fs)\geq 0$ for each multiplicity-free $\fs\in\Sq\overline{I}$ with $\#\fs\ge2$.

First, assume $t\notin \fs$. Then we may regard $\fs\in\Sq I$, and we have $\overline{m}(\fs)=m(\fs)\geq 0$ as desired.

Secondly, assume $t\in\fs=(i_1, \dots, i_n)$. Assume $i_n=t$ without loss of generality. By definition of $\overline{m}$, there exist elements $a, b\in \ft$ such that $\overline{m}(i_{n-1}, t)=m(i_{n-1}, a)$ and $\overline{m}(t, i_1)=m(b, i_1)$. Let
\[\Sq I\ni\widetilde{\mathfrak{s}}:=\left\{\begin{array}{ll}
(i_1, i_2, \dots, i_{n-1}, a, a+1, \dots, b-1, b)&a\le b\\
(i_1, i_2, \dots, i_{n-1}, a, a+1, \dots, \ell-1, 0, \cdots, b-1)&a>b.
\end{array}\right.\]
If $a\le b$, then
\begin{align*}
0\le m(\widetilde{\fs})&=\sum_{k\in [1, n-2]}m(i_k, i_{k+1})+m(i_{n-1}, a) + \sum_{j\in[a,b-1]}m(j,j+1) + m(b, i_1)\stackrel{(\ref{normalized})}{=}\overline{m}(\fs).
\end{align*}
If $a>b$, then
\begin{align*}
m_{\min}\le m(\widetilde{\fs})=&\sum_{k=1}^{n-2}m(i_k, i_{k+1})+m(i_{n-1}, a) +  \sum_{j=a}^{\ell-2}m(j,j+1)+m(\ell-1,0)+ \sum_{j=0}^{b-1}m(j,j+1) + m(b, i_1)\\
&\stackrel{(\ref{normalized})}{=}\overline{m}(\fs) + m_{\min}.
\end{align*}
Therefore $\overline{m}(\fs)\ge0$ holds in both cases.

Since $\#\overline{I}<\#I$, by the induction hypothesis, there exists a map $\overline{s} : \overline{I}\to \ZZ$ such that $\overline{s}\overline{m}$ is non-negative.
Let $s : I \to \ZZ$ be a map such that $s|_{I\setminus \ft} = \overline{s}$ and $s(i) = \overline{s}(t)$ for any $i\in \ft$.
Then $s m$ is clearly non-negative. We complete the proof.
\end{proof}

\subsection{Almost constancy of matrices}

The following elementary notion plays a central role.

\begin{dfn-prop}\label{dfn-prop-app-constant}
We call a sequence $(\ell_i)_{i\in I}$ of integers \emph{almost constant} if one of the following equivalent conditions are satisfied.
\begin{itemize}
\item There exists $\ell\in\ZZ$ such that $\ell_i\in\{\ell,\ell+1\}$ for each $i\in I$.
\item For all $i,j\in I$, we have $\ell_i-\ell_j\in\{-1,0,1\}$.
\end{itemize}
When $I$ is finite, the following condition is also equivalent, where   $\ell$ is the average of $\ell_i$.
\begin{itemize}
\item $|\ell_i-\ell|<1$ for all $i\in I$.
\end{itemize}
\end{dfn-prop}

For each $\nu\in\Aut(I)$, we define an element $\nu\in\Aut(I^2)$ by $\nu(i,j):=(\nu(i),\nu(j))$.

\begin{dfn}\label{define m-data}
Let $I$ be a finite set. An \emph{m-data} on $I$ is a triple $(m,a,\nu)$ consisting of maps $m:I^2 \to\ZZ$, $a:I\to \ZZ$ and $\nu\in\Aut(I)$ satisfying the following conditions.
\begin{enumerate}
\item For each $i,j\in I$, we have
\[m(\nu(i,j))=m(i,j)-a(i)+a(j).\]
\item There exists $a_\mathrm{av}\in\QQQ$ such that, for each $\nu$-orbit $I'\subset I$, we have
\[a_\mathrm{av}\cdot\# I'=\sum_{i\in I'}a(i).\]
\end{enumerate}
We call an m-data \emph{non-negative} if $m$ is non-negative.
We call an m-data \emph{almost constant} if the sequence $(a(i))_{i\in I}$ is almost constant.

For $s:I\to\ZZ$, we have a new m-data $s(m,a,\nu) =(sm, sa, \nu)$ called a \emph{conjugate} given by
\[\mbox{$(sm)(i, j)=m(i, j)+s(i)-s(j)$ and $(sa)(i)=a(i)+s(i)-s(\nu(i))$}\]
for each $i, j\in I$.
\end{dfn}

We give two examples of m-data which are necessary in Section \ref{section: GP and GME}.

\begin{ex}
Let $A$ be a ring-indecomposable basic $\NN$-graded AS-Gorenstein algebra with the Gorenstein parameter $p^A$ and $\{e_i \mid i\in I\}$ its complete set of orthogonal primitive idempotents.
By Proposition \ref{prop average a-invariant} and \eqref{eq-m-a-a}, the map $a^A : I \to \ZZ; i \mapsto -p_i$ and the Nakayama permutation $\nu : I \to I$ induces a non-negative m-data $(m^A, a^A, \nu)$ on $I$.
\end{ex}

\begin{ex}\label{extend to m-data}
Let $I$ be a finite set, $H$ a subset of $I^2$, and $(m,a,\nu)$ a triple consisting of maps $m:H \to\ZZ$, $a:I\to \ZZ$ and $\nu\in\Aut(I)$ satisfying $\nu(H)=H$, the condition (2) in Definition \ref{define m-data}, and
\[m(\nu(i,j))=m(i,j)-a(i)+a(j)\ \mbox{ for each }\ (i,j)\in H.\]
Then we can extend $m$ to a map $m':I^2\to\ZZ$ such that $(m',a,\nu)$ is an $m$-data. Moreover, for a given $N\in\ZZ$, we can choose $m'$ such that $\min m'(I^2\setminus H)\ge N$.
\end{ex}

\begin{proof}
Fix a complete set $K\subset I^2\setminus H$ of representatives of $(I^2\setminus H)/\langle\nu\rangle$. Fix $L\in\ZZ$, and extend $m$ to $m'$ by
\[m'(\nu^\ell(i,j)):=L-\sum_{k=0}^{\ell-1}a(\nu^k(i))+\sum_{k=0}^{\ell-1}a(\nu^k(j))\ \mbox{ for each }\ (i,j)\in K\ \mbox{ and }\ \ell\in\ZZ.\]
This is well-defined thanks to the condition (2) in Definition \ref{define m-data}, and clearly satisfies the condition (1) in Definition \ref{define m-data}. Thus $(m',a,\nu)$ is an m-data. The last condition is also satisfied if $L$ is enough large.
\end{proof}

In this subsection we prove the following result.

\begin{thm}\label{thm-nonneg-app-constant}
Let $(m,a,\nu)$ be an m-data.
\begin{enumerate}
\item $(m,a,\nu)$ is conjugate to an almost constant m-data.
\item If $(m,a,\nu)$ is $\Sigma$-non-negative, then it is conjugate to an almost constant non-negative m-data.
\end{enumerate}
\end{thm}

We start with preparing elementary properties of the floor function. Let $\lfloor x \rfloor:=\mathrm{max}\{n \in \ZZ \mid n\leq x \}$ for $x\in\RR$.
The key step of the proof of Theorem \ref{thm-nonneg-app-constant} is the following.

\begin{lem}\label{floor sequence}
Let $n>0$ and $p$, $q$, $r$ be integers.
\begin{enumerate}
\item The sequence $(\lfloor(ir+p)/n\rfloor-\lfloor(ir+q)/n\rfloor)_{i\in\ZZ}$ is almost constant.
\item If $n$ and $r$ are coprime, then $p=\sum_{i=0}^{n-1}(\lfloor ir/n \rfloor-\lfloor (ir-p)/n \rfloor)$.
\end{enumerate}
\end{lem}

\begin{proof}
(1) Since $\frac{ir+p}{n}-1<\lfloor\frac{ir+p}{n}\rfloor\le\frac{ir+p}{n}$ and $\frac{ir+q}{n}-1<\lfloor\frac{ir+q}{n}\rfloor\le\frac{ir+q}{n}$, we have
\[\frac{p-q}{n}-1=\frac{ir+p}{n}-1-\frac{ir+q}{n}<\left\lfloor\frac{ir+p}{n}\right\rfloor-\left\lfloor\frac{ir+q}{n}\right\rfloor<\frac{ir+p}{n}-\frac{ir+q}{n}+1=\frac{p-q}{n}+1.\]
Therefore $\lfloor(ir+p)/n\rfloor-\lfloor(ir+q)/n\rfloor\in\{\lfloor(p-q)/n\rfloor,\lfloor(p-q)/n\rfloor+1\}$ holds for all $i\in\ZZ$.

(2)  Write $p=sn+t$ for $s,t\in\ZZ$ with $0\leq t<n$. Then we have
\begin{align*} \lfloor ir/n \rfloor-\lfloor (ir-p)/n \rfloor=\left\{\begin{array}{ll}s+1& \mbox{$ir \in [0, t-1]+n\ZZ$},\\ s& \mbox{else}.\end{array}\right.
\end{align*}
Since $n$ and $r$ are coprime, we have $\sum_{i=0}^{n-1}(\lfloor ir/n \rfloor-\lfloor (ir-p)/n \rfloor)=t(s+1)+(n-t)s=p$.
\end{proof}

The following notion plays a key role.

\begin{dfn}\label{define floor}
For $c\in\QQQ$ and a positive integer $n$ satisfying $cn\in\ZZ$, consider a map
\[f_{c,n}:\ZZ/n\ZZ\to\ZZ;\ i\mapsto\lfloor (i+1)c\rfloor -\lfloor ic\rfloor,\]
which is almost constant by Lemma \ref{floor sequence}(1).
We say that an m-data $(m , a, \nu)$ is \emph{of floor type} if
\begin{enumerate}
\item[$\bullet$] there exists a decomposition
\begin{equation}\label{decompose I}
I=\bigsqcup_{x\in J}I_x=\bigsqcup_{x\in J}\ZZ/n_x\ZZ
\end{equation}
into $\nu$-orbits such that, for each $x\in J$, $\nu$ acts on $I_x=\ZZ/n_x\ZZ$ as $i\mapsto i+1$,
\item[$\bullet$] $a|_{I_x}=f_{a_\mathrm{av},n_x}$ holds, where $a_\mathrm{av}\in\QQQ$ is given in Definition \ref{define m-data}(2).
\end{enumerate}
\end{dfn}

We prove Theorem \ref{thm-nonneg-app-constant} in the following stronger form.

\begin{thm}\label{thm-nonneg-app-constant 2}
Let $(m,a,\nu)$ be an m-data.
\begin{enumerate}
\item $(m,a,\nu)$ is conjugate to an m-data of floor type.
\item If $(m,a,\nu)$ is $\Sigma$-non-negative, then it is conjugate to a non-negative m-data of floor type.
\end{enumerate}
\end{thm}

The following result implies Theorem \ref{thm-nonneg-app-constant 2}(1).

\begin{lem}\label{lem-cyclic-mdata}
Let $(m, a, \nu)$ be an m-data on $I$.
\begin{enumerate}
\item There exists a map $s : I \to \ZZ$ such that $s(m, a, \nu)$ is of floor type.
\item Assume that $(m,a,\nu)$ is of floor type, and write $a_\mathrm{av}=r/g$ for coprime integers $r$ and $g\ge1$. Then we have $a\circ \nu^g = a$ and $m(\nu^g(i,j))=m(i, j)$ for any $i,j\in I$.
\end{enumerate}
\end{lem}

\begin{proof}
(1) Take a decomposition \eqref{decompose I}. For each $i\in I_x
=\ZZ/n_x\ZZ$ with $x\in J$, let $s(i)= \sum_{j=0}^{i-1}a(j)-\lfloor ia_\mathrm{av}\rfloor$. Then the map $s:I\to\ZZ$ satisfies
\[(sa)(i)=a(i)+s(i)-s(i+1)=\lfloor (i+1)a_\mathrm{av} \rfloor -\lfloor ia_\mathrm{av} \rfloor=f_{a_\mathrm{av},n_x}(i).\]

(2) The first assertion is clear from $\lfloor (i+g+1)r/g \rfloor -\lfloor (i+g)r/g \rfloor=\lfloor (i+1)r/g \rfloor -\lfloor ir/g \rfloor$.
The second assertion follows from the first one and Definition \ref{define m-data}(1).
\end{proof}

We are ready to prove Theorem \ref{thm-nonneg-app-constant 2}(2).

\begin{proof}[Proof of Theorem \ref{thm-nonneg-app-constant 2}(2)]
Since $\Sigma$-non-negative maps are preserved by taking conjugations, by Lemma \ref{lem-cyclic-mdata}(1), we may assume that $(m, a, \nu)$ is a $\Sigma$-non-negative m-data of floor type. Take a decomposition \eqref{decompose I}. Write $a_\mathrm{av}=r/g$ for coprime integers $r$ and $g\ge1$. By Lemma \ref{lem-cyclic-mdata}(2), we have $a\circ\nu^g=a$ and
\begin{equation}\label{period g}
m(\nu^g(i,j))=m(i,j)\ \mbox{ for each }\ i,j\in I.
\end{equation}
Define maps
\begin{align*}
&m':I^2\to\ZZ;\ (i,j)\mapsto\sum_{k=0}^{g-1}m(\nu^k(i,j)),\\
&\overline{m}:J^2\to\ZZ;\ (x,y)\mapsto\min\{m'(i,j)\mid (i,j)\in I_x\times I_y\}.
\end{align*}
We give an explicit example of $m'$ and $\overline{m}$ in Example \ref{example of m'} below.

We claim that $\overline{m}$ is $\Sigma$-non-negative.
By Lemma \ref{easy fact of m}(3), it suffices to prove that, for each bijection $\sigma:J\to J$, $\overline{m}(\sigma):=\sum_{x\in J}\overline{m}(x,\sigma(x))$ is non-negative.
For each $x\in J$, fix $(i_x,j_x)\in I_x\times I_{\sigma(x)}$ satisfying $\overline{m}(x,\sigma(x))=m'(i_x,j_x)$. Then for each $\ell\in\ZZ$, we have
\begin{equation}\label{overline{m}=sum m}
\sum_{k=0}^{g-1}m(\nu^{\ell+k}(i_x,j_x))\stackrel{\eqref{period g}}{=}m'(i_x,j_x)=\overline{m}(x,\sigma(x)).
\end{equation}
For $L:=\lcm \{n_x\mid x\in J\}$, we have $\nu^L=1_I$. Consider a multiset
\[S:=\{\nu^k(i_x,j_x)\mid 0\le k\le L-1,\ x\in J\}.\]
Then each element $i\in I_x$ with $x\in J$ appears exactly $L/n_x$ times as the first (respectively, second) entry of the elements in $S$. Thus there exists $\fs=(h_1,\ldots,h_{Lt})\in\Sq I$ with $t=\# J$ such that
\[S=\{(h_k,h_{k+1})\mid k\in\ZZ/Lt\ZZ\}\]
as multisets. Since $m$ is $\Sigma$-non-negative, we have
\[0\le m(\fs)=\sum_{0\le k\le L-1,\ x\in J}m(\nu^k(i_x,j_x))\stackrel{\eqref{overline{m}=sum m}}{=}(L/g)\sum_{x\in J}\overline{m}(x,\sigma(x))=(L/g)\overline{m}(\sigma).\]
Thus $\overline{m}(\sigma)\ge0$ holds as desired.

Since $\overline{m}$ is $\Sigma$-non-negative, by Theorem \ref{lem-mat-pos}, there exists a map $\overline{s} : J\to \ZZ$ such that $\overline{s}\overline{m} : J^2\to \ZZ$ is non-negative, that is, for each $x,y\in J$, we have
\begin{equation}\label{m+s-s}
\overline{m}(x,y)+\overline{s}(x)-\overline{s}(y)\ge0.
\end{equation}
Define a map
\begin{align*}
s:I\to\ZZ\ \mbox{ by }\ s(i)=\lfloor ir/g \rfloor-\lfloor (ir-\overline{s}(x))/g \rfloor\ \mbox{ for each }\ i\in I_x=\ZZ/n_x\ZZ.
\end{align*}
We claim that $s(m,a,\nu)$ is almost constant and non-negative.

For $x\in J$ and $i\in I_x=\ZZ/n_x\ZZ$, we have
\begin{align*}
&a(i)+s(i)-s(\nu(i))\\
&=\left(\left\lfloor \frac{(i+1)r}{g} \right\rfloor-\left\lfloor \frac{ir}{g} \right\rfloor\right)+\left(\left\lfloor \frac{ir}{g} \right\rfloor-\left\lfloor \frac{ir-\overline{s}(x)}{g} \right\rfloor\right)-\left(\left\lfloor \frac{(i+1)r}{g} \right\rfloor-\left\lfloor \frac{(i+1)r-\overline{s}(x)}{g} \right\rfloor\right)\\
&=\left\lfloor \frac{(i+1)r-\overline{s}(x)}{g} \right\rfloor-\left\lfloor \frac{ir-\overline{s}(x)}{g} \right\rfloor.
\end{align*}
Thus $sa:I\to\ZZ$ is almost constant by Lemma \ref{floor sequence}(1).

We claim that, for each $i,j\in I$, $((sm)(\nu^k(i,j)))_{k\in\ZZ}$ is almost constant. In fact, for each $(i,j)\in I_x\times I_y$ with $x,y\in J$, we have
\begin{align*}
(sm)(\nu^\ell(i,j))
&=m(i,j)+\sum_{k=0}^{\ell-1}(a(\nu^k(j))-a(\nu^k(i)))+s(\nu^\ell(i))-s(\nu^\ell(j))\\
&=m(i,j)+\left(\left\lfloor\frac{(j+\ell)r}{g}\right\rfloor-\left\lfloor\frac{jr}{g}\right\rfloor-\left\lfloor\frac{(i+\ell)r}{g}\right\rfloor+\left\lfloor\frac{ir}{g}\right\rfloor\right)\\
&\hspace{5mm}+\left(\left\lfloor\frac{(i+\ell)r}{g}\right\rfloor-\left\lfloor\frac{(i+\ell)r-\overline{s}(x)}{g}\right\rfloor\right)-\left(\left\lfloor\frac{(j+\ell)r}{g}\right\rfloor-\left\lfloor\frac{(j+\ell)r-\overline{s}(y)}{g}\right \rfloor\right)\\
&=m(i,j)-\left\lfloor\frac{jr}{g}\right\rfloor+\left\lfloor\frac{ir}{g}\right\rfloor+\left\lfloor\frac{(j+\ell)r-\overline{s}(y)}{g}\right \rfloor-\left\lfloor\frac{(i+\ell)r-\overline{s}(x)}{g}\right\rfloor
\end{align*}
Thus Lemma \ref{floor sequence}(1) shows the claim.

It remains to show that $sm:I^2\to\ZZ$ is non-negative.
For each $(i,j)\in I_x\times I_y$ with $x,y\in J$, we have
\begin{align*}
\sum_{k=0}^{g-1}(sm)(\nu^k(i,j))&=
\sum_{k=0}^{g-1}m(\nu^k(i,j))+\sum_{k=0}^{g-1}s(\nu^k(i))-\sum_{k=0}^{g-1}s(\nu^k(j))\\
&\stackrel{\text{Lem.\,\ref{floor sequence}(2)}}{=}m'(i,j)+\overline{s}(x)-\overline{s}(y)\stackrel{\eqref{m+s-s}}{\ge}0.
\end{align*}
Thus the almost constancy of $((sm)(\nu^k(i,j)))_{k\in\ZZ}$ implies $(sm)(i,j)\ge0$.
\end{proof}

\begin{ex}\label{example of m'}
We give an explicit example showing $m'$ and $\overline{m}$ in the proof.
Let $I=(\ZZ/4\ZZ)\sqcup(\ZZ/6\ZZ)$ and $a_\mathrm{av}=1/2$. Take a bijection $\{1,2,3,4,5,6,7,8,9,10\}\simeq I$ given by $\{1,2,3,4\}\simeq\ZZ/4\ZZ$; $i\mapsto i-1$ and $\{5,6,7,8,9,10\}\simeq\ZZ/6\ZZ$; $i\mapsto i-5$, and describe $a$ and $m$ as elements in $\ZZ^{10}$ and $\rM_{10}(\ZZ)$. Since $(m,a,\nu)$ is of floor type, we have
{\small
\begin{align*}
&a=(0,1,0,1,0,1,0,1,0,1)\ \mbox{ and }\ m= \left[\begin{array}{cccc|cccccc}
b&c&d&e&f&g&f&g&f&g\\
e+1&b&c+1&d&g+1&f&g+1&f&g+1&f\\
d&e&b&c&f&g&f&g&f&g\\
c+1&d&e+1&b&g+1&f&g+1&f&g+1&f\\ \hline
h&i&h&i&j&k&l&m&n&p\\
i+1&h&i+1&h&p+1&j&k+1&l&m+1&n\\
h&i&h&i&n&p&j&k&l&m\\
i+1&h&i+1&h&m+1&n&p+1&j&k+1&l\\
h&i&h&i&l&m&n&p&j&k\\
i+1&h&i+1&h&k+1&l&m+1&n&p+1&j
\end{array}\right].
\end{align*}
}
Taking $m'$, the matrix becomes stable under $(i,j)\to(\nu(i),\nu(j))$:
{\small 
\[m'=\left[\begin{array}{cccc|cccccc}
2b&2c+1&2d&2e+1&2f&2g+1&2f&2g+1&2f&2g+1\\
2e+1&2b&2c+1&2d&2g+1&2f&2g+1&2f&2g+1&2f\\
2d&2e+1&2b&2c+1&2f&2g+1&2f&2g+1&2f&2g+1\\
2c+1&2d&2e+1&2b&2g+1&2f&2g+1&2f&2g+1&2f\\ \hline
2h&2i+1&2h&2i+1&2j&2k+1&2l&2m+1&2n&2p+1\\
2i+1&2h&2i+1&2h&2p+1&2j&2k+1&2l&2m+1&2n\\
2h&2i+1&2h&2i+1&2n&2p+1&2j&2k+1&2l&2m+1\\
2i+1&2h&2i+1&2h&2m+1&2n&2p+1&2j&2k+1&2l\\
2h&2i+1&2h&2i+1&2l&2m+1&2n&2p+1&2j&2k+1\\
2i+1&2h&2i+1&2h&2k+1&2l&2m+1&2n&2p+1&2j
\end{array}\right]\]
}
Taking the minimum of each block, we obtain the matrix $\overline{m}$:
\[\overline{m}=\left[\begin{array}{cc}
\min\{2b,2c+1,2d,2e+1\}&\min\{2f,2g+1\}\\
\min\{2h,2i+1\}&\min\{2j,2k+1,2l,2m+1,2n,2p+1\}
\end{array}\right].\]
\end{ex}

\end{document}